\documentclass[english]{article}

\usepackage[T1]{fontenc}
\usepackage[latin9]{inputenc}
\usepackage{float}
\usepackage{amsmath}
\usepackage{graphicx}
\usepackage{amssymb}
\usepackage{framed}
\usepackage{subfigure}
\usepackage{babel}
\usepackage[a4paper, left=1cm, right=4cm]{geometry}
\usepackage{enumerate}
\usepackage{longtable}
\usepackage{mdwlist}
\usepackage{paralist}

\usepackage[]{hyperref} 

\newtheorem{theorem}{Theorem}[section]
\newtheorem{lemma}[theorem]{Lemma}
\newtheorem{proposition}[theorem]{Proposition}
\newtheorem{definition}[theorem]{Definition}
\newtheorem{corollary}[theorem]{Corollary}

\numberwithin{equation}{section}

\newcommand{\qed}{\hspace*{\fill} $\blacksquare$\medskip}

\newenvironment{proof}
{\noindent {\em Proof}.\,\,}
{\qed}

\def \Z {\mathbb Z}

\def \N {\mathbb N}

\def \cX {\mathcal X}
\def \cY {\mathcal Y}
\def \cT {\mathcal T}
\def \cR {\mathcal R}

\def \cV {\mathcal V}
\def \cI {\mathcal I}
\def \cB {\mathcal B}
\def \cF {\mathcal F}

\def \epsi {\varepsilon}
\def \prm 	 {^{\prime}}

\def \starred {^{\star}}

\def \bottom {\cF}
\def \boundary {\partial}

\def \supp	 {\text{supp}}
\def \subconf {\prec}

\def \Lambdaminus {\Lambda^{-}}

\def \metaset	{\cX_\mathrm{meta}}
\def \groundset 	{\cX_\mathrm{stab}}
\newcommand 	{\stablev}[1] {V_{#1}}
\def \comlev 	{\varPhi}
\newcommand 	{\lowset}[1] {\cI_{#1}}

\def \molset	 {\cV_{\star,\nb}^{4\nb}}
\def \nbset		{\cV_{\star,\nb}}
\newcommand {\nbis}[1]	 {\cV_{\star, #1}}

\def \refpath	{\omega\starred}

\def \setb {\varpi_{2}} 

\def \ta	{1}
\def \tb	{2}
\def \na	{n_{\ta}}
\def \nb	{n_{\tb}}
\def \freeb {\mathrm{fp}[2]}

\def \abbar	{\ta\tb\text{--bar}}
\def \abbars	{\ta\tb\text{--bars}}
\def \btile		{\tb\text{--tile}}
\def \btiles		{\tb\text{--tiles}}
\def \btiled		{\tb\text{--tiled}}
\def \btiling	{\tb\text{--tiling}}

\def \D	 {\Delta}
\def \Da	 {\Delta_{1}}
\def \Db	 {\Delta_{2}}

\def \tile {\text{t}}

\def \fromnorth	{\rotatebox[origin = c, units = 4]{1}{$\dashv$}}

\def \sfromnorth {\scalebox{0.6}{\fromnorth}}

\newcommand {\nmexpansion}[1] {\cR^{\sfromnorth}  \left(#1 \right)}

\newcommand {\nrec}[1] 	{{\rotatebox[origin = c, units = 4]{0}{$\sqcap$}}  \left(#1 \right)}

\def \notesmessage{}

\begin{document}

\author{
\renewcommand{\thefootnote}{\arabic{footnote}}
F.\ den Hollander \footnotemark[1] \, \footnotemark[2]
\\
\renewcommand{\thefootnote}{\arabic{footnote}}
F.R.\ Nardi \footnotemark[3] \, \footnotemark[2]
\\
\renewcommand{\thefootnote}{\arabic{footnote}}
A.\ Troiani \footnotemark[1]
}

\title{Kawasaki dynamics with two types of particles:\\
stable/metastable configurations and communication heights}

\footnotetext[1]{
Mathematical Institute, Leiden University, P.O.\ Box 9512,
2300 RA Leiden, The Netherlands
}
\footnotetext[2]{
EURANDOM, P.O.\ Box 513, 5600 MB Eindhoven, The Netherlands
}
\footnotetext[3]{
Technische Universiteit Eindhoven, P.O.\ Box 513, 5600 MB Eindhoven, 
The Netherlands
}

\maketitle

\marginpar{\footnotesize{\notesmessage}}


\begin{abstract}
This is the second in a series of three papers in which we study a 
two-dimensional lattice gas consisting of two types of particles
subject to Kawasaki dynamics at low temperature in a large finite 
box with an open boundary. Each pair of particles occupying neighboring 
sites has a negative binding energy provided their types are different, 
while each particle has a positive activation energy that depends on 
its type. There is no binding energy between particles of the same 
type. At the boundary of the box particles are created and annihilated 
in a way that represents the presence of an infinite gas reservoir. We 
start the dynamics from the empty box and are interested in the 
transition time to the full box. This transition is triggered by 
a critical droplet appearing somewhere in the box.

In the first paper we identified the parameter range for which the
system is metastable, showed that the first entrance distribution on 
the set of critical droplets is uniform, computed the expected 
transition time up to and including a multiplicative factor of 
order one, and proved that the nucleation time divided by its expectation 
is exponentially distributed, all in the limit of low temperature. 
These results were proved under \emph{three hypotheses}, and involve 
\emph{three model-dependent quantities}: the energy, the shape and 
the number of critical droplets. In the second paper we prove the 
first and the second hypothesis and identify the energy of critical 
droplets. In the third paper we settle the rest.

Both the second and the third paper deal with understanding the \emph{geometric 
properties} of subcritical, critical and supercritical droplets, which are crucial in 
determining the metastable behavior of the system, as explained in the first 
paper. The geometry turns out to be considerably more complex than for 
Kawasaki dynamics with one type of particle, for which an extensive literature 
exists. The main motivation behind our work is to understand metastability of 
multi-type particle systems. 

\end{abstract}


\pagebreak


\section{Introduction}
\label{sec introduction}

Section~\ref{sec model and dynamics description} defines the model,
Section~\ref{sec basic not} introduces basic notation, 
Section~\ref{sec key theorems} states the main theorems, while
Section~\ref{sec discussion} discusses the main theorems and provides
further perspectives.


\subsection{Lattice gas subject to Kawasaki dynamics}
\label{sec model and dynamics description}

Let $\Lambda \subset \Z^2$ be a large box centered at the origin (later it 
will be convenient to choose $\Lambda$ rhombus-shaped). Let
\begin{equation}
\label{intextLam}
\begin{aligned}
\partial^-\Lambda &= \{x\in\Lambda\colon\,\exists\,y\notin\Lambda\colon\,|y-x|=1\},\\
\partial^+\Lambda &= \{x\notin\Lambda\colon\,\exists\,y\in\Lambda\colon\,|y-x|=1\},\\ 
\end{aligned}
\end{equation}
be the internal, respectively, external boundary of $\Lambda$, and put $\Lambda^-
= \Lambda\backslash\partial^-\Lambda$ and $\Lambda^+=\Lambda\cup\partial^+\Lambda$. 
With each site $x\in\Lambda$ we associate a variable $\eta(x) \in \{0,1,2\}$ indicating 
the absence of a particle or the presence of a particle of type $\ta$ or type $\tb$. 
A configuration $\eta=\{\eta(x)\colon\,x\in\Lambda\}$ is an element of 
$\cX=\{0,1,2\}^\Lambda$. To each configuration $\eta$ we associate an energy 
given by the Hamiltonian
\begin{equation}
\label{Ham1}
H = -U \sum_{(x,y)\in\Lambda^{*,-}} 
1_{\{\eta(x)\eta(y) = 2\}}\\
+ \Da \sum_{x\in\Lambda} 1_{\{\eta(x)=1\}} 
+ \Db \sum_{x\in\Lambda} 1_{\{\eta(x)=2\}},
\end{equation}
where $\Lambda^{*,-}=\{(x,y)\colon\,x,y\in\Lambda^-,\,|x-y|=1\}$ is the set of non-oriented
bonds inside $\Lambda^-$, $-U<0$ is the \emph{binding energy} between neighboring particles 
of \emph{different} types inside $\Lambda^-$, and $\Da>0$ and $\Db>0$ are the \emph{activation 
energies} of particles of type $\ta$, respectively, $\tb$ inside $\Lambda$. W.l.o.g.\ 
we will assume that 
\begin{equation}
\Da \leq \Db.
\end{equation} 
The Gibbs measure associated with $H$ is 
\begin{equation}
\label{Gibbsmeasure}
\mu_\beta(\eta) = \frac{1}{Z_\beta}\,e^{-\beta H(\eta)}, \qquad \eta\in\cX,
\end{equation}
where $\beta\in (0,\infty)$ is the inverse temperature and $Z_\beta$ is the normalizing 
partition sum.

Kawasaki dynamics is the continuous-time Markov process, $(\eta_t)_{t \geq 0}$ with 
state space $\cX$ whose transition rates are
\begin{equation}
\label{rate}
c_\beta(\eta,\eta\prm) = e^{-\beta [H(\eta\prm)-H(\eta)]_+},
\qquad \eta,\eta\prm\in\cX,\,\eta\neq\eta\prm,\,\eta\leftrightarrow\eta\prm, 
\end{equation}
where $\eta\leftrightarrow\eta\prm$ means that $\eta\prm$ can be obtained from $\eta$ 
by one of the following moves:
\begin{itemize}
\item[$\bullet$]
interchanging $0$ and $1$ or $0$ and $2$ between two neighboring sites in $\Lambda$\\
(``hopping of particles in $\Lambda$''),
\item[$\bullet$]
changing $0$ to $1$ or $0$ to $2$ in $\partial^-\Lambda$\\
(``creation of particles in $\partial^-\Lambda$''),
\item[$\bullet$]
changing $1$ to $0$ or $2$ to $0$ in $\partial^-\Lambda$\\
(``annihilation of particles in $\partial^-\Lambda$''),
\end{itemize}
and $c_\beta(\eta,\eta\prm) = 0$ otherwise. Note that this dynamics preserves particles 
in $\Lambda$, but allows particles to be created and annihilated in $\partial^-\Lambda$. 
Think of the latter as describing particles entering and exiting $\Lambda$ along 
non-oriented bonds between $\partial^+\Lambda$ and $\partial^-\Lambda$ (the rates of 
these moves are associated with the bonds rather than with the sites). The pairs 
$(\eta,\eta\prm)$ with $\eta\leftrightarrow\eta\prm$ are called \emph{communicating
configurations}, the transitions between them are called \emph{allowed moves}. Note 
that particles in $\partial^-\Lambda$ do not interact: the interaction only 
works in $\Lambda^-$. 

The dynamics defined by (\ref{Ham1}) and (\ref{rate}) models the behavior inside 
$\Lambda$ of a lattice gas in $\Z^2$, consisting of two types of particles subject 
to random hopping with hard-core repulsion and with binding between different 
neighboring types. We may think of $\Z^2\backslash\Lambda$ as an \emph{infinite 
reservoir} that keeps the particle densities fixed at $\rho_\ta=e^{-\beta\Da}$ 
and $\rho_\tb=e^{-\beta\Db}$. In the above model this reservoir is replaced by an 
\emph{open boundary} $\partial^-\Lambda$, where particles are created and annihilated 
at a rate that matches these densities. Thus, the dynamics is a \emph{finite-state} 
Markov process, ergodic and reversible with respect to the Gibbs measure $\mu_\beta$   
in \eqref{Gibbsmeasure}.

Note that there is \emph{no} binding energy between neighboring particles of the 
\emph{same} type. Consequently, the model does \emph{not} reduce to Kawasaki dynamics 
for one type of particle when $\Da=\Db$.


\subsection{Notation}
\label{sec basic not}

To state our main theorems in Section~\ref{sec key theorems}, we need some notation.

\begin{definition}
\label{def1}
(a) $\Box$ is the configuration where $\Lambda$ is empty.\\
(b) $\boxplus$ is the set consisting of the two configurations where $\Lambda$ 
is filled with the largest possible checkerboard droplet such that all particles
of type $\tb$ are surrounded by particles of type $\ta$.\\
(c) $\omega \colon\,\eta\to\eta\prm$ is any path of allowed moves from $\eta \in \cX$ 
to $\eta\prm \in \cX$.\\
(d) $\comlev(\eta,\eta\prm)$ is the communication height between $\eta,\eta\prm
\in\cX$ defined by
\begin{equation}
\comlev(\eta,\eta\prm) = \min_{\omega \colon\,\eta\to\eta\prm}
\max_{\xi\in\omega} H(\xi),
\end{equation}
and $\comlev(A,B)$ is its extension to non-empty sets $A,B\subset\cX$ defined by
\begin{equation}
\comlev(A,B) = \min_{\eta\in A,\eta\prm\in B} \comlev(\eta,\eta\prm).
\end{equation}
(e) $\stablev{\eta}$ is the stability level of $\eta\in\cX$ defined by
\begin{equation}
\stablev{\eta} = \comlev({\eta},\lowset{\eta}) - H(\eta),
\end{equation}
where $\lowset{\eta}=\{\xi\in\cX\colon\,H(\xi)<H(\eta)\}$ is the set of
configurations with energy lower than $\eta$.\\
(f) $\groundset = \{\eta\in\cX\colon\,H(\eta)=\min_{\xi\in\cX} H(\xi)\}$ is the 
set of stable configurations, i.e., the set of configurations with mininal energy.\\
(g) $\metaset = \{\eta\in\cX\colon\,\stablev{\eta}=\max_{\xi\in\cX\backslash\groundset} 
\stablev{\xi}\}$ is the set of metastable configurations, i.e., the set of non-stable 
configurations with maximal stability level.\\
(h) $\Gamma=\stablev{\eta}$ for $\eta\in\metaset$ (note that $\eta \mapsto \stablev{\eta}$ 
is constant on $\metaset$), $\Gamma\starred=\comlev(\Box,\boxplus) - H(\Box)$
(note that $H(\Box) = 0$). 
\end{definition}
 
In \cite{dHNT11} we were interested in the transition of the Kawasaki dynamics from 
$\Box$ to $\boxplus$ in the limit as $\beta\to\infty$. This transition, which is
viewed as a crossover from a ``gas phase'' to a ``liquid phase'', is triggered 
by the appearance of a \emph{critical droplet} somewhere in $\Lambda$. The critical 
droplets form a subset of the set of configurations realizing the energetic minimax 
of the paths of the Kawasaki dynamics from $\Box$ to $\boxplus$, which all have 
energy $\Gamma\starred$ because $H(\Box)=0$. 

In \cite{dHNT11} we showed that the first entrance distribution on the set of critical 
droplets is uniform, computed the expected transition time up to and including a 
multiplicative factor of order one, and proved that the nucleation time divided by 
its expectation is exponentially distributed, all in the limit as $\beta\to\infty$. 
These results, which are typical for metastable behavior, were proved under 
\emph{three hypotheses}: 
\begin{itemize}
\item[(H1)] 
$\groundset=\boxplus$.
\item[(H2)] 
There exists a $V\starred<\Gamma\starred$ such that $V_\eta\leq V\starred$ for 
all $\eta\in\cX\backslash\{\Box,\boxplus\}$.
\item[(H3)] 
A hypothesis about the shape of the configurations in the essential gate for
the transition from $\Box$ to $\boxplus$ (for details see \cite{dHNT11}).
\end{itemize} 
Hypotheses (H1--H3) are the \emph{geometric input} that is needed to derive the 
main theorems in \cite{dHNT11} with the help of the \emph{potential-theoretic
approach} to metastability as outlined in Bovier~\cite{B09}. In the present paper 
we prove (H1--H2) and identify the energy $\Gamma\starred$ of critical droplets. In \cite{dHNTpr} 
we settle the rest.

\begin{lemma}
{\rm (H1--H2)} imply that $\stablev{\Box} = \Gamma\starred$, and hence that 
$\metaset = \Box$ and $\Gamma = \Gamma\starred$.
\end{lemma}
\begin{proof}
By Definition~\ref{def1}(e--h) and (H1), $\boxplus\in\lowset{\Box}$, which implies that 
$\stablev{\Box}\leq\Gamma\starred$. We show that (H2) implies $\stablev{\Box}
=\Gamma\starred$. The proof is by contradiction. Suppose that $\stablev{\Box}
<\Gamma\starred$. Then, by Definition~\ref{def1}(h), there exists a  
$\eta_{0}\in\lowset{\Box}\backslash\boxplus$ such that $\comlev(\Box,\eta_{0})
-H(\Box)<\Gamma\starred$. But (H2), together with the finiteness of $\cX$, implies
that there exist an $m \in \N$ and a sequence $\eta_{1},\ldots,\eta_{m}
\in\cX$ with $\eta_{m}=\boxplus$ such that $\eta_{i+1}\in\lowset{\eta_{i}}$ and 
$\comlev(\eta_{i},\eta_{i+1})\leq H(\eta_{i})+V\starred$ for $i=0,\ldots,m-1$. 
Therefore
\begin{equation}
\comlev(\eta_{0},\boxplus)
\leq \max_{i=0,\ldots,m-1} \comlev(\eta_{i}, \eta_{i+1})
\leq \max_{i=0,\ldots,m-1} [H(\eta_{i}) + V\starred]
= H(\eta_{0}) + V\starred 
< H(\Box) + \Gamma\starred,
\end{equation}
where in the first inequality we use that $\comlev(\eta,\sigma)\leq\max\{\comlev(\eta,\xi),
\comlev(\xi,\sigma)\}$ for all $\eta,\sigma,\xi\in\cX$, and in the last inequality that 
$\eta_{0}\in\lowset{\Box}$ and $V\starred<\Gamma\starred$. It follows that
\begin{equation}
\comlev(\Box, \boxplus)-H(\Box) 
\leq \max\{\comlev(\Box,\eta_{0})-H(\Box),\comlev(\eta_{0},\boxplus)-H(\Box)\}
< \Gamma\starred,
\end{equation}
which contradicts Definition~\ref{def1}(h).
Observe that the proof uses that $\metaset$ consists of a single configuration.
\end{proof}

Hypotheses (H1--H2) imply that $(\metaset,\groundset)=(\Box,\boxplus)$, and that 
the highest energy barrier between any two configurations in $\cX$ is the one 
separating $\Box$ and $\boxplus$, i.e., $(\Box,\boxplus)$ is the unique \emph{metastable
pair}. Hypothesis (H3) is needed only to find the asymptotics of the prefactor of
the expected transition time in the limit as $\Lambda\to\Z^2$. The main theorems 
in \cite{dHNT11} involve \emph{three model-dependent quantities}: the energy, the 
shape and the number of critical droplets.


\subsection{Main theorems}
\label{sec key theorems}

In \cite{dHNT11} it was shown that $\Da + \Db < 4U$ is the \emph{metastable region}, 
i.e., the region of parameters for which $\Box$ is a local minimum but not a global 
minimum of $H$. Moreover, it was argued that within this region the subregion where 
$\Da,\Db<U$ is of no interest because the critical droplet consists of two free 
particles, one of type $\ta$ and one of type $\tb$. Therefore the \emph{proper 
metastable region} is 
\begin{equation}
\label{propmetreg}
0< \Da \leq \Db, \quad \Da + \Db < 4U, \quad \Db \geq U,
\end{equation}
as indicated in Fig.~\ref{fig-propmetreg}.

\begin{figure}[htbp]
\begin{centering}
{\includegraphics[width=0.3\textwidth]{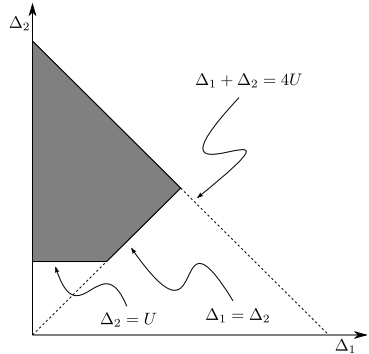}}
\par\end{centering}
\caption{Proper metastable region.}
\label{fig-propmetreg}
\end{figure}

In this present paper, the analysis will be carried out for the subregion where
\begin{equation}
\label{subpropmetreg}
0 < \Da < U, \quad  \Db - \Da > 2U, \quad \Da + \Db < 4U,
\end{equation}
as indicated in Fig.~\ref{fig-subpropmetreg}. \emph{Note}: The second and third 
restriction imply the first restriction. Nevertheless, we write all three because 
each plays an important role in the sequel.

\begin{figure}[htbp]
\begin{centering}
{\includegraphics[width=0.3\textwidth]{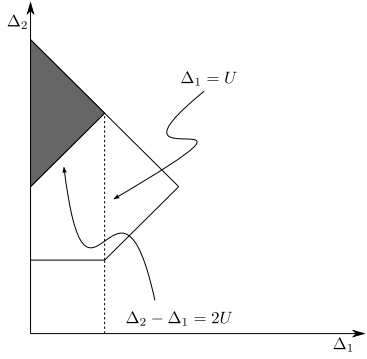}}
\par\end{centering}
\caption{Subregion of the proper metastable region given by (\ref{subpropmetreg}).}
\label{fig-subpropmetreg}
\end{figure}

The following three theorems are the main result of the present paper and are valid
subject to (\ref{subpropmetreg}). We write $\lceil\cdot\rceil$ to denote the upper 
integer part.

\begin{theorem}
\label{theorem ground states}
$\groundset=\boxplus$.
\end{theorem}

\begin{theorem}
\label{theorem recurrence}
There exists a $V\starred\leq 10U-\Da$ such that $\stablev{\eta} \leq V\starred$ 
for all $\eta\in\cX\backslash\{\Box,\boxplus\}$. Consequently, if $\Gamma\starred
>10U-\Da$, then $\metaset = \Box$ and $\Gamma=\Gamma\starred$. 
\end{theorem}

\begin{theorem}
\label{theorem communication height}
$\Gamma\starred = - [\ell\starred(\ell\starred-1)+1](4U-\Da -\Db) 
+ (2\ell\starred+1)\Da +\Db$ with
\begin{equation}
\label{eq value of critical length} 
\ell\starred = \left\lceil \frac{\Da}{4U - \Da - \Db} \right\rceil \in \N.
\end{equation}
\end{theorem}

Theorem~\ref{theorem ground states} settles hypothesis (H1) in \cite{dHNT11}, 
Theorem~\ref{theorem recurrence} settles hypothesis (H2) in \cite{dHNT11} 
when $\Gamma\starred>10U-\Da$, while Theorem~\ref{theorem communication height} 
identifies $\Gamma\starred$. 

As soon as $V\starred<\Gamma\starred$, the energy landscape does not contain 
wells deeper than those surrounding $\Box$ and $\boxplus$. Theorems~\ref{theorem
ground states} and \ref{theorem recurrence} imply that this occurs at least when 
$\Gamma\starred>10U-\Da$, while Theorem~\ref{theorem communication height} identifies 
$\Gamma\starred$ and allows us to exhibit a further subregion of (\ref{subpropmetreg}) 
where the latter inequality is satisfied. This further subregion contains the shaded 
region in Fig.~\ref{fig-relpar}.  

\begin{figure}[htbp]
\begin{centering}
{\includegraphics[width=0.3\textwidth]{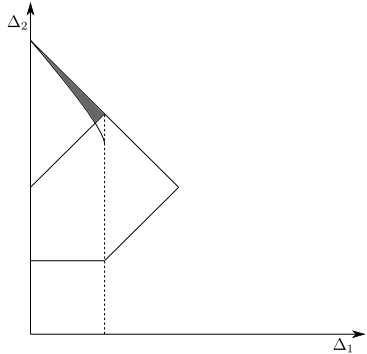}}
\par\end{centering}
\caption{The parameter region where $\Gamma\starred>10U-\Da$ contains the shaded region.}
\label{fig-relpar}
\end{figure}


\subsection{Discussion}
\label{sec discussion}

{\bf 1.}
In Section~\ref{sec identification of gamma} we will see that the \emph{critical 
droplets} for the crossover from $\Box$ to $\boxplus$ consist of a \emph{rhombus-shaped
checkerboard with a protuberance plus a free particle}, as indicated in 
Fig.~\ref{fig-critical_droplet}. A more detailed description will be given in \cite{dHNTpr}. 

\begin{figure}[htbp]
\begin{centering}
{\includegraphics[width=0.3\textwidth]{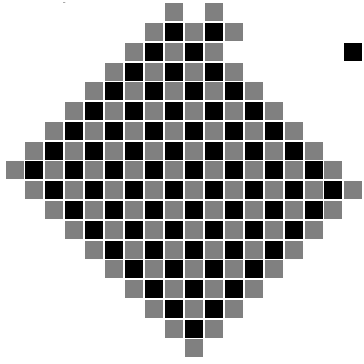}}
\par\end{centering}
\caption{A critical droplet. Light-shaded squares are particles of type $\ta$, 
dark-shaded squares are particles of type $\tb$. The particles of type $\tb$ form
an $\ell\starred\times(\ell\starred-1)$ quasi-square with a protuberance attached 
to one of its longest sides, and are all surrounded by particles of type $\ta$. 
In addition, there is a free particle of type $\tb$. As soon as this free particle
attaches itself ``properly'' to a particle of type $\ta$ the dynamics is ``over the hill''
(see \cite{dHNT11}, Section 2.3, item 3).}
\label{fig-critical_droplet}
\end{figure}

\medskip\noindent
{\bf 2.}
Abbreviate
\begin{equation}
\label{epsdef}
\epsi = 4U - \Da - \Db
\end{equation}
and write $\ell\starred=(\Da/\epsi)+\iota$ with $\iota\in [0,1)$. Then an easy 
computation shows that $\Gamma\starred=(\Da)^2/\epsi+\Da+4U+\epsi\iota(1-\iota)$. 
From this we see that 
\begin{equation}
\ell\starred \sim \Da/\epsi, \qquad \Gamma\starred \sim (\Da)^2/\epsi,
\qquad \epsi \downarrow 0. 
\end{equation}
The limit $\epsi \downarrow 0$ corresponds to the \emph{weakly supersaturated}
regime, where the lattice gas wants to condensate but the energetic threshold 
to do so is high (because the critical droplet is large). From the viewpoint
of metastability this regime is the most interesting. The shaded region in 
Fig.~\ref{fig-relpar} captures this regime for all $0<\Da<U$. This region 
contains the set of parameters where $(\Da)^2/\epsi+\Da+4U>10U-\Da$, i.e.,
$\epsi/U<(\Da/U)^2/[6-2(\Da/U)]$.   

\medskip\noindent
{\bf 3.}
The simplifying features of (\ref{subpropmetreg}) over (\ref{propmetreg}) are 
the following: $\Da<U$ implies that each time a particle of type $\ta$ enters 
$\Lambda$ and attaches itself to a particle of type $\tb$ in a droplet the 
energy goes down, while $\Db-\Da>2U$ implies that no particle of type 
$\tb$ sits on the boundary of a droplet that has minimal energy given the 
number of particles of type $\tb$ in the droplet. In \cite{dHNT11} we 
conjectured that the metastability results presented there actually hold 
throughout the region given by (\ref{propmetreg}), even though the critical 
droplets will be \emph{different} when $\Da \geq U$.   

As will become clear in Section~\ref{sec ground states}, the constraint $\Da<U$ 
has the effect that in all configurations that are local minima of $H$ all 
particles on the boundary of a droplet are of type $\ta$. It will turn out 
that such configurations consist of a single \emph{rhombus-shaped checkerboard 
droplet}. We expect that as $\Da$ increases from $U$ to $2U$ there is a gradual 
transition from a rhombus-shaped checkerboard critical droplet to a square-shaped 
checkerboard critical droplet. This is one of the reasons why it is difficult 
to go beyond (\ref{subpropmetreg}).

\medskip\noindent
{\bf 4.} 
What makes Theorem~\ref{theorem recurrence} hard to prove is that the estimate 
on $V_\eta$ has to be uniform in $\eta\notin\{\Box,\boxplus\}$. In configurations
containing several droplets and/or droplets close to $\partial^-\Lambda$ there may
be a lack of free space making the motion of particles inside $\Lambda$ difficult.
The mechanisms developed in Section~\ref{sec Recurrence} allow us to realize an
\emph{energy reduction} to a configuration that lies on a suitable \emph{reference
path for the nucleation} within an energy barrier $10U-\Da$ also in the absence of 
free space around each droplet. 

We will see in Section~\ref{sec Recurrence} that for droplets sufficiently far 
away from other droplets and from $\partial^-\Lambda$ a reduction within an energy 
barrier $\leq 4U + \Da$ is possible. Thus, if we would be able to control the 
configurations that fail to have this property, then we would have $V\starred\leq 4U + \Da$ 
and, consequently, would have $\metaset=\Box$ and $\Gamma=\Gamma\starred$ throughout 
the subregion given by (\ref{subpropmetreg}) because $\Gamma\starred>4U + \Da$.

Another way of phrasing the last observation is the following. We view the 
``liquid phase'' as the configuration filling the entire box $\Lambda$. If, 
instead, we would let the liquid phase correspond to the set of configurations 
filling most of $\Lambda$ but staying away from $\partial^-\Lambda$, then the 
metastability results derived in \cite{dHNT11} would apply throughout the 
subregion given by (\ref{subpropmetreg}).     

\medskip\noindent
{\bf 5.}
Theorems~\ref{theorem ground states} and \ref{theorem communication height}
can actually be proved \emph{without} the restriction $\Db-\Da>2U$. However, 
removal of this restriction makes the task of showing that in droplets with minimal 
energy all particles of type $\tb$ are surrounded by particles of type $\ta$ 
more involved than what is done in Section~\ref{sec ground states}. We omit 
this extension, since the restriction $\Db-\Da>2U$ is needed for 
Theorem~\ref{theorem recurrence} anyway.

\medskip\noindent
{\bf Outline.}
Section~\ref{sec defnot} contains preparations.
Theorems~\ref{theorem ground states}--\ref{theorem communication height}
are proved in Sections \ref{sec ground states}--\ref{sec Recurrence}, 
respectively. The proofs are \emph{purely combinatorial}, and are rather 
involved due to the presence of two types of particles rather than one.
Sections~\ref{sec ground states}--\ref{sec identification of gamma}
deal with \emph{statics} and Section~\ref{sec Recurrence} with \emph{dynamics}.
Section~\ref{sec Recurrence}
is technically the hardest and takes up about half of the paper. More detailed outlines
are given at the beginning of each section.


\section{Coordinates, definitions and polyominoes}
\label{sec defnot}

Section~\ref{sec coordinates} introduces two coordinate systems that are used 
to describe the particle configurations: standard and dual. Section~\ref{sec def} 
lists the main geometric definitions that are needed in the rest of the paper.
Section~\ref{sec lemma polyominoes} proves a lemma about polyominoes (finite 
unions of unit squares) and Section~\ref{sec bonds in tiled clusters} a lemma
about $\btiled$ clusters (checkerboard configurations where all particles of 
type $\tb$ are surrounded by particles of type $\ta$). These lemmas are needed 
in Section~\ref{sec ground states} to identify the droplets of minimal energy 
given the number of particles of type $\tb$ in $\Lambda$.  


\subsection{Coordinates}
\label{sec coordinates}

\newcounter{notcounter}
\begin{list}{\textbf{\arabic{notcounter}.~}}
{\usecounter{notcounter}
\labelsep=0em \labelwidth=0em \leftmargin=0em \itemindent=0em}

\item
A site $i\in\Lambda$ is identified by its \emph{standard coordinates} 
$(x_{1}(i),x_{2}(i))$, and is called odd when $x_{1}(i)+x_{2}(i)$ is 
odd and even when $x_{1}(i)+x_{2}(i)$ is even. The standard coordinates 
of a particle $p$ in $\Lambda$ are denoted by $x(p) = (x_{1}(p),x_{2}(p))$. 
The \emph{parity} of a particle $p$ is defined as $x_{1}(p)+x_{2}(p)
+\eta(x(p))$ modulo 2, and $p$ is said to be odd when the parity is $1$ 
and even when the parity is $0$.

\item
A site $i\in\Lambda$ is also identified by its \emph{dual coordinates}
\begin{equation}
u_1(i) = \frac{x_1(i) - x_2(i)}{2}, \qquad u_2(i) = \frac{x_1(i) + x_2(i)}{2}.
\end{equation}
Two sites $i$ and $j$ are said to be \emph{adjacent}, written $i \sim j$, 
when $|x_1(i)-x_1(j)|+ |x_2(i)-x_2(j)|=1$ or, equivalently, $|u_1(i)-u_1(j)| 
= |u_2(i)-u_2(j)| = \tfrac12$ (see Fig.~\ref{fig-coord}).

\item
For convenience, we take $\Lambda$ to be the $(L+\tfrac32) \times (L+\tfrac32)$ 
dual square centered at the origin for some $L\in\N$ with $L>2\ell\starred$
(to allow for $H(\boxplus) < H(\Box)$; 
see Section \ref{sec standard configurations are optimal btiled configurations}). 
Particles interact only inside $\Lambdaminus$, which is the 
$(L+\tfrac{1}{2}) \times (L+\tfrac{1}{2})$ dual square centered at the origin. 
This dual square, a \emph{rhombus} in standard coordinates, is convenient because 
the local minima of $H$ are rhombus-shaped as well (see Section~\ref{sec ground states}). 

\end{list}

\begin{figure}[htbp]
\centering
\subfigure[]
{\includegraphics[height=0.24\textwidth]{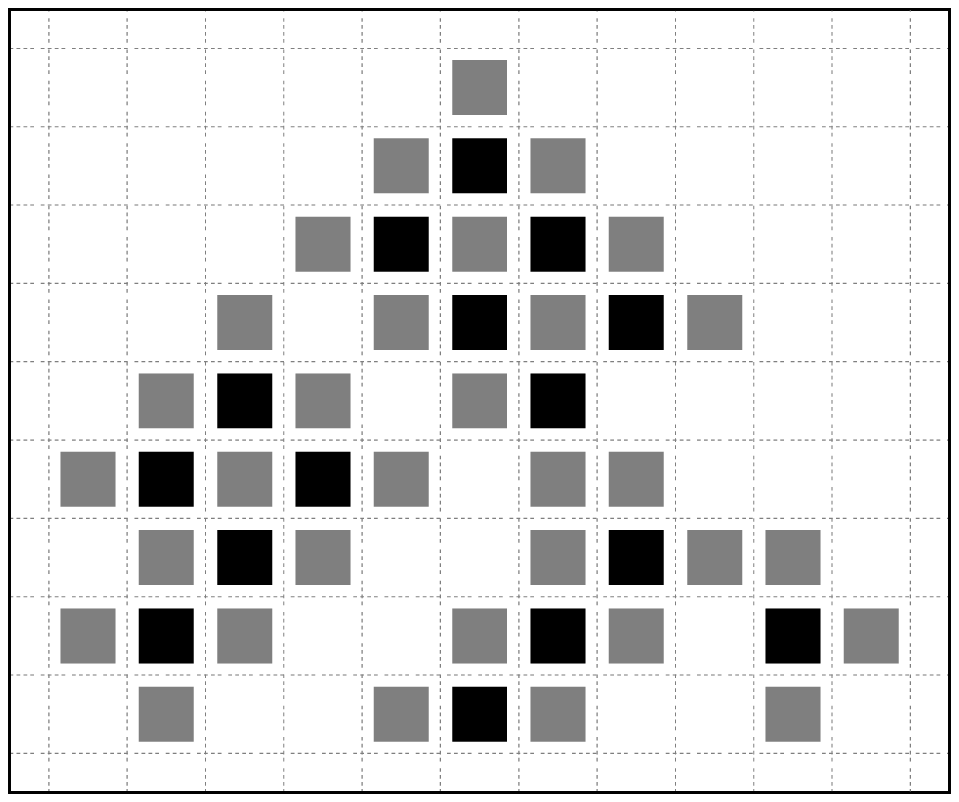}}
\qquad
\subfigure[]
{\includegraphics[height=0.24\textwidth]{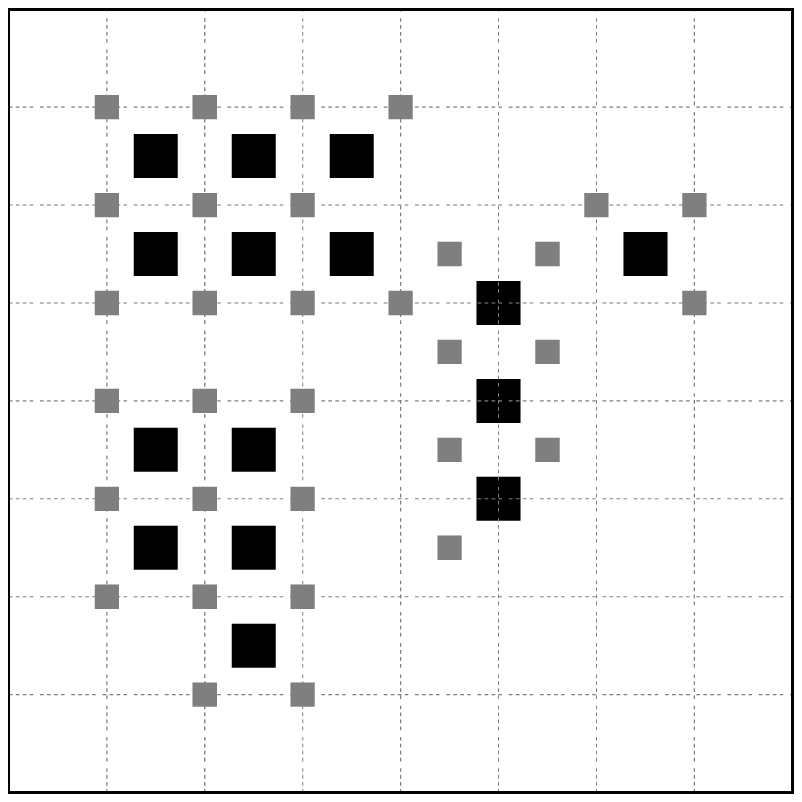}}
\caption{A configuration represented in: (a) standard coordinates; 
(b) dual coordinates. Light-shaded squares are particles of type $\ta$,
dark-shaded squares are particles of type $\tb$. In dual coordinates, 
particles of type $\tb$ are represented by larger squares than particles 
of type $\ta$ to exhibit the ``tiled structure'' of the configuration.}
\label{fig-coord}
\end{figure}


\subsection{Definitions}
\label{sec def}

\newcounter{defcounter}
\begin{list}{\textbf{\arabic{defcounter}.~}}
{\usecounter{defcounter}
\labelsep=0em \labelwidth=0em \leftmargin=0em \itemindent=0em}

\item
A site $i\in\Lambda$ is said to be \emph{lattice-connecting} in the configuration $\eta$ 
if there exists a lattice path $\lambda$ from $i$ to $\partial^-\Lambda$ such that 
$\eta(j) = 0$ for all $j \in \lambda$ with $j \neq i$. We say that a particle $p$ is 
lattice-connecting if $x(p)$ is a lattice-connecting site.
\label{def group lattice}

\item	
Two particles in $\eta$ at sites $i$ and $j$ are called \emph{connected} if $i \sim j$ and 
$\eta(i)\eta(j) = 2$. If two particles $p_{1}$ and $p_{2}$ are connected, then we say 
that there is an \emph{active bond} $b$ between them. The bond $b$ is said to be \emph{incident} 
to $p_{1}$ and $p_{2}$. A particle $p$ is said to be \emph{saturated} if it is connected to 
four other particles, i.e., there are four active bonds incident to $p$. The support of 
the configuration $\eta$, i.e., the union of the unit squares centered at the occupied 
sites of $\eta$, is denoted by $\supp{(\eta)}$. For a configuration $\eta$, $n_\ta(\eta)$ 
and $n_\tb(\eta)$ denote the number of particles of type $\ta$ and $\tb$ in $\eta$, and 
$B(\eta)$ denotes the number of active bonds. The energy of $\eta$ equals $H(\eta)
=\Da n_\ta(\eta)+\Db n_\tb(\eta)-UB(\eta)$.
\label{def group connected particles and bonds}

\item
Let $G(\eta)$ be the \emph{graph} associated with $\eta$, i.e., $G(\eta)=(V(\eta),E(\eta))$, 
where $V(\eta)$ is the set of sites $i\in\Lambda$ such that $\eta(i)\ne 0$, and $E(\eta)$ 
is the set of the pairs $\{i,j\}$, $i,j\in V(\eta)$, such that the particles at sites $i$ 
and $j$ are connected. A configuration $\eta'$ is called a \emph{subconfiguration} of 
$\eta$, written $\eta' \subconf \eta$, if $\eta'(i)=\eta(i)$ for all $i\in\Lambda$ 
such that $\eta'(i)>0$. A subconfiguration $c\subconf\eta$ is a \emph{cluster} if the 
graph $G(c)$ is a maximal connected component of $G(\eta)$. The set of non-saturated 
particles in $c$ is called the \emph{boundary} of $c$, and is denoted by $\boundary c$. 
Clearly, all particles in the same cluster have the same parity. Therefore the concept of 
parity extends from particles to clusters. 
\label{def group clusters and parity}

\item
For a site $i\in\Lambda$, the \emph{tile} centered at $i$, denoted by $\tile(i)$, is the 
set of five sites consisting of $i$ and the four sites adjacent to $i$. If $i$ is an even 
site, then the tile is said to be even, otherwise the tile is said to be odd. The five 
sites of a tile are labeled $a$, $b$, $c$, $d$, $e$ as in Fig.~\ref{fig-2tile}. 
The sites labeled $a$, $b$, $c$, $d$ are called \emph{junction sites}. 
If a particle $p$ sits at site $i$, then $\tile(i)$ is also denoted by $\tile(p)$ and is 
called the tile associated with $p$. In standard coordinates, a tile is a square of size 
$\sqrt{2}$. In dual coordinates, it is a unit square. 
\label{def group tiles}

\item
A tile whose central site is occupied by a particle of type $\tb$ and whose junction 
sites are occupied by particles of type $\ta$ is called a \emph{$\btile$} (see
Fig.~\ref{fig-2tile}). Two $\btiles$ are said to be adjacent if their particles
of type $\tb$ have dual distance 1. A horizontal (vertical) \emph{$\abbar$} is a 
maximal sequence of adjacent $\btiles$ all having the same horizontal (vertical) 
coordinate. If the sequence has length $1$, then the $\abbar$ is called a \emph{$\btiled$ 
protuberance}. A cluster containing at least one particle of type $\tb$ such that 
all particles of type $\tb$ are saturated is said to be $\btiled$. A $\btiled$ 
configuration is a configuration consisting of $\btiled$ clusters only.
\label{def group btiles}

\begin{figure}[htbp]
\centering
\subfigure[]
{\includegraphics[height=0.12\textwidth]{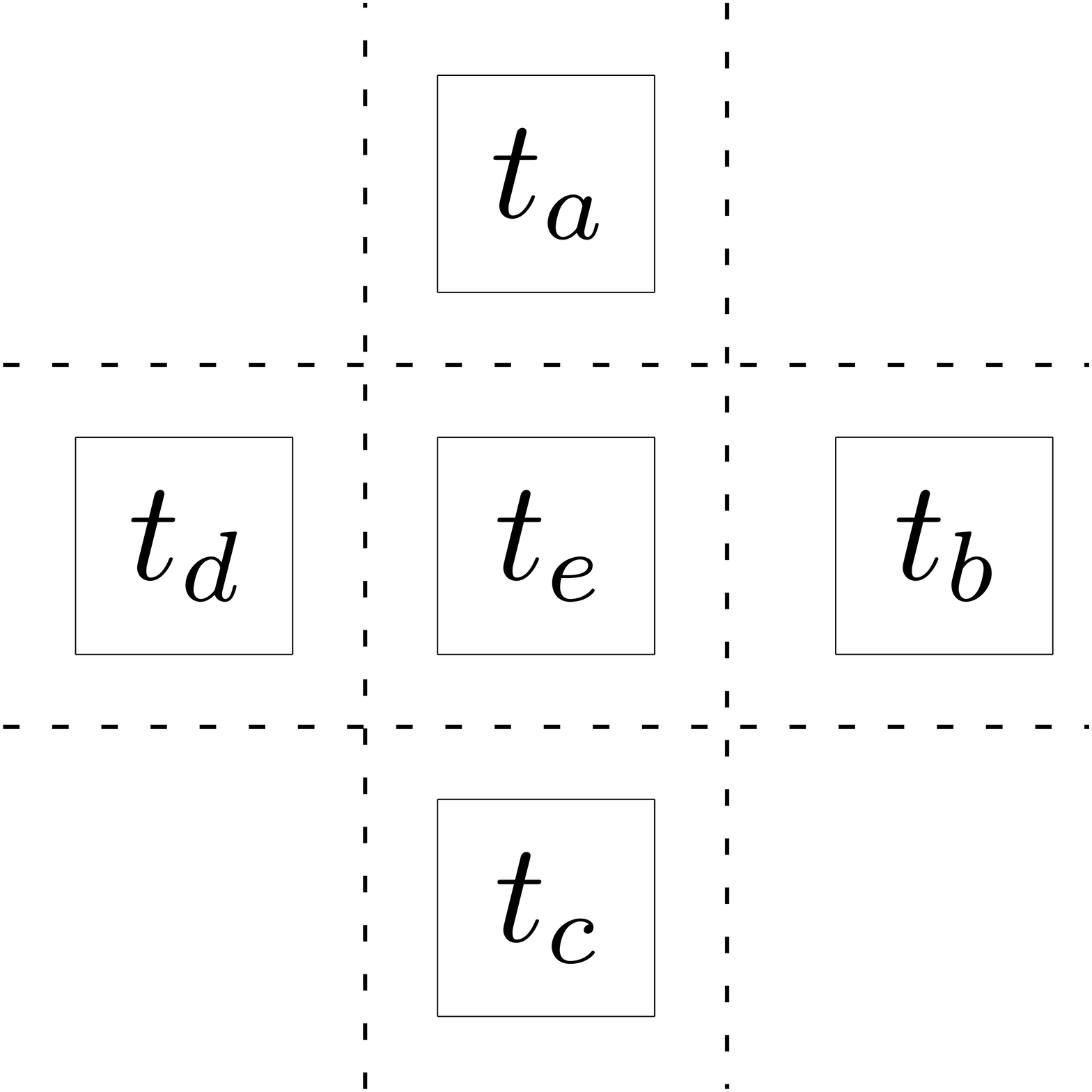}}
\quad
\subfigure[]
{\includegraphics[height=0.12\textwidth]{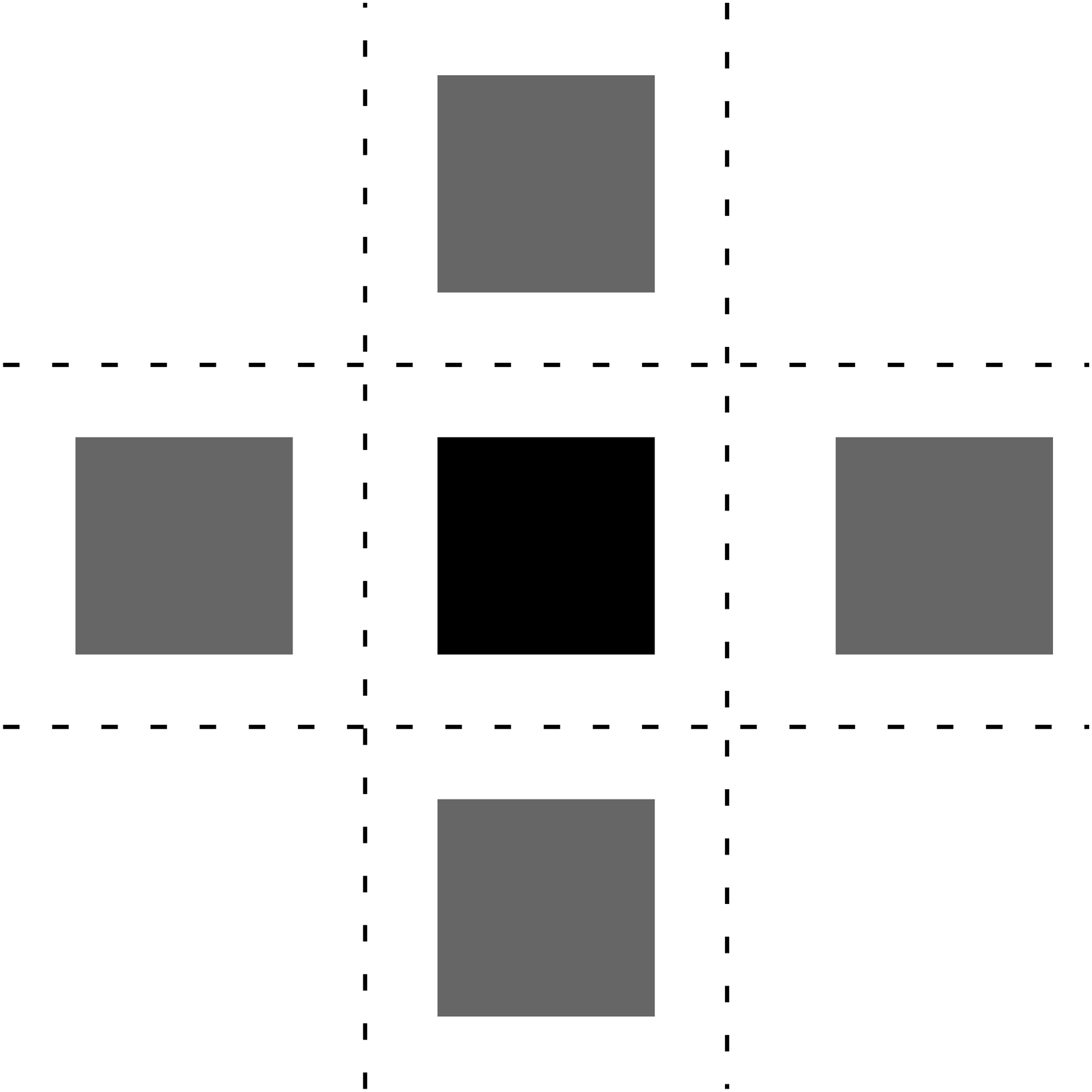}}
\quad
\subfigure[]
{\includegraphics[height=0.12\textwidth]{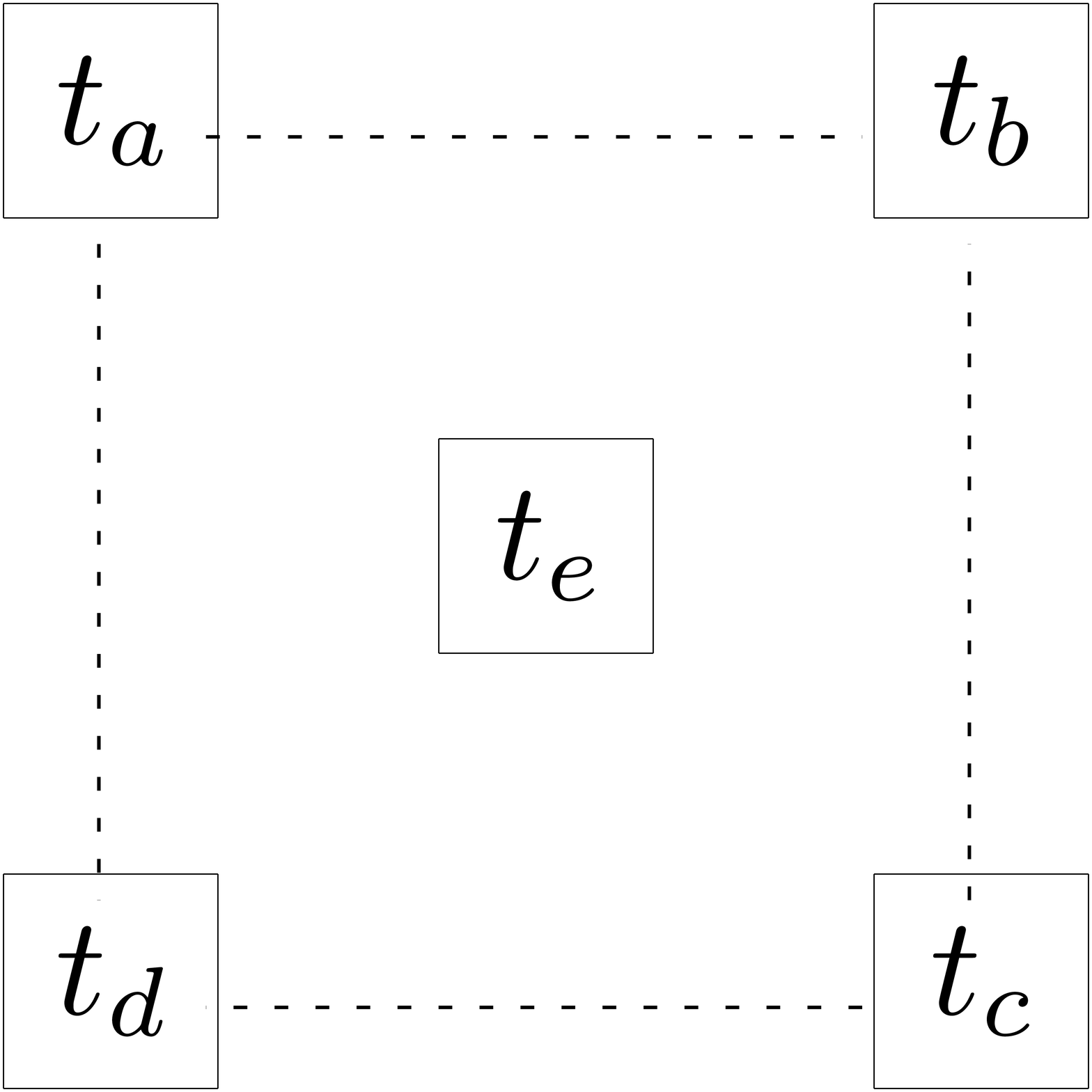}}
\quad
\subfigure[]
{\includegraphics[height=0.12\textwidth]{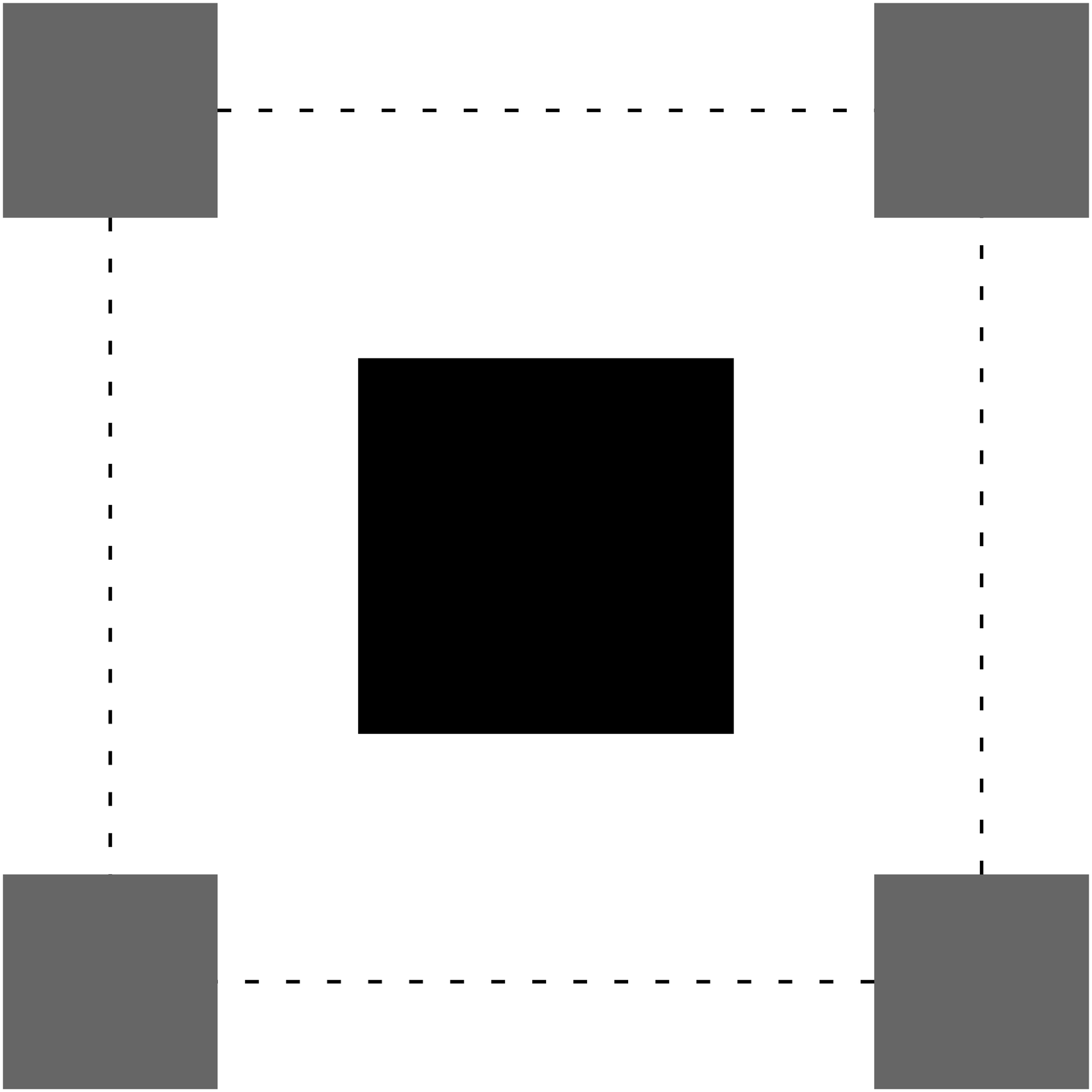}}
\caption{Tiles: (a) standard representation of the labels of a tile; (b) standard 
representation of a $\btile$; (c) dual representation of the labels of a tile; 
(d) dual representation of a $\btile$.}
\label{fig-2tile}
\end{figure}
\label{def group btile}

\item
The \emph{tile support} of a configuration $\eta$ is defined as
\begin{equation}
[\eta] = \bigcup_{p \in \setb(\eta) } \tile(p),
\end{equation}
where $\setb(\eta)$ is the set of particles of type $\tb$ in $\eta$. Obviously,
$[\eta]$ is the union of the tile supports of the clusters making up $\eta$.
\label{def group tile support}
For a standard cluster $c$ the \emph{dual perimeter}, denoted by $P(c)$, is the length of 
the Euclidean boundary of its tile support $[c]$ (which includes an inner boundary
when $c$ contains holes). The dual perimeter $P(\eta)$ of a $\btiled$ configuration 
$\eta$ is the sum of the dual perimeters of the clusters making up $\eta$.
\label{def group dual perimeter}

\item 
$\nbset$ is the set of configurations such that in $\Lambda^{--}$ the number of 
particles of type $\tb$ is $\nb$. $\molset$ is the set of configurations such that in 
$\Lambda^{--}$ the number of particles of type $\tb$ is $\nb$, the number of active 
bonds is $4\nb$, and there is no isolated particle of type $\ta$. In other words,
$\molset$ is the set of $\btiled$ configurations with $\nb$ particles of type $\tb$. 
The lower index $\star$ is used to indicate that configurations in these sets can 
have an arbitrary number of particles of type $\ta$.
\label{def group set}
A configuration $\eta$ is called \emph{standard} if $\eta \in \molset$, and its tile 
support is a standard polyomino in dual coordinates (see Definition~\ref{def standard
polyominoes} below for the definition of a standard polyomino).
\label{def group standard configuration}

\item
A \emph{unit hole} is an empty site such that all four of its neighbors are 
occupied by particles of the same type (either all of type $\ta$ or all of type $\tb$).
\label{def group unit hoyle}
An empty site with three neighboring sites occupied by a particle of type $\ta$ 
is called a \emph{good dual corner}. In the dual representation a good dual corner 
is a concave corner (see Fig.~\ref{fig:convex and concave corners}).
\label{def group good corner}

\end{list}


\subsection{A lemma on polyominoes}
\label{sec lemma polyominoes}

The tile support of a cluster $c$ can be represented by a polyomino, i.e., a finite 
union of unit squares. The following notation is used:
\begin{description} 
\item{$\ell_1(c)=$} 
width of $c$ (= number of columns).
\item{$\ell_2(c)=$} 
height of $c$ (= number of rows).
\item{$v_i(c)=$} 
number of vertical edges in the $i$-th non-empty row of $c$.
\item{$h_j(c)=$}
number of horizontal edges in the $j$-th non-empty column of $c$.
\item{$P(c)=$} 
length of the perimeter of $c$.
\item{$Q(c)=$} 
number of holes in $c$.
\item{$\psi(c)=$} 
number of convex corners of $c$.  
\item{$\phi(c)=$} 
number of concave corners of $c$.
\end{description} 

Note that $\psi(c) = \sum_{i=1}^{N(c)} \psi(i)$ and $\phi(c) = \sum_{i=1}^{N(c)} \phi(i)$, 
where $N(c)$ is the number of vertices in the polyomino representing $c$. If two edges 
$e_{1}$ and $e_{2}$ are incident to vertex $i$ at a right angle with a unit square inside 
and no unit squares outside, then $\psi(i) = 1$ and $\phi(i) = 0$ (Fig.~\ref{fig:convex
and concave corners}(a)). On the other hand, if there is no unit square inside and three 
unit squares outside, then $\psi(i) = 0$ and $\phi(i) = 1$ (Fig.~\ref{fig:convex and concave 
corners}(b)). If four edges $e_{1}$, $e_{2}$, $e_{3}$, $e_{4}$ are incident to vertex $i$, 
with two unit squares in opposite angles, then $\psi(i) = 0$ and $\phi(i) = 2$ 
(Fig.~\ref{fig:convex and concave corners}(c)).

\begin{figure}[htbp]
\centering
\subfigure[]
{\includegraphics[height=0.10\textwidth]{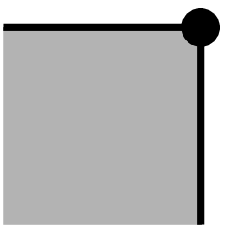}}
\qquad
\subfigure[]
{\includegraphics[height=0.10\textwidth]{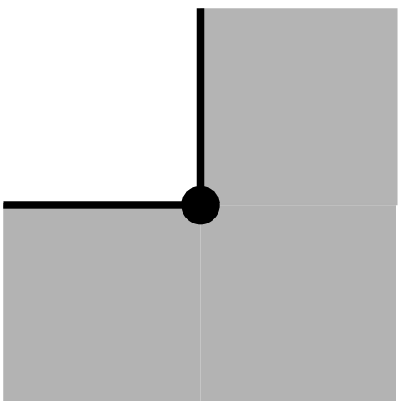}}
\qquad
\subfigure[]
{\includegraphics[height=0.10\textwidth]{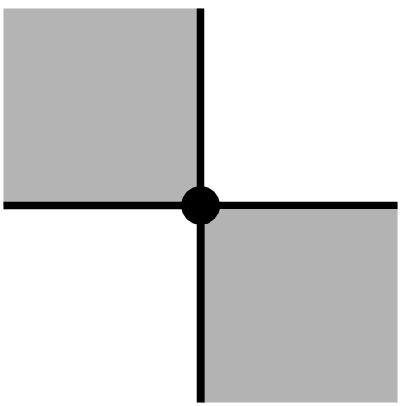}}
\caption{Corners of polyominoes: (a) one convex corner; (b) one concave corner; 
(c) two concave corners. Shaded mean occupied by a unit square.}
\label{fig:convex and concave corners}
\end{figure}

\begin{definition}
\label{def standard polyominoes}
{\rm [Alonso and Cerf~\cite{AC96}.]}
A polyomino is called monotone if its perimeter is equal to the perimeter of its 
circumscribing rectangle. A polyomino whose support is a quasi-square (i.e., a 
rectangle whose side lengths differ by at most one), with possibly a bar attached 
to one of its longest sides, is called a standard polyomino.
\end{definition}

In the sequel, a key role will be played by the quantity 
\begin{equation}
\label{Tdef}
\cT(c) = 2P(c) + [\psi(c) - \phi(c)] = 2P(c) + 4 - 4Q(c).
\end{equation}

\begin{lemma}
\label{lemma alpha optimality}
(i) All polyominoes $c$ with a fixed number of monominoes minimizing $\cT(c)$ are 
single-component monotone polyominoes of minimal perimeter, which include the standard 
polyominoes.\\
(ii) If the number of monominoes is $\ell^2$, $\ell^2-1$, $\ell(\ell-1)$ or $\ell(\ell-1)-1$
for some $\ell\in\N\backslash\{1\}$, then the standard polyominoes are the only minimizers 
of $\cT(c)$. 
\end{lemma}

\begin{proof}
In the proof we assume w.l.o.g.\ that the polyomino consists of a single cluster $c$.

\medskip\noindent
(i) The proof uses projection. Pick any non-monotone cluster $c$. Let 
\begin{equation}
\tilde{c} = (\pi_{2} \circ \pi_{1})(c),
\end{equation}
where $\pi_{2}$ and $\pi_{1}$ denote the vertical, respectively, the horizontal 
projection of $c$. The effect of vertical and horizontal projection is illustrated 
in Fig.~\ref{fig-projections}. By construction, $\tilde{c}$ is a monotone polyomino 
(see e.g.\ the statement on Ferrers diagrams in the proof of Alonso and Cerf \cite{AC96}, 
Theorem 2.2). 

\begin{figure}[htbp]
\centering
{\includegraphics[height=0.16\textwidth]{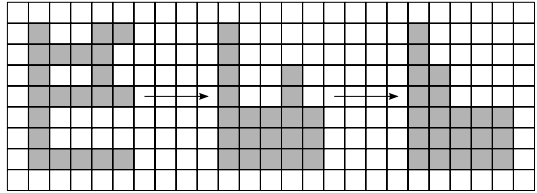}}
\caption{Effect of vertical and horizontal projection.}
\label{fig-projections}
\end{figure}

Suppose first that $Q(c)=0$. Then $\cT(c)=2P(c)+4$. Since $c$ is not monotone, 
we have $P(\tilde{c})<P(c)$, and so $c$ is not a minimizer of $\cT(c)$. 

Suppose next that $Q(c) \geq 1$. Since
\begin{equation}
\label{Pid1}
P(c) = \sum_{i=1}^{\ell_2(c)} v_i(c) + \sum_{j = 1}^{\ell_1(c)} h_j(c)
\end{equation}
and every hole belongs to at least one row and one column, we have
\begin{equation}
\label{pol1}
P(c) \geq 2[\ell_1(c) + \ell_2(c)] + 4Q(c).
\end{equation}
On the other hand, since $\tilde{c}$ is a monotone polyomino, we have $v_i(\tilde{c}) 
= h_j(\tilde{c}) = 2$ for all $i$ and $j$, and so
\begin{equation}
\label{pol2}
P(\tilde{c}) = 2[\ell_1(\tilde{c}) + \ell_2(\tilde{c})].
\end{equation}
Moreover, since $\ell_1(\tilde{c}) \leq \ell_1(c) $ and $\ell_2(\tilde{c}) \leq \ell_2(c)$, 
we can combine (\ref{pol1}--\ref{pol2}) to get
\begin{equation}
\label{eq minimal reduction of perimeter}
P(\tilde{c}) - P(c) \leq -4Q(c), 
\end{equation}
Using (\ref{eq minimal reduction of perimeter}), we obtain
\begin{equation}
\cT(\tilde{c}) - \cT(c)
= [2P(\tilde{c}) + 4] - [2P(c) + 4 - 4Q(c)]
= 2[P(\tilde{c}) - P(c)] + 4Q(c) \leq  -4Q(c) \leq -4 < 0,		
\end{equation}
and so $c$ is not a minimizer of $\cT(c)$.

\medskip\noindent
(ii) We saw in the proof of (i) that if $c$ is a minimizer of $\cT(c)$, then $c$ is monotone,
and hence does not contain holes and minimizes $P(c)$. The claim therefore follows from 
Alonso and Cerf~\cite{AC96}, Corollary 3.7, which states that if the number of monominoes 
is $\ell^2$, $\ell^2-1$, $\ell(\ell-1)$ or $\ell(\ell-1)-1$ for some $\ell\in\N\backslash\{1\}$, 
then the standard polyominoes are the only minimizers of $P(c)$. 
\end{proof}


\subsection{Relation between $\cT$ and the number of missing bonds in $\btiled$ clusters}
\label{sec bonds in tiled clusters}

In this section we consider $\btiled$ clusters and link the number of particles of type $\ta$ 
and type $\tb$ to the number of active bonds and the geometric quantity $\cT$ considered in 
Section~\ref{sec lemma polyominoes}.  

\begin{lemma}
\label{lemma number of bonds in a molecule cluster} 
For any $\btiled$ cluster $c$ (i.e., $c\in\molset$ for some $\nb$), $4\na(c) = B(c)+\cT(c)$ 
and $4\nb(c) = B(c)$.
\end{lemma}

\begin{proof}
The claim of the lemma is equivalent to the affirmation that $\cT(c)=M(c)$ with $M(c)$ the 
number of missing bonds in $c$. Indeed, informally, for every unit perimeter two bonds are 
lost with respect to the four bonds that would be incident to each particle of type $\ta$ 
if it were saturated, while one bond is lost at each convex corner and one bond is gained at 
each concave corner. 

Formally, let $p$ be a particle of type $\ta$, $B(p)$ the number of bonds 
incident to $p$, and $M(p)=4-B(p)$ the number of missing bonds of $p$. Consider 
the set of particles of type $\ta$ at the boundary of a $\btiled$ cluster, i.e., 
the set of non-saturated particles of type $\ta$. Each of these particles belongs 
to one of four classes (see Fig.~\ref{fig-boundary_classes}):

\begin{description}
\item{class $1$}: 
$p$ has two neighboring particles of type $\tb$ belonging to the same $\abbar$.
\item{class $2$}: 
$p$ has two neighboring particles of type $\tb$ belonging to different $\abbars$.
\item{class $3$}: 
$p$ has three neighboring particles of type $\tb$.
\item{class $4$}: 
$p$ has one neighboring particle of type $\tb$.
\end{description}

\begin{figure}[htbp]
\centering
\subfigure[]
	{\includegraphics[height=0.12\textwidth]{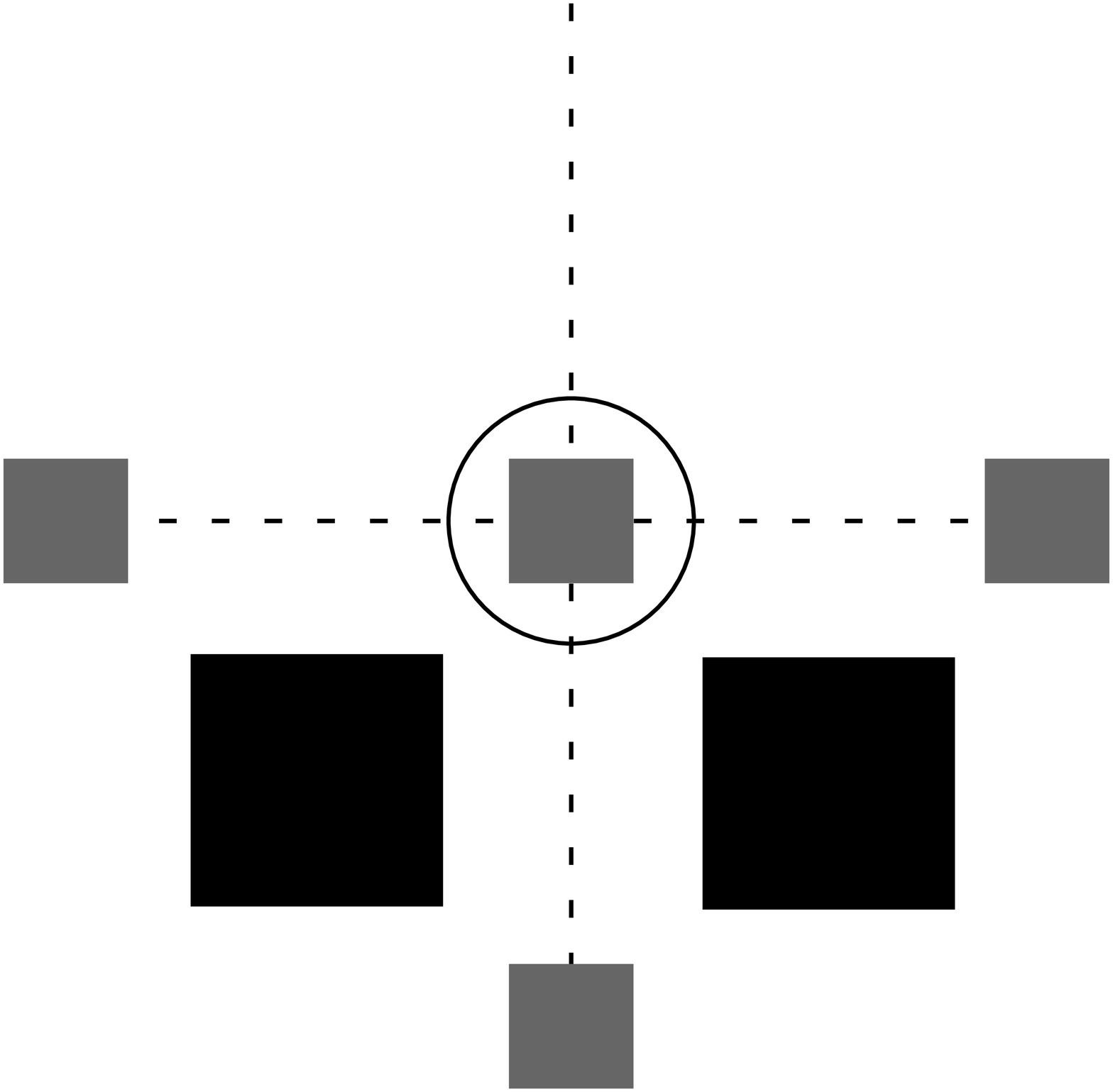}}
\qquad
\subfigure[]
	{\includegraphics[height=0.12\textwidth]{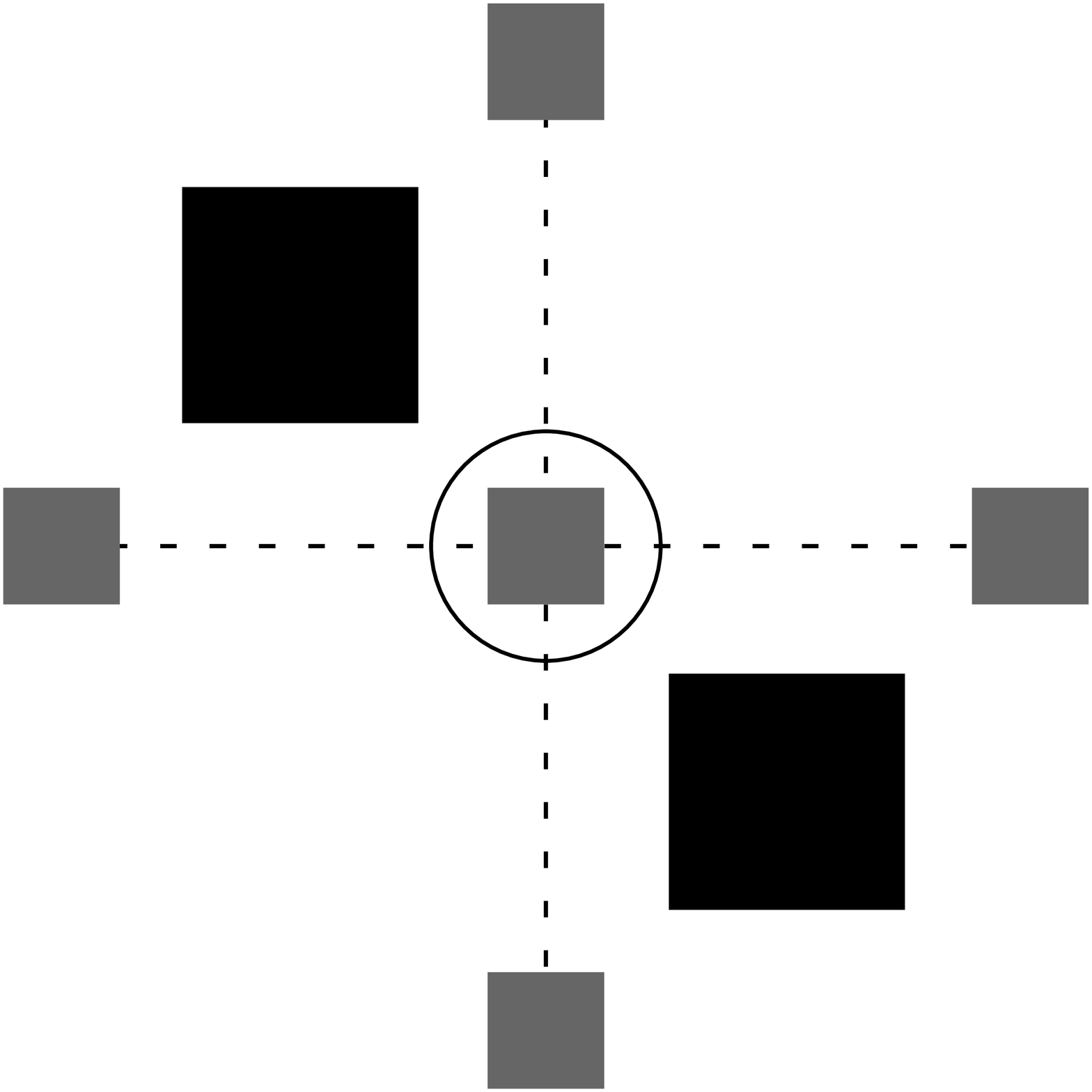}}
\qquad
\subfigure[]
	{\includegraphics[height=0.12\textwidth]{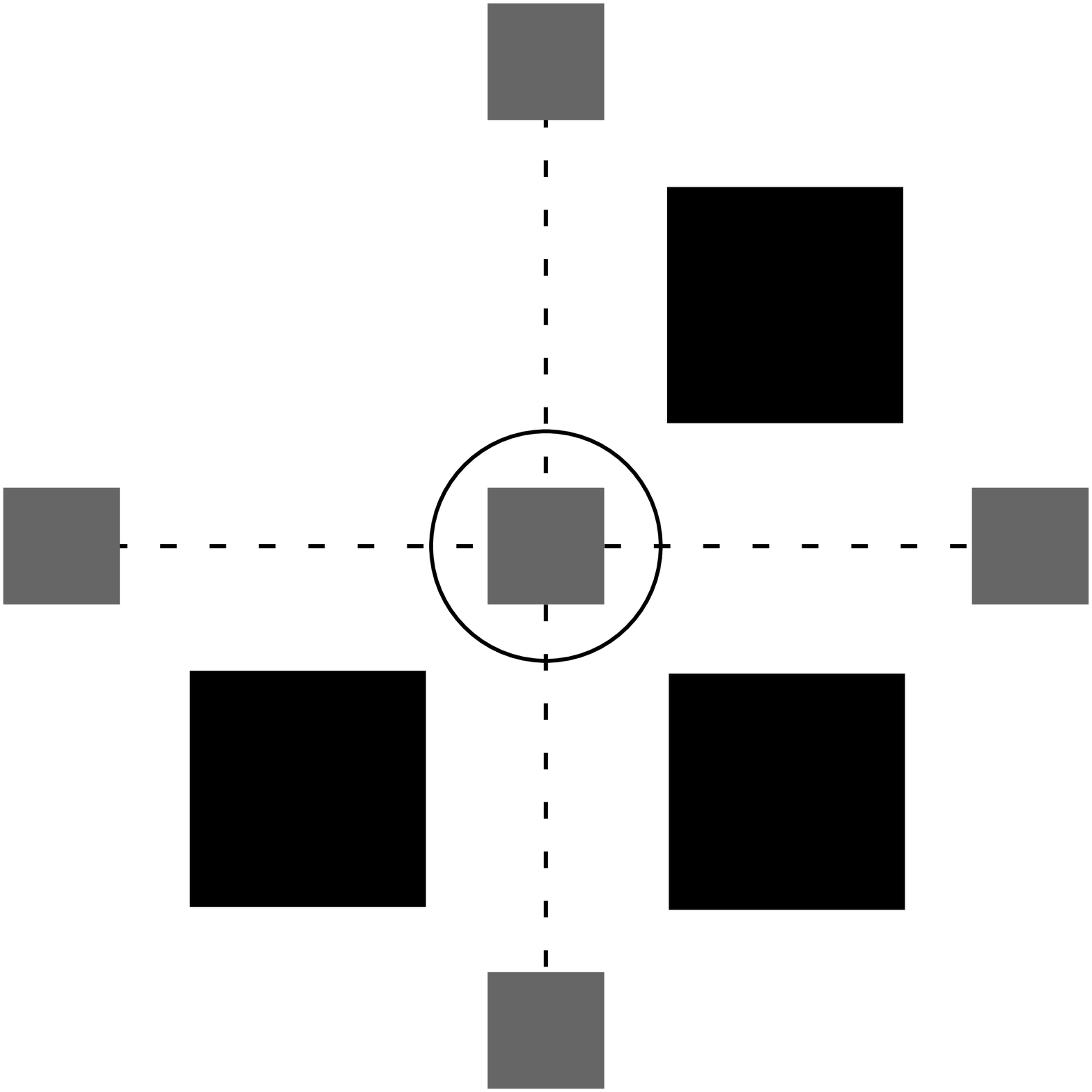}}
\qquad
\subfigure[]
	{\includegraphics[height=0.12\textwidth]{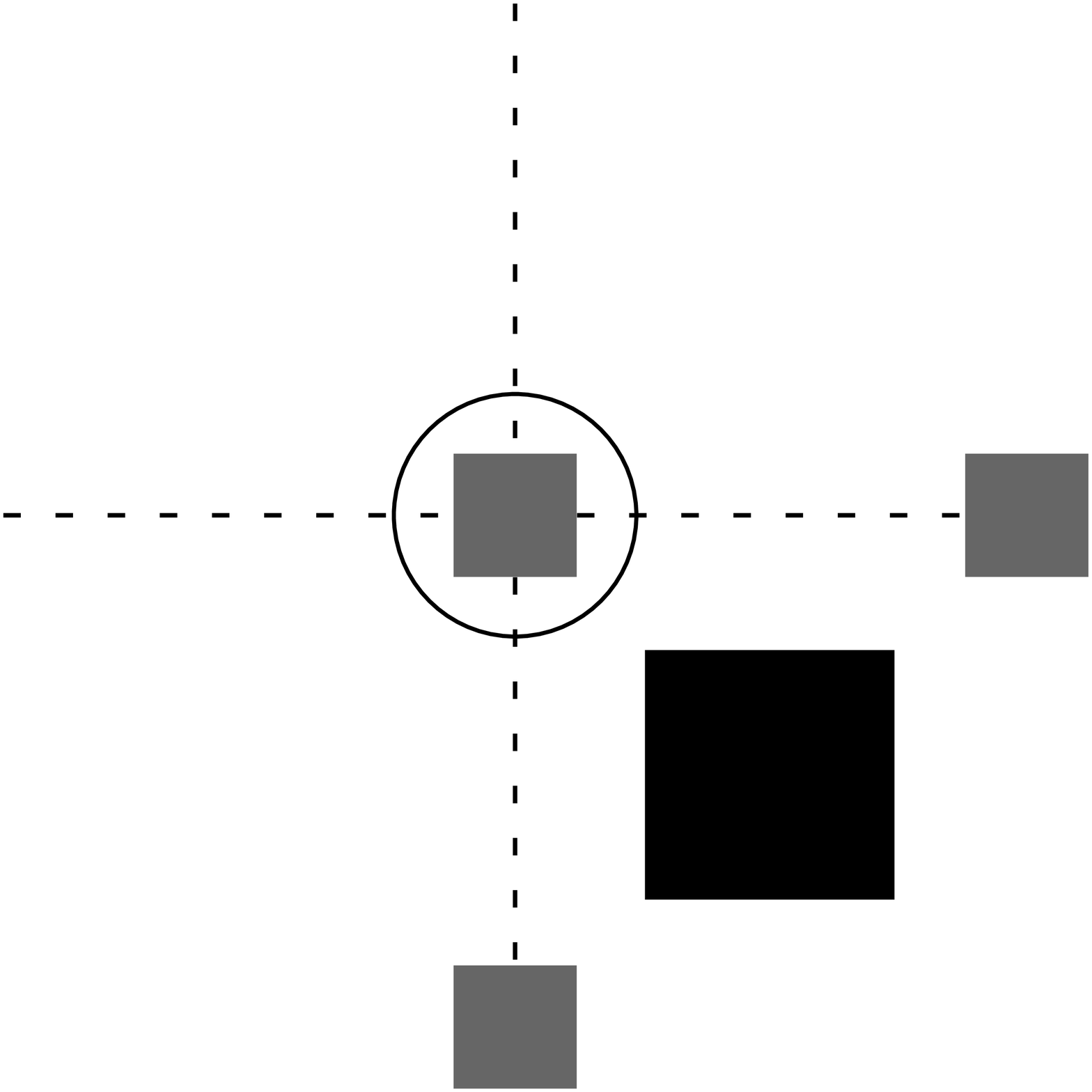}}
\caption{The circled boundary particle of type $\ta$ belongs to: (a) class $1$; (b) class $2$; 
(c) class $3$; (d) class $4$.}
\label{fig-boundary_classes}
\end{figure}
Let $M_k(c)$ be the number of missing bonds of particles of class $k$ in 
cluster $c$, and $A_k(c)$ the number of edges incident to particles of class 
$k$ in cluster $c$. Then
\begin{equation}
M_1(c) = 2,\,A_1(c) = 2; \quad 
M_2(c) = 2,\,A_2(c) = 4; \quad 
M_3(c) = 1,\,A_3(c) = 2; \quad 
M_4(c) = 3,\,A_4(c) = 2. 
\end{equation}
Let $N_k(c)$ be the number of particles of class $k$ of type $\ta$ in cluster $c$.
Observing that a cluster has two concave corners per particle of class $2$, one 
concave corner per particle of class $3$ and one convex corner per particle of 
class $4$, we can write
\begin{equation}
\label{eq claimed number of missing bond rewritten}
\cT(c) = 2P(c) - 2N_2(c) - N_3(c) + N_4(c).
\end{equation}
Since the dual perimeter of a cluster is equal to its total number of 
dual edges, we have
\begin{equation}
\label{eq dual perimeter in terms of particles of type 1}
2P(c) = \sum_{k=1}^4 A_k(c)N_k(c) = 2N_1(c) + 4N_2(c) + 2N_3(c) + 2N_4(c)
\end{equation}
(the sum counts each edge of the $\btile$ twice). The total number of missing 
bonds, on the other hand, is
\begin{equation}
\label{eq number of missing bonds}
M(c) = \sum_{k=1}^4 M_k(c)N_k(c) = 2N_1(c) + 2N_2(c) + N_3(c) + 3N_4(c).
\end{equation}
Combining (\ref{eq claimed number of missing bond rewritten}--\ref{eq number of missing bonds}), 
we arrive at $\cT(c)=M(c)$.
\end{proof}


\section{Proof of Theorem~\ref{theorem ground states}: identification of $\groundset$}
\label{sec ground states}

Recall that $\Lambdaminus$ (the part of $\Lambda$ where particles interact) is an 
$(L + \tfrac{1}{2}) \times (L + \tfrac{1}{2})$ dual square with $L>2\ell\starred$. 
Let $\eta_{\text{stab}},\eta_{\text{stab}}\prm$ be the configurations consisting of 
a $\btiled$ dual square of  size $L$ with even parity, respectively, odd parity. 
These two configurations have the same energy. Theorem~\ref{theorem ground states} 
says that $\groundset=\{\eta_{\text{stab}},\eta_{\text{stab}}\prm\}=\boxplus$. 
Section~\ref{sec standard configurations are optimal btiled configurations} contains 
two lemmas about $\btiled$ configurations with minimal energy. Section~\ref{sec
configurations of minimal energy} uses these two lemmas to prove Theorem~\ref{theorem
ground states}.


\subsection{Standard configurations  are minimizers among $\btiled$ configurations}
\label{sec standard configurations are optimal btiled configurations}

\begin{lemma}
\label{lemma optimal $2$-tiled configurations} 
Within $\molset$, the standard configurations achieve the minimal energy.
\end{lemma}

\begin{proof}
Recall from item~\ref{def group connected particles and bonds} in Section~\ref{sec def} 
that
\begin{equation}
\label{Hform}
H(\eta) = \Da\na(\eta)+ \Db\nb(\eta)-UB(\eta).
\end{equation}
In $\molset$ both $\nb$ and $B=4\nb$ are fixed, and hence $\min_{\eta\in\molset} H(\eta)$ 
is attained at a configuration minimizing $\na$. By Lemma~\ref{lemma number of bonds
in a molecule cluster}, if $\eta\in\molset$, then
\begin{equation}
\label{nanbform}
\na(\eta) = \tfrac14[B(\eta) + \cT(\eta)], \qquad \nb(\eta)=\tfrac14 B(\eta).
\end{equation}
Hence, to minimize $\na(\eta)$ we must minimize $\cT(\eta)$. The claim therefore follows 
from Lemma~\ref{lemma alpha optimality}(i).
\end{proof}

For a standard configuration the computation of the energy is straightforward. For 
$\ell\in\N$, $\zeta\in\{0,1\}$ and $k\in\N_0$ with $k \leq \ell+\zeta$, let 
$\eta^{\ell,\zeta,k}$ denote the standard configuration consisting of an 
$\ell\times(\ell+\zeta)$ (quasi-)square with a bar of length $k$ attached to
one of its longest sides (see Fig.~\ref{fig-standardconf}).

\begin{figure}[htbp]
\begin{centering}
\includegraphics[width=0.2\textwidth]{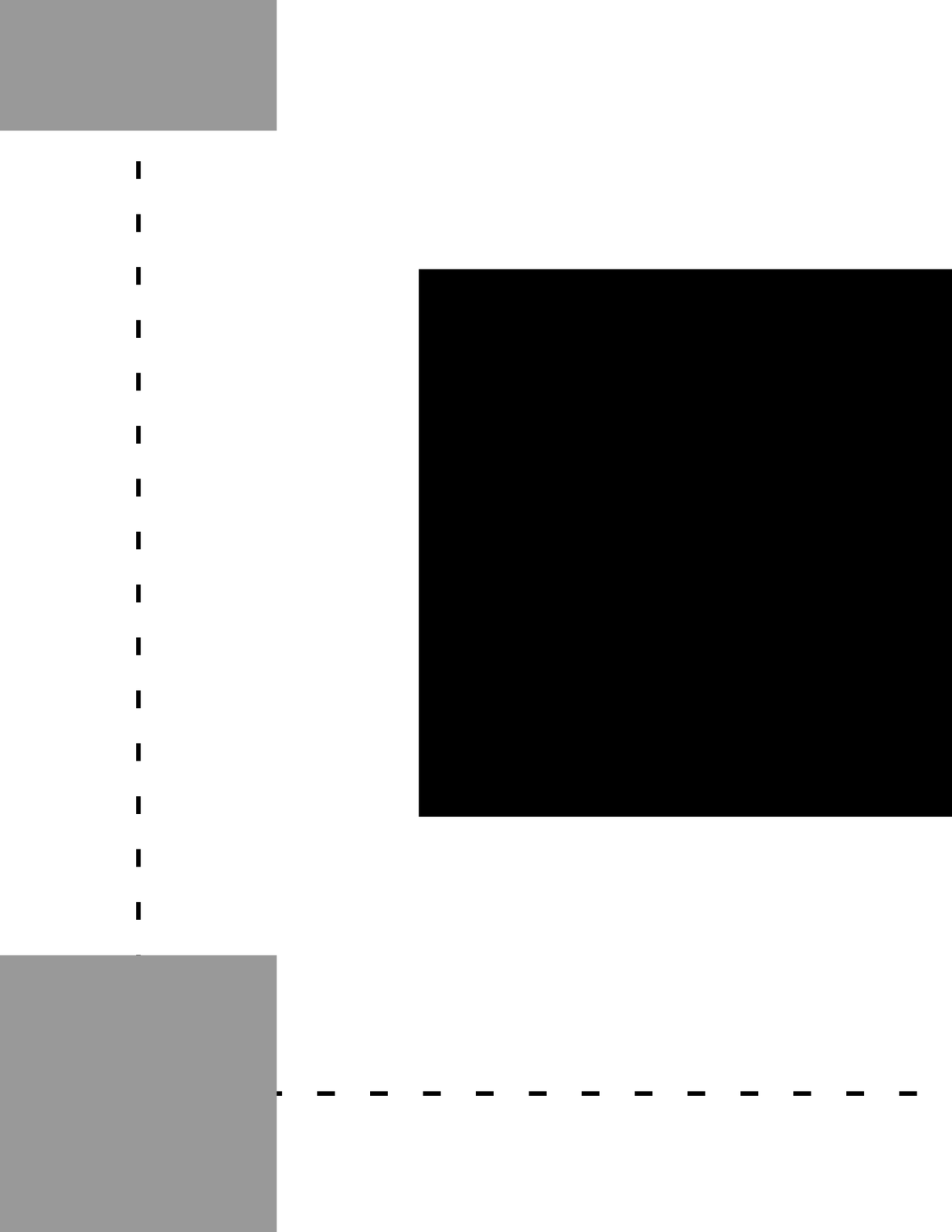}
\par\end{centering}
\caption{A standard configuration with $\ell = 7, \zeta = 1$ and $k = 5$.}
\label{fig-standardconf}
\end{figure}

\begin{lemma}
\label{lemma energy of btiled standard configurations}
The energy of $\eta^{\ell,\zeta,k}$ is (recall {\rm \eqref{epsdef}})
\begin{equation}
\label{eq energy of bitiled standard configuration}
H(\eta^{\ell,\zeta,k}) = -\epsi[\ell(\ell+\zeta) + k] 
+ \Da[\ell + (\ell + \zeta) + 1 + 1_{\{k > 0\}}].
\end{equation}
\end{lemma}

\begin{proof}
Note that $P(\eta^{\ell,\zeta,k})=2[\ell+(\ell+\zeta)+1_{\{k > 0\}}]$ 
and $Q(\eta^{\ell,\zeta,k})=0$, so that 
\begin{equation}
\label{Tid}
\cT(\eta^{\ell,\zeta,k})=4[\ell+(\ell+\zeta)+1+1_{\{k > 0\}}]. 
\end{equation}
Also note that
\begin{equation}
\label{Bid}
B(\eta^{\ell,\zeta,k}) = 4[\ell+(\ell+\zeta)+k],
\end{equation}
because all particles of type $\tb$ are saturated. However, by (\ref{Hform}--\ref{nanbform}), 
we have 
\begin{equation}
\label{Hformalt}
H(\eta^{\ell,\zeta,k}) = -\tfrac14\epsi B(\eta^{\ell,\zeta,k})
+ \tfrac14\cT(\eta^{\ell,\zeta,k})\Da,
\end{equation} 
and so the claim follows by combining (\ref{Tid}--\ref{Hformalt}).
\end{proof}

Note that the energy increases by $\Da-\epsi$ (which is $>0$ if and only if  
$\ell\starred\geq 2$ by (\ref{eq value of critical length})) when a bar of 
length $k=1$ is added, and decreases by $\epsi$ each time the bar is extended. 
Note further that
\begin{equation}
\label{Hjumps}
H(\eta^{\ell,1,0})-H(\eta^{\ell,0,0}) = \Da-\ell\epsi, 
\qquad
H(\eta^{\ell+1,0,0})-H(\eta^{\ell,1,0}) = \Da-(\ell+1)\epsi, 
\end{equation}
which show that the energy of a growing sequence of standard configurations 
goes up when $\ell<\ell\starred$ and goes down when $\ell\geq\ell\starred$. 
The highest energy is attained at $\eta^{\ell\starred-1,1,1}$, which is the 
critical droplet in Fig.~\ref{fig-critical_droplet}.

It is worth noting that $H(\eta_{s}^{2\ell\starred,0,0}) < 0$, i.e., the energy 
of a dual square of side length $2\ell\starred$ is lower than the energy of 
$\Box$. This is why we assumed $L>2\ell\starred$, to allow for $H(\boxplus)<H(\Box)$.


\subsection{Stable configurations}
\label{sec configurations of minimal energy}

In this section we use Lemmas~\ref{lemma optimal $2$-tiled configurations}--\ref{lemma 
energy of btiled standard configurations} to prove Theorem~\ref{theorem ground states}.

\begin{proof} 
Let $\eta$ denote any configuration in $\groundset$. Below we will show that:
\begin{itemize}
\item[(A)] $\eta$ does not contain any particle in $\partial^-\Lambda$.
\item[(B)] $\eta$ is a $\btiled$ configuration, i.e., $\eta\in\molset$ for 
some $\nb$ ($= \nb(\eta)$).
\end{itemize}
Once we have (A) and (B), we observe that $\eta$ cannot contain a number of $\btiles$ 
larger than $L^{2}$. Indeed, consider the tile support of $\eta$. Since $\Lambdaminus$ 
is an $(L+\tfrac12) \times (L+\tfrac12)$ dual square, if the tile support of $\eta$ fits 
inside $\Lambdaminus$, then so does the dual circumscribing rectangle of $\eta$. But any 
rectangle of area $\ge L^2$ has at least one side of length $L + 1$. Hence $\nb(\eta) 
\leq L^{2}$, and therefore the number of $\btiles$ in $\eta$ is at most $L^{2}$. By 
Lemmas~\ref{lemma optimal $2$-tiled configurations}--\ref{lemma energy of btiled standard
configurations}, the global minimum of the energy is attained at the largest dual quasi-square 
that fits inside $\Lambdaminus$, since $L > 2 \ell\starred$. We therefore conclude that
$\eta\in\{\eta_{\text{stab}}, \eta_{\text{stab}}\prm\}$, which proves the claim. 

\medskip\noindent
\underline{Proof of (A)}. 
Since in $\partial^{-}\Lambda$ particles do not feel any 
interaction but have a positive energy cost, removal of a particle from 
$\partial^{-}\Lambda$ always lowers the energy. 

\medskip\noindent
\underline{Proof of (B)}. 
We note the following three facts:
\begin{enumerate}
\item[(1)] 
$\eta$ does not contain isolated particles of type $\ta$.
\item[(2)] 
$\partial^{-}\Lambdaminus$ does not contain any particle of type $\tb$.
\item[(3)] 
All particles of type $\tb$ in $\eta$ have all their neighboring sites 
occupied by a particle.
\end{enumerate}
For (1), simply note that the configuration obtained from $\eta$ by removing isolated 
particles has lower energy. For (2), note that particles in $\partial^-\Lambdaminus$ 
have at most two active bonds. Therefore, if $\eta$ would have a particle of type 
$\tb$ in $\partial^-\Lambdaminus$, then the removal of that particle would lower 
the energy, because $\Db - \Da > 2U$ and $\Da > 0$ (recall (\ref{subpropmetreg})) 
imply $\Db > 2U$. For (3), note that if a particle of type $\tb$ has an empty neighboring 
site, then the addition of a particle of type $\ta$ at this site lowers the energy, 
because $\Da < U$ (recall (\ref{subpropmetreg})).

We can now complete the proof of (B) as follows. The constraint $\Db - \Da > 2U$ implies 
that any particle of type $\tb$ in $\eta$ must have at least three neighboring sites 
occupied by a particle of type $\ta$. Indeed, the removal of a particle of type $\tb$ 
with at most two active bonds lowers the energy. But the fourth neighboring site 
must also be occupied by a particle of type $\ta$. Indeed, suppose that this site 
would be occupied by a particle of type $\tb$. Then this particle would have at most 
three active bonds. Consider the configuration $\tilde{\eta}$ obtained from $\eta$ after 
replacing this particle by a particle of type $\ta$. Then $B(\tilde{\eta})-B(\eta)\geq -2$, 
$\na(\tilde{\eta})-\na(\eta)=1$ and $\nb(\tilde{\eta})-\nb(\eta)=-1$. Consequently, 
$H(\tilde{\eta})-H(\eta)\leq \Da-\Db + 2U < 0$. Hence, any particle of type $\tb$ in 
$\eta$ must be saturated.
\end{proof}


\section{Proof of Theorem~\ref{theorem communication height}: identification
of $\Gamma\starred=\comlev(\Box, \boxplus)$}
\label{sec identification of gamma}

In Section~\ref{sec id sub} we prove Theorem~\ref{theorem communication height} subject 
to the following lemma.

\begin{lemma}
\label{lemma key}
For any $n_\tb \leq L^{2}$, the configurations of minimal energy with $\nb$ particles
of type $\tb$ belong to $\molset$, i.e., are $\btiled$ configurations.  
\end{lemma}

\noindent
The proof of this lemma is given in Section~\ref{sec proof of lemma key}.


\subsection{Proof of Theorem~\ref{theorem communication height} subject to
Lemma~\ref{lemma key}}
\label{sec id sub}

\begin{proof}
For $\cY\subset\cX$, define the \emph{external boundary} of $\cY$ by $\boundary\cY=\{\eta\in\cX
\backslash\cY\colon\,\exists\eta\prm\in\cY,\,\eta\leftrightarrow\eta\prm \}$ and the 
\emph{bottom} of $\cY$ by $\bottom(\cY)=\arg\min_{\eta\in\cY} H(\eta)$. According to 
Manzo, Nardi, Olivieri and Scoppola~\cite{MNOS04}, Section 4.2, $\comlev(\Box,\boxplus) 
= \min_{\eta\in\boundary\cB} H(\eta)$ for $\cB\subset\cX$ any (!) set with the following 
properties:
\begin{enumerate} 
\item[(I)]
$\cB$ is connected via allowed moves, $\Box\in\cB$ and $\boxplus\notin\cB$. 
\item[(II)]
There is a path $\refpath\colon\,\Box\rightarrow\boxplus$ such that $\left\{\arg\max_{\eta 
\in \refpath} H(\eta)\right\}\cap\bottom(\boundary\cB) \neq \emptyset$.
\end{enumerate}
Thus, our task is to find such a $\cB$ and compute the lowest energy of $\boundary\cB$.

For (I), choose $\cB$ to be the set of all configurations $\eta$ such that $\nb(\eta) 
\leq\ell\starred(\ell\starred-1)+1$. Clearly this set is connected, contains $\Box$ and 
does not contain $\boxplus$.

For (II), choose $\refpath$ as follows. A particle of type $\tb$ is brought inside $\Lambda$
($\D H=\Db$), moved to the origin and is saturated by four times bringing a particle of type 
$\ta$ ($\D H=\Da$) and attaching it to the particle of type $\tb$ ($\D H=-U$). After 
this first $\btile$ has been completed, $\refpath$ follows a sequence of increasing 
$\btiled$ dual quasi-squares. The passage from one quasi--square to the next is obtained 
by \emph{adding} a $\abbar$ to one of the longest sides, as follows. First a particle of 
type $\tb$ is brought inside $\Lambda$ ($\D H = \Db$) and is attached to one of the longest 
sides of the quasi-square ($\D H = - 2U$). Next, twice a particle of type $\ta$ is brought 
inside the box ($\D H = \Da$) and is attached to the (not yet saturated) particle of type 
$\tb$ ($\D H = - U$) in order to complete a $\btiled$ protuberance. Finally, the $\abbar$ 
is completed by bringing a particle of type $\tb$ inside $\Lambda$ ($\D H = \Db$), moving 
it to a concave corner ($\D H = - 3U$), and saturating it with a particle of type $\ta$ 
($\D H = \Da$, respectively, $\D H = -U$). It is obvious that $\refpath$ eventually hits 
$\boxplus$. The path $\refpath$ is referred to as the \emph{reference path for the 
nucleation}.

Call $\eta\starred$ the
configuration in $\refpath$ consisting of an $\ell\starred
\times (\ell\starred - 1)$ quasi-square, a $\btiled$ protuberance attached to one 
of its longest sides, and a free particle of type $\tb$ (see Fig.~\ref{fig-critconf};
there are many choices for $\refpath$
depending on where the $\btiled$ protuberances are added; all these choices are equivalent.
Note that, in the notation of
Lemma~\ref{lemma
energy of btiled standard configurations},
$\eta\starred=\eta^{\ell\starred-1,1,1} + \freeb$,
where $+ \freeb$ denotes the addition of a free particle of type $\tb$. 
Observe that:
\begin{itemize} 
\item[(a)] 
$\refpath$ exits $\cB$ via the configuration $\eta\starred$; 
\item[(b)] 
$\eta\starred\in\bottom(\boundary\cB)$; 
\item[(c)] 
$\eta\starred \in \left\{\arg\max_{\eta\in\refpath} H(\eta) \right\}$.
\end{itemize} 
Observation (a) is obvious, while (b) follows from Lemmas~\ref{lemma optimal $2$-tiled
configurations} and \ref{lemma key}. To see (c), note the following: (1) The total energy 
difference obtained by adding a $\abbar$ of length $\ell$ on the side of a $\btiled$ cluster 
is $\D H(\text{adding a } \abbar) =  \Da-\epsi\ell$, which changes sign at $\ell=\ell\starred$ 
(recall \eqref{Hjumps}); (2) The configurations of maximal energy in a sequence of growing 
quasi-squares are those where a free particle of type $\tb$ enters the box after the $\btiled$ 
protuberance has been completed. Thus, within energy barrier $2 \Da + 2 \Db - 4U=4U-\epsi$ 
the $\abbar$ is completed downwards in energy. This means that, after configuration $\eta\starred$ 
is hit, the dynamics can reach the $\btiled$ dual square of $\ell\starred \times \ell\starred$ 
while staying below the energy level $H(\eta\starred)$. Since all $\btiled$ dual quasi-squares 
larger than $\ell\starred \times (\ell\starred - 1)$ have an energy smaller than that of the
$\btiled$ dual quasi-square $\ell\starred \times (\ell\starred - 1)$ itself, the path $\refpath$ 
does not again reach the energy level $H(\eta\starred$). 

Because of (a--c), we have $\comlev(\Box,\boxplus) = H(\eta\starred)$. To complete the proof, 
use Lemma~\ref{lemma energy of btiled standard configurations} to compute
\begin{equation}
H(\eta\starred) = H(\eta^{\ell\starred-1,1,1} + \freeb) = -\epsi[\ell\starred(\ell\starred-1)+1] 
+ \Da(2 \ell\starred+1) +\Db.
\end{equation}
\end{proof}

\begin{figure}[htbp]
\begin{centering}
\includegraphics[width=0.35\textwidth]{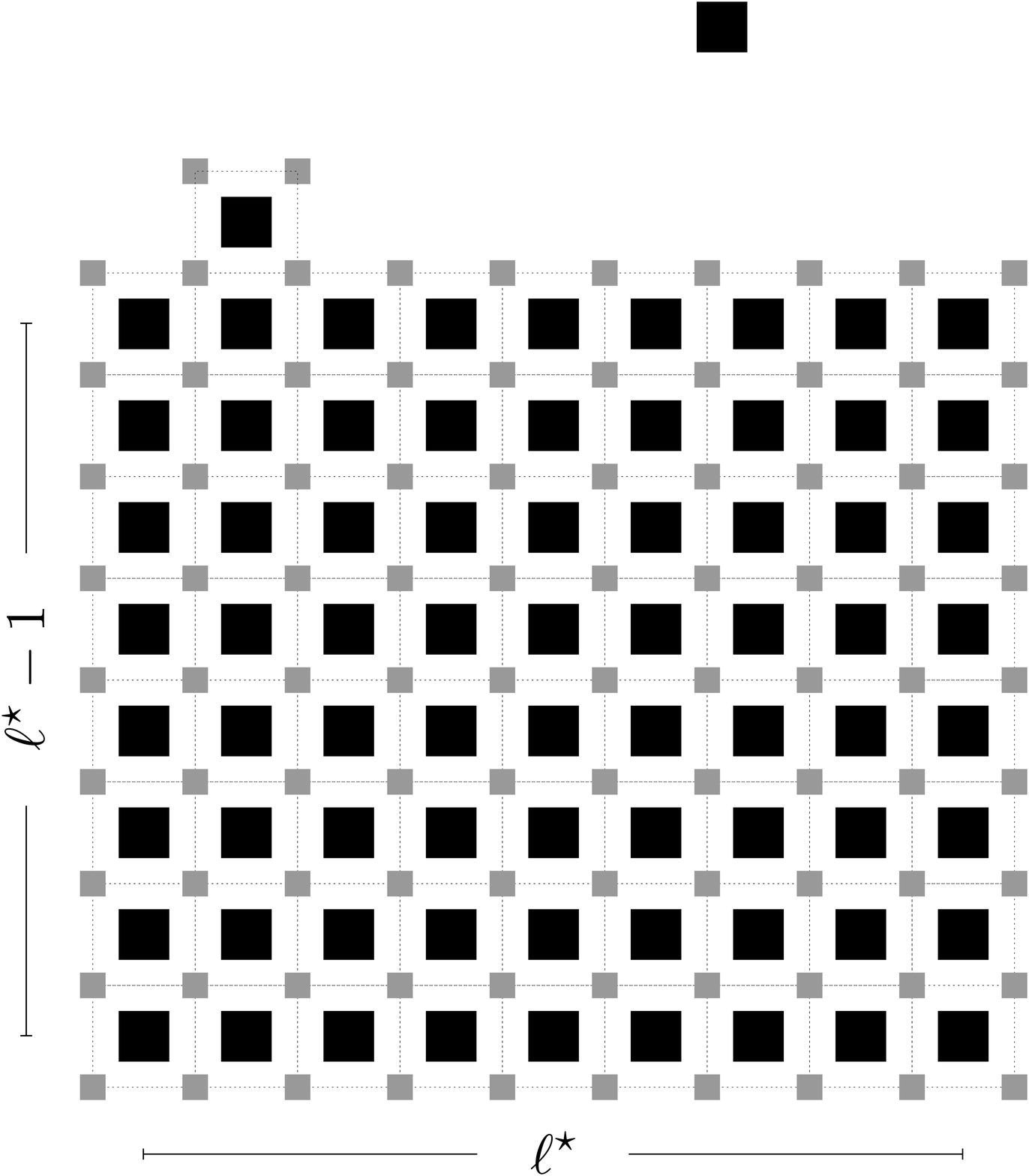}
\par\end{centering}
\caption{A critical configuration $\eta\starred$. This is the dual version of the
critical droplet in Fig.~\ref{fig-critical_droplet}.}
\label{fig-critconf}
\end{figure}


\subsection{Proof of Lemma~\ref{lemma key}}
\label{sec proof of lemma key} 

The proof of Lemma~\ref{lemma key} is carried out in two steps. In Section~\ref{sec scmitt}
we show that the claim holds for single-cluster configurations with a fixed number of 
particles of type $\tb$. In Section~\ref{sec cmefnp} we extend the claim to general 
configurations with a fixed number of particles of type $\tb$. 


\subsubsection{Single clusters of minimal energy are $\btiled$ clusters}
\label{sec scmitt}

\begin{lemma}
\label{lemma optimality of btiled clusters} 
For any single-cluster configuration $\eta \in \nbset \backslash \molset$ there exists 
a configuration $\tilde{\eta} \in \molset$ such that $H(\tilde{\eta})< H(\eta)$.
\end{lemma}

\begin{proof}
Pick any $\eta \in \nbset \backslash \molset$. Every neighboring site of a particle 
of type $\tb$ in the cluster is either empty or occupied by a particle of type $\ta$,
and there is at least one non-saturated particle of type $\tb$. Since $\eta$ consists 
of a single cluster, $\tilde{\eta}$ can be constructed in the following way:
\begin{itemize}
\item 
$\tilde{\eta}(i) = \eta(i)$ for all $i \in \supp(\eta)$.
\item 
$\tilde{\eta}(j) = 1$ for all $j \notin \supp(\eta)$ such that there exists 
an $i \sim j$ with $\eta(i) = 2$.
\end{itemize}
Since
\begin{equation}
\label{eq comparison single cluster}
\begin{aligned}
H(\eta) &= \Da\na(\eta) + \Db\nb(\eta)  - UB(\eta),\\
H(\tilde{\eta}) &= \Da\na(\tilde{\eta}) + \Db\nb(\tilde{\eta}) - UB(\tilde{\eta}),
\end{aligned}
\end{equation}
and $\nb(\eta)=\nb(\tilde{\eta})$, we have
\begin{equation}
\label{Hdiff}
H(\tilde{\eta}) - H(\eta) 
=  \Da[\na(\tilde{\eta}) - \na(\eta)] - U[B(\tilde{\eta}) - B(\eta)].
\end{equation}
By construction, $B(\tilde{\eta}) - B(\eta) \geq \na(\tilde{\eta}) - \na(\eta) > 0$. 
Since $0 < \Da < U$ (recall \eqref{subpropmetreg}), it follows from \eqref{Hdiff} that 
$H(\tilde{\eta}) < H(\eta)$.
\end{proof}


\subsubsection{Configurations of minimal energy with fixed number of particles of type $\tb$}
\label{sec cmefnp}

\begin{lemma}
\label{lemma optimality fixed number of particles of type tb}
For any $\nb$ and any configuration $\eta \in \nbset$ consisting of at least two 
clusters, any configuration $\eta\starred$ such that $\eta\starred$ is a single 
cluster, $\eta\starred \in \molset$ and $\eta\starred$ is a standard configuration
satisfies $H(\eta\starred) < H(\eta)$.
\end{lemma}

\begin{proof}
Let $\eta \in \nbis{\nb}$ be a configuration consisting of $k > 1$ clusters, labeled 
$c_1,\ldots,c_k$. Let $\eta^{\nb(c_{i})}$ denote any standard configuration with 
$\nb(c_{i})$ particles of type $\tb$. By Lemmas~\ref{lemma optimal $2$-tiled 
configurations} and \ref{lemma optimality of btiled clusters}, we have
\begin{equation}
\label{est1a}
H(\eta) = \sum_{i=1}^{k} H(c_{i}) \geq \sum_{i = 1}^{k} H(\eta^{\nb(c_{i})}).
\end{equation}
By Lemma~\ref{lemma number of bonds in a molecule cluster}, we have
(recall (\ref{epsdef}))
\begin{equation}
\label{est1b}
\begin{aligned}
\sum_{i=1}^{k} H(\eta^{\nb(c_{i})})	
& =	\sum_{i=1}^{k} \big[\Da\na(\eta^{\nb(c_{i})}) + \Db\nb(\eta^{\nb(c_{i})}) 
-UB(\eta^{\nb(c_{i})})\big]\\
& = \sum_{i=1}^{k} \big[
\Da \big\{ \nb(\eta^{\nb(c_{i})}) + \tfrac14\cT(\eta^{\nb(c_{i})}) \big\} 
+ \Db\nb(\eta^{\nb(c_{i})}) - U 4\nb(\eta^{\nb(c_{i})})\big]\\
& =	\sum_{i = 1}^{k} \big[- \epsi \nb (\eta^{\nb(c_{i})})
+ \tfrac14 \Da\cT(\eta^{\nb(c_{i})})\big].
\end{aligned}
\end{equation}
But from Lemma~\ref{lemma alpha optimality} it follows that
\begin{equation}
\label{est1c}
\sum_{i=1}^{k} \cT(\eta^{\nb(c_{i})})  > \cT\big(\eta^{ \sum_{i=1}^{k} \nb(c_{i}) }\big),
\end{equation}
where $\eta^{\sum_{i=1}^{k} \nb(c_{i})}$ denotes any standard configuration with $\sum_{i=1}^{k}
\nb(c_{i}) = \nb(\eta)$ particles of type $\tb$. Combining (\ref{est1a}--\ref{est1c}), we 
arrive at
\begin{equation}
H(\eta) > - \epsi\nb(\eta)+ \tfrac14\Da\cT(\eta^{\nb(\eta)}) = H(\eta^{\nb(\eta)}).
\end{equation}
\end{proof}


\section{Proof of Theorem~\ref{theorem recurrence}: upper bound on $\stablev{\eta}$ for
$\eta\notin\{\Box,\boxplus\}$}
\label{sec Recurrence}

In this section we show that for any configuration $\eta\notin\{\Box,\boxplus\}$ 
it is possible to find a path $\omega\colon\,\eta\to\eta\prm$ with $\eta\prm\in
\{\Box, \boxplus\}$ such that $\max_{\xi\in\omega} H(\xi) \leq H(\eta)+V\starred$ 
with $V\starred \le10U-\Da$ and $\eta\prm\in I_\eta$. By Definition~\ref{def1}(c--e), 
this implies that $\stablev{\eta}\leq V\starred$ for all $\eta\notin\{\Box,\boxplus\}$ 
and therefore settles Theorem~\ref{theorem recurrence}.

Section~\ref{sec general reduction} describes an \emph{energy reduction algorithm} 
to find $\omega$. Roughly, the idea is that if $\eta$ contains only ``subcritical 
clusters'', then these clusters can be removed one by one to reach $\Box$, while if 
$\eta$ contains some ``supercritical cluster'', then this cluster can be taken as 
a stepping stone to construct a path to $\boxplus$ that goes via a sequence of 
increasing rectangles. In particular, the supercritical cluster is first extended 
to a $\btiled$ rectangle touching the north-boundary of $\Lambda$, after that it 
is extended to a $\btiled$ rectangle touching the west-boundary and the east-boundary 
of $\Lambda$, and finally it is extended to $\boxplus$.

To carry out this task, six \emph{energy reduction mechanisms} are needed, which 
are introduced and explained in Section~\ref{sec enegredmech}: 
\begin{itemize}
\item
Moving unit holes inside $\btiled$ clusters
(Section~\ref{sec motion of particles inside clusters}).
\item 
Adding and removing $\abbars$ from lattice-connecting rectangles
(Section~\ref{sec growth of clusters}).
\item
Changing bridges into $\abbars$
(Section~\ref{sec main procedures}).
\item 
Maximally expanding $\btiled$ rectangles
(Section~\ref{maximal expansion btiled rectangle}).
\item 
Merging adjacent $\btiled$ rectangles
(Section~\ref{sec sliding adjacent btiled rectangles}).
\item 
Removing subcritical clusters
(Section~\ref{sec preparation procedure}). 
\end{itemize}
Each of Sections~\ref{sec motion of particles inside clusters}--\ref{sec preparation procedure} 
states a definition and a lemma, and uses these to prove a proposition about the relevant 
energy reduction mechanism. The six propositions thus obtained will be crucial for the energy
reduction algorithm in Section~\ref{sec general reduction}.
  
In Section~\ref{sec beams and pillars} we begin by defining beams and pillars,
which are needed throughout Section~\ref{sec enegredmech}.


\subsection{Beams and pillars}
\label{sec beams and pillars}

\begin{lemma}
\label{lemma btiles are optimal tiles}
Let $\eta$ be a configuration containing a tile $t$ that has at least three junction sites 
occupied by a particle of type $\ta$. Then the configuration $\eta\prm$ obtained from $\eta$
by turning $t$ into a $\btile$ satisfies $H (\eta\prm) \leq H(\eta)$.
\end{lemma}

\begin{figure}[htbp]
\renewcommand{\thesubfigure}{(\roman{subfigure})}
\centering
\subfigure[]
{\includegraphics[width=0.055\textwidth]{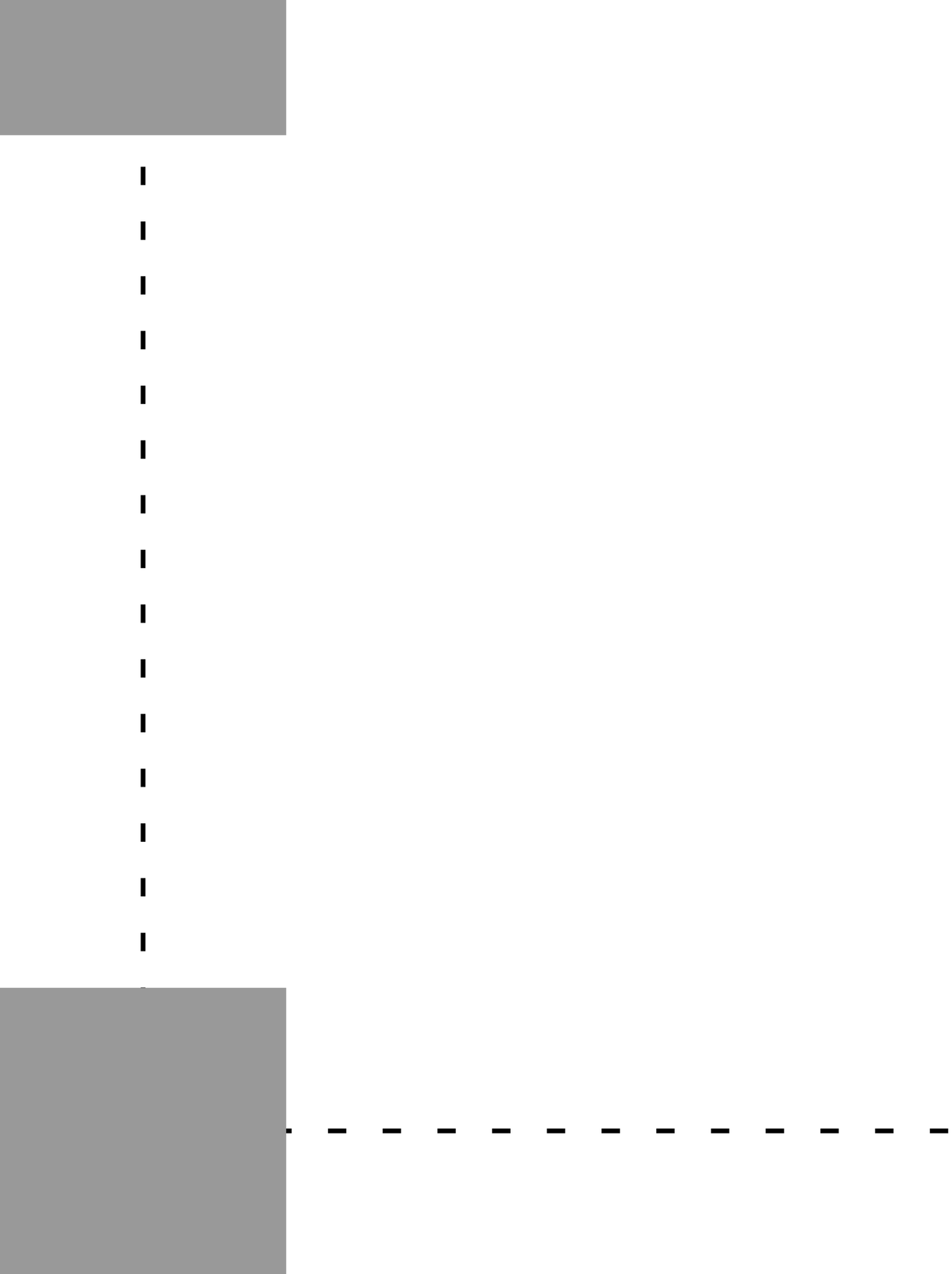}}
\quad
\subfigure[]
{\includegraphics[width=0.055\textwidth]{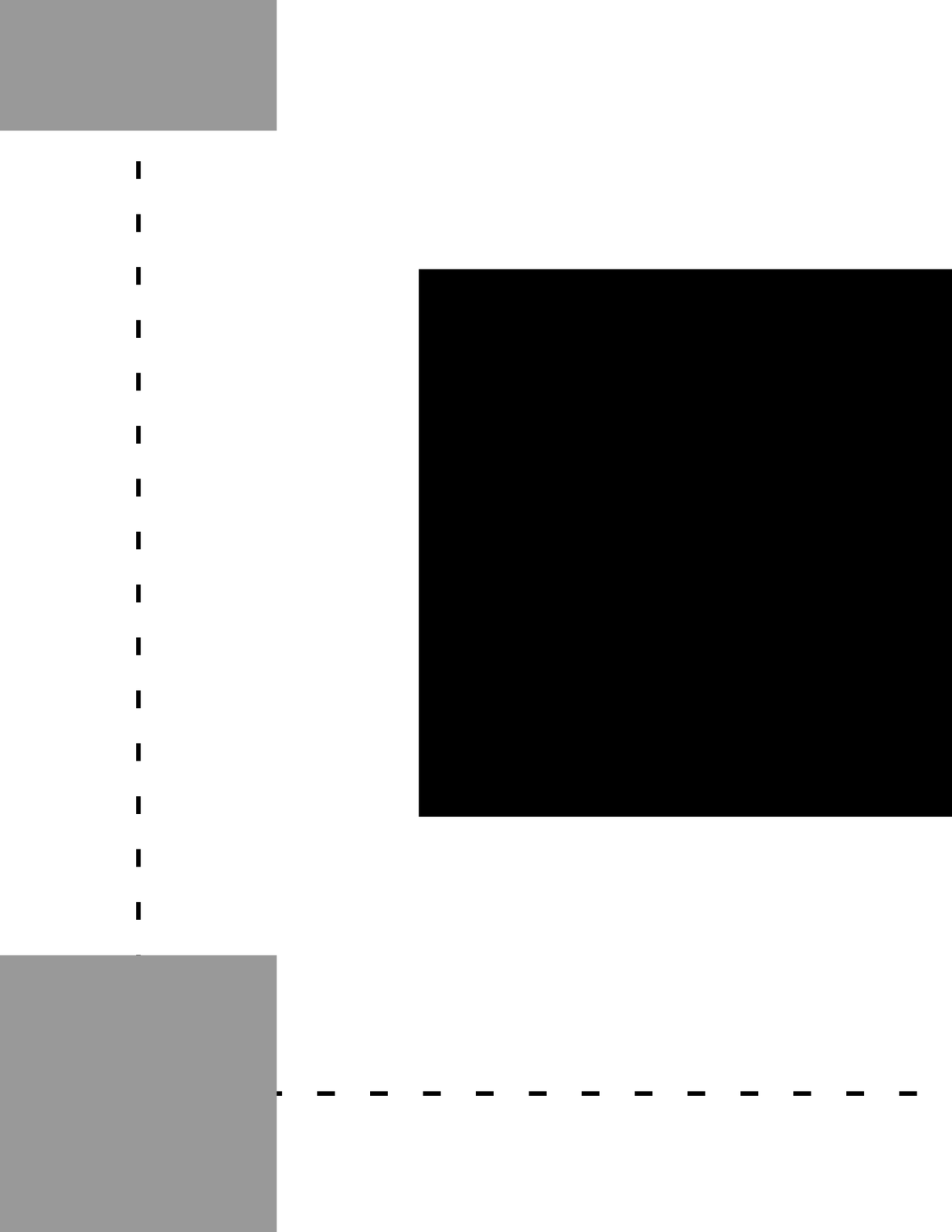}}
\quad
\subfigure[]
{\includegraphics[width=0.055\textwidth]{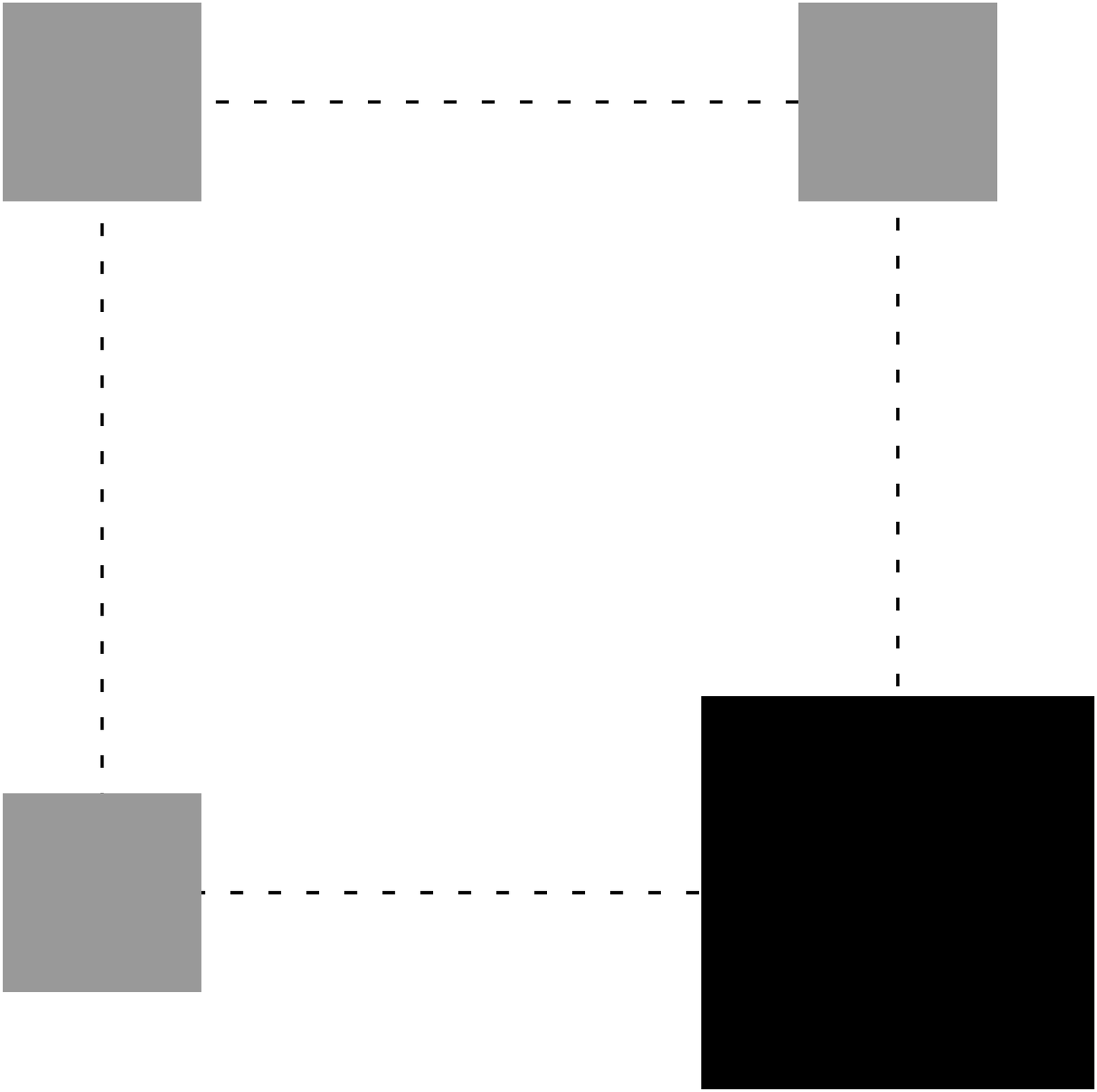}}
\quad
\subfigure[]
{\includegraphics[width=0.055\textwidth]{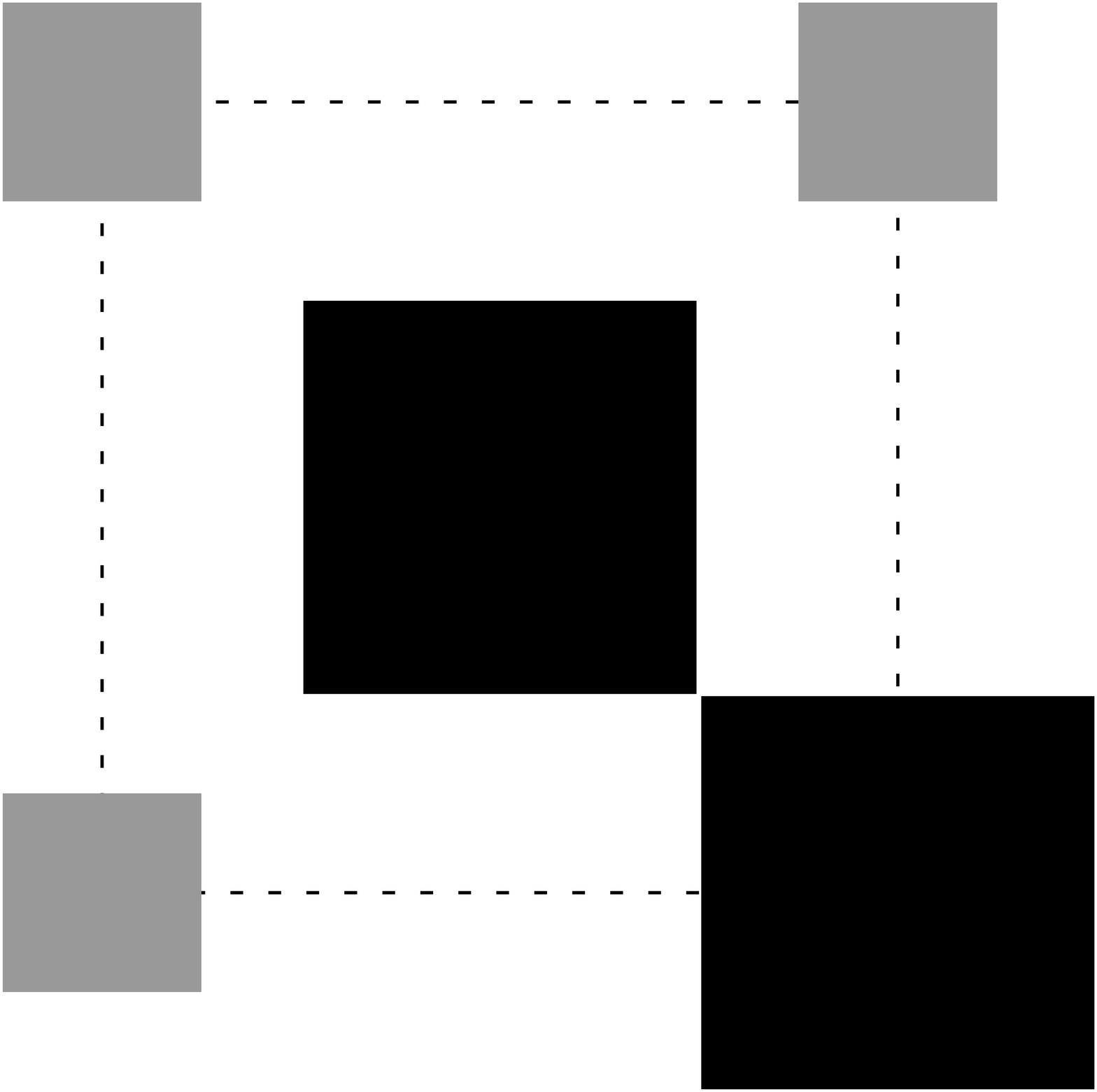}}
\quad
\subfigure[]
{\includegraphics[width=0.055\textwidth]{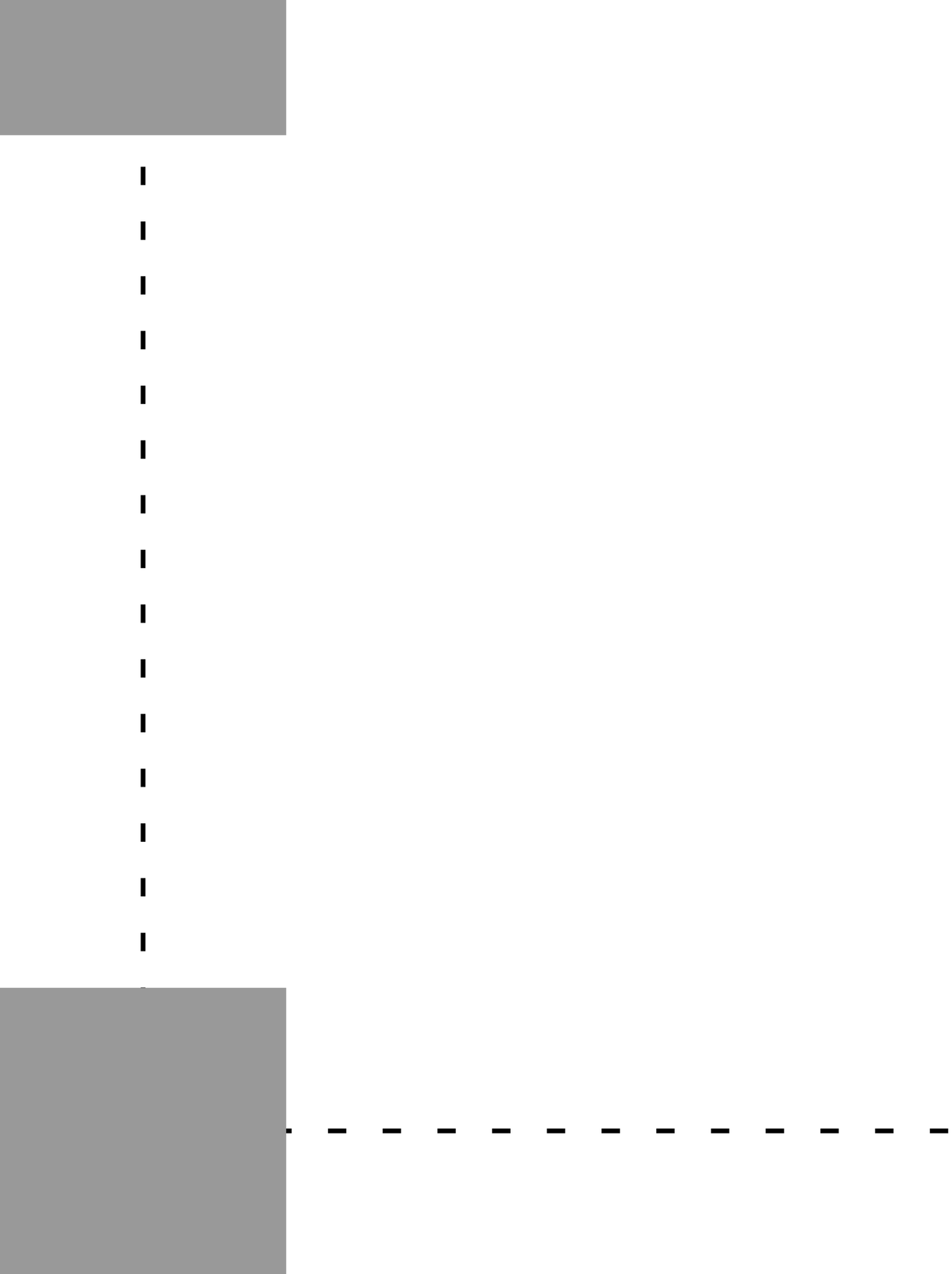}}
\quad
\subfigure[]
{\includegraphics[width=0.055\textwidth]{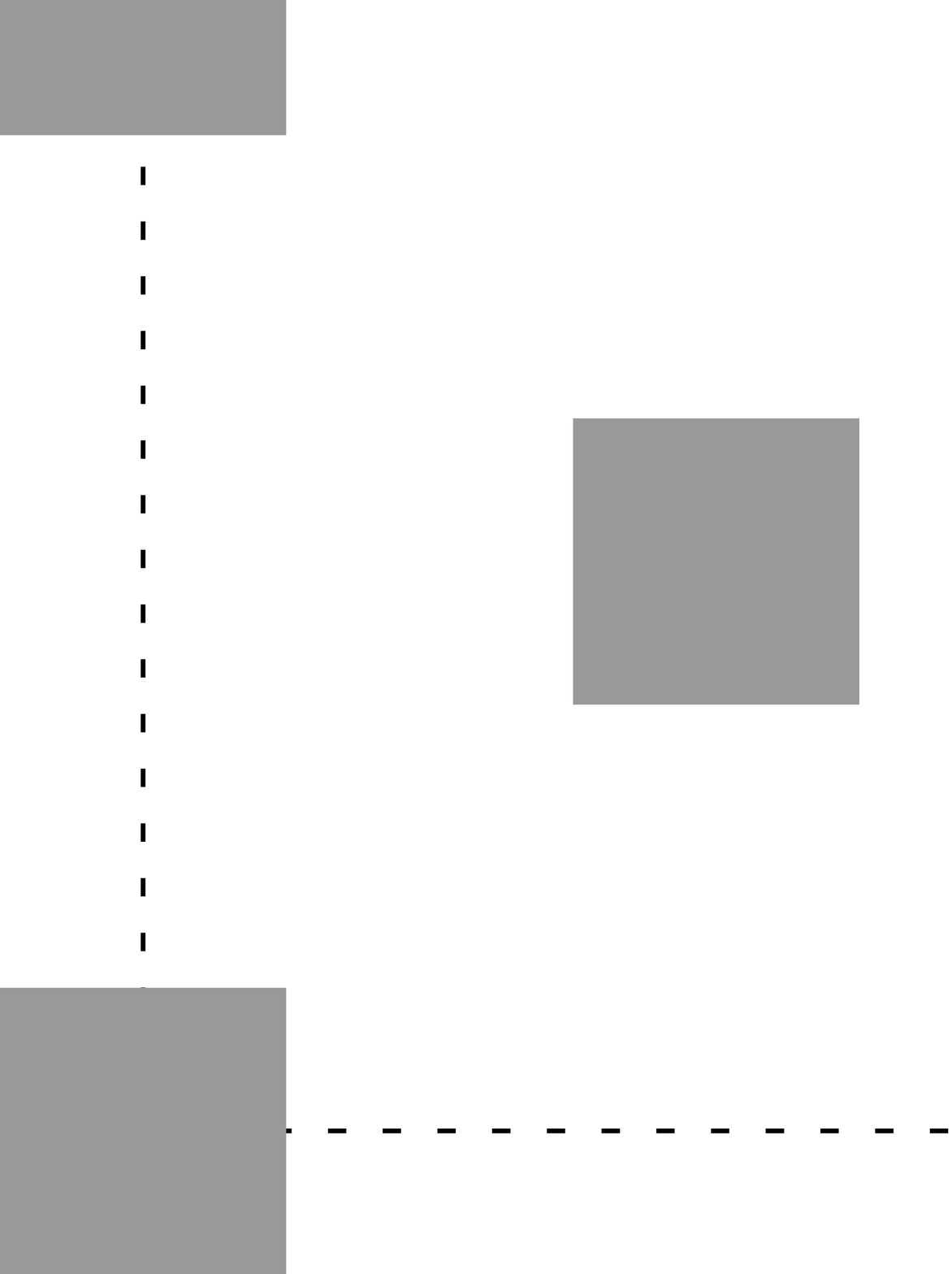}}
\quad
\subfigure[]
{\includegraphics[width=0.055\textwidth]{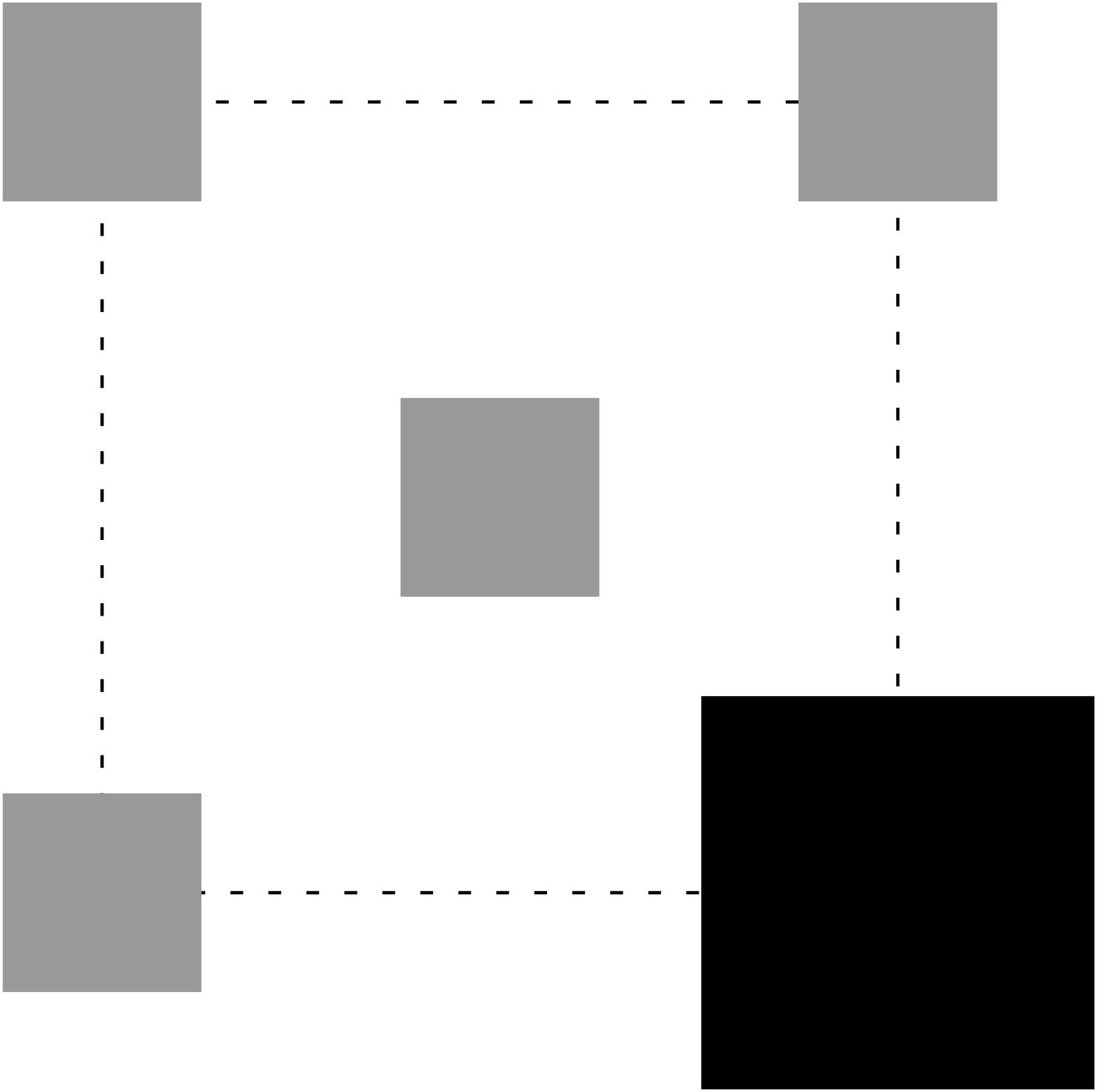}}
\quad
\subfigure[]
{\includegraphics[width=0.05\textwidth]{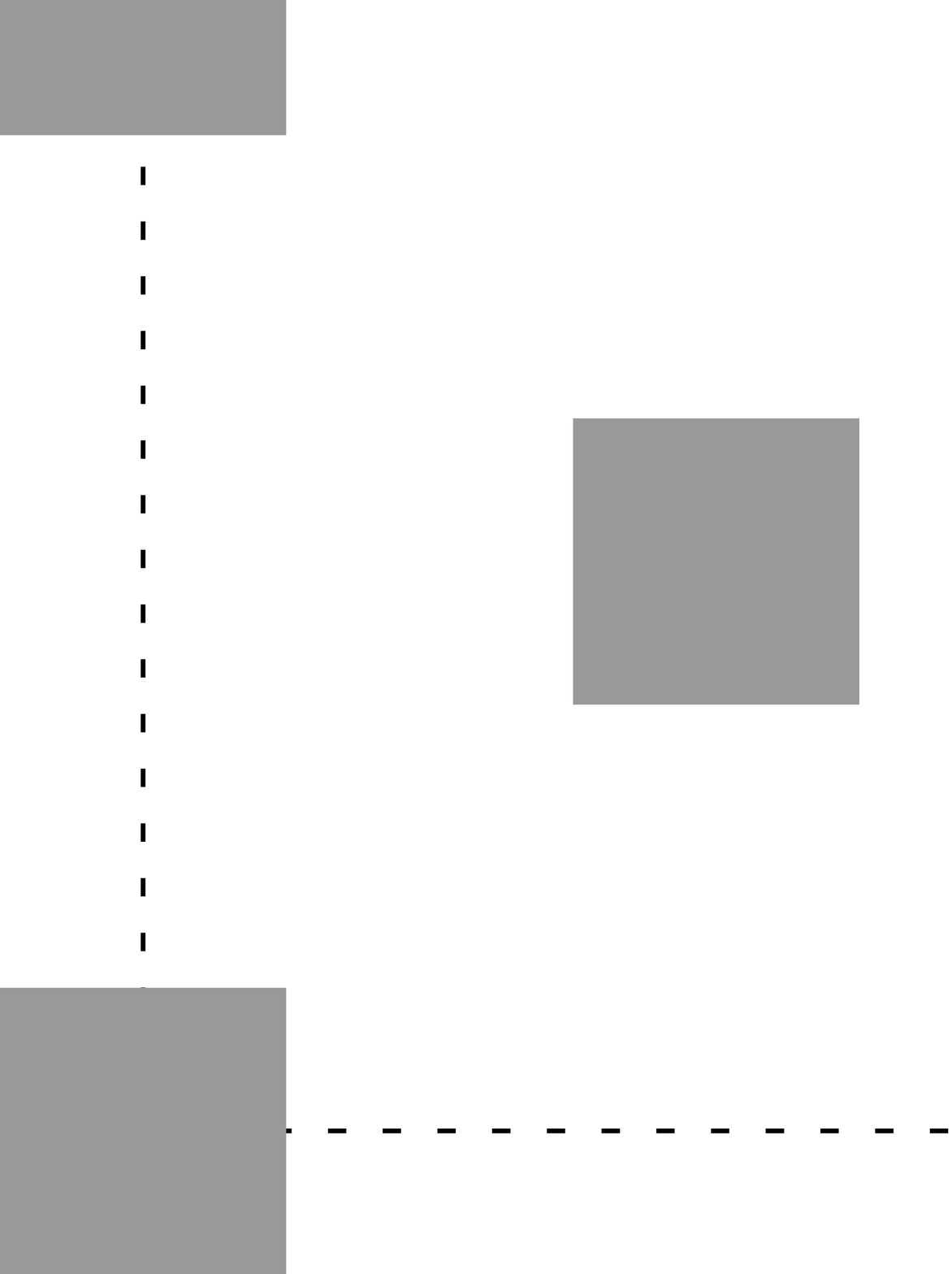}}
\caption{Possible tiles with at least three junction sites occupied by 
a particle of type $\ta$.}
\label{fig-possible_quasitiles}
\renewcommand{\thesubfigure}{(\alpha{subfigure})}
\end{figure}

\begin{proof}
W.l.o.g.\ we may assume that $\eta(t_a)=\eta(t_b)=\eta(t_d)=1$, and that $\eta'$ is
the configuration in Fig.~\ref{fig-2tile}(d), i.e., $\eta'(t_a)=\eta'(t_b)=\eta'(t_c)
=\eta'(t_d)=1$, $\eta'(t_e)=2$. The following eight cases are possible 
(see Fig.~\ref{fig-possible_quasitiles} and recall \eqref{subpropmetreg}):

\begin{enumerate}[(i)]
\item 	
$(\eta(t_{c}),\eta(t_{e})) = (0,0)$.
One particle of type $\ta$ and one particle of type $\tb$ are added, and at least four new 
bonds are activated: $\D H \le \Da + \Db - 4U < 0$.
\item
$(\eta(t_{c}),\eta(t_{e})) = (0,2)$.
One particle of type $\ta$ is added, and one new bond is activated: $\D H = \Da - U < 0$.
\label{ENUM ADDPARTICLEOFTYPE1}
\item
$(\eta(t_{c}),\eta(t_{e})) = (2,0)$.
One particle of type $\tb$ is moved to another site without deactivating any bonds, 
after which case~(\ref{ENUM ADDPARTICLEOFTYPE1}) applies.
\item
$(\eta(t_{c}),\eta(t_{e})) = (2,2)$.
One particle of type $\tb$ with at most three active bonds is replaced by one particle of 
type $\ta$ with at least one active bond: $\D H \leq \Da - \Db + 2U < 0$.		
\item
$(\eta(t_{c}),\eta(t_{e})) = (1,0)$.
One particle of type $\tb$ is added, and four new bonds are activated: $\D H = \Db - 4U < 0$. 
\label{ENUM ADDPARTICLEOFTYPE2}
\item	
$(\eta(t_{c}),\eta(t_{e})) = (0,1)$.
One particle of type $\ta$ is moved to another site without deactivating any active bond, 
one particle of type $\tb$ is added, and at least four new bonds are activated: 
$\D H \le \Db - 4U < 0$.
\item
$(\eta(t_{c}),\eta(t_{e})) = (2,1)$.
Two particles are exchanged without deactivating any bonds: $\D H\leq 0$.
\item 
$(\eta(t_{c}),\eta(t_{e})) = (1,1)$.	
One particle of type $\ta$ is replaced by a particle of type $\tb$, and four new 
bonds are activated: $\D H = \Db -\Da - 4U < 0$. 
\label{ENUM CHANGE1INTO2}
\end{enumerate}
\end{proof}

\begin{definition}
\label{definition of beam}
A beam of length $\ell$ is a row (or column) of $\ell+1$ particles of type 
$\ta$ at dual distance $1$ of each other. A pillar is a particle of type 
$\ta$ at dual distance $1$ of the beam not located at one of the two ends
of the beam. The particle in the beam sitting next to the pillar divides the 
beam into two sections. The lengths of these two sections are $\geq 0$ and 
sum up to $\ell$. The support of a pillared beam is the union of all the 
tile supports. The support consists of three rows (or columns) of sites -- 
an upper, middle and lower row (or column) -- which are referred to as roof, 
center and basement
({\rm see Fig.~\ref{fig:beam-pillar}}).
\end{definition}

\begin{figure}[htbp]
\centering
\includegraphics[height=0.07\textwidth]{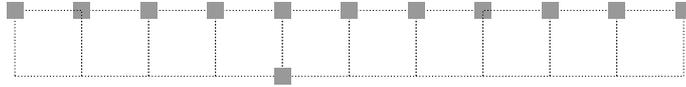}
\caption{A south-pillared horizontal beam of length $10$ with a west-section 
of length 4 and an east-section of length 6.}
\label{fig:beam-pillar}
\end{figure}

Note that a beam can have more than one pillar. Lemma~\ref{lemma btiles are optimal tiles} 
implies the following.

\begin{corollary}
\label{corollary optimality of btiled pillared beams}
Let $\eta$ be a configuration containing a pillared beam $\tilde{b}$ such that
$\supp(\tilde{b})$ is not $\btiled$.
Then the configuration $\eta\prm$ obtained 
from $\eta$ by $\btiling$ $\supp(\tilde{b})$ satisfies $H (\eta\prm) \leq H (\eta)$.
\end{corollary}


\subsection{Six energy reduction mechanisms}
\label{sec enegredmech}


\subsubsection{Moving unit holes inside $\btiled$ clusters}
\label{sec motion of particles inside clusters}

In this section we show how a unit hole can move inside a $\btiled$ cluster. 
In particular, we show that such motion is possible within an energy barrier $6U$
by changing the configuration only locally.

\begin{definition}
A set of sites $S$ inside $\Lambda$ obtained from a $4 \times 4$ square after 
removing the four corner sites is called a slot.
\end{definition}

\noindent
Given a slot $S$, we assign a label to each of the $12$ sites in $S$ as in
Fig.~\ref{fig:slot lemma} (a): first clockwise in the center of $S$ and then
clockwise on the boundary of $S$. We call the pairs $(S_{1},S_{3})$ and 
$(S_{2},S_{4})$ slot-conjugate sites.

\begin{lemma}
\label{lemma on motion of particles inside a cluster}
Let $S$ be a slot, and let $\eta_{0}$ be any configuration such that all particles 
in $S$ have the same parity. W.l.o.g.\ this parity may be taken to be even, so that 
$\eta(S_{1})=0$ and $\eta(S_{3})=\tb$. Let $\eta_{1}$ be the configuration obtained 
from $\eta$ by interchanging the states of $S_{1}$ and $S_{3}$. Then $H(\eta_{0})
= H(\eta_{1})$, and there exists a path $\omega\colon\,\eta_{0}\to\eta_{1}$ that 
never exceeds the energy level $H(\eta_{0})+ 6U$.
\end{lemma}

\begin{figure}[htbp]
\centering
\subfigure[]
{\includegraphics[width=0.16\textwidth]{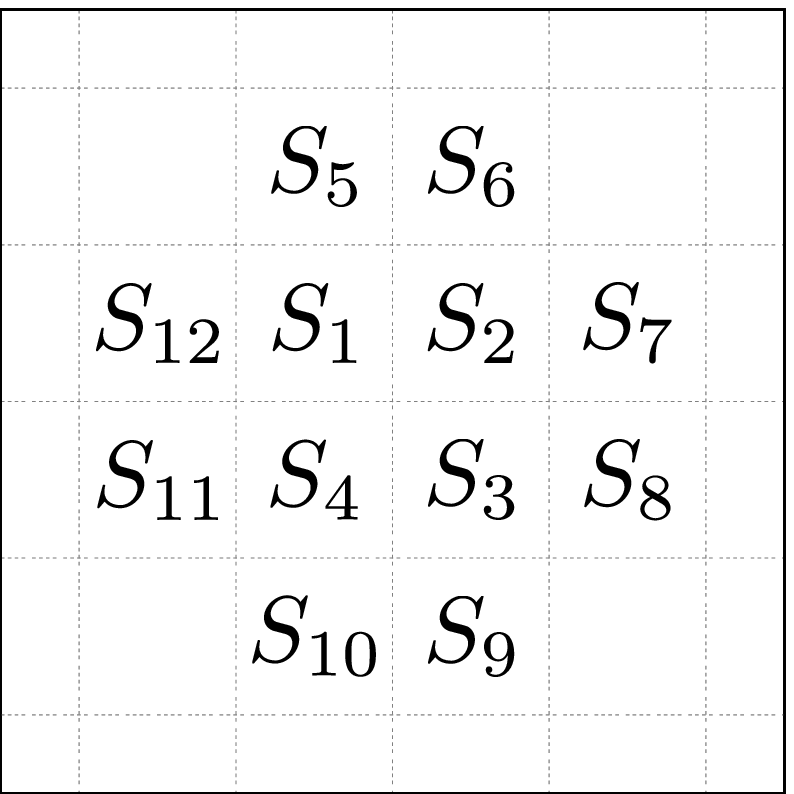}}
\quad
\subfigure[]
{\includegraphics[width=0.16\textwidth]{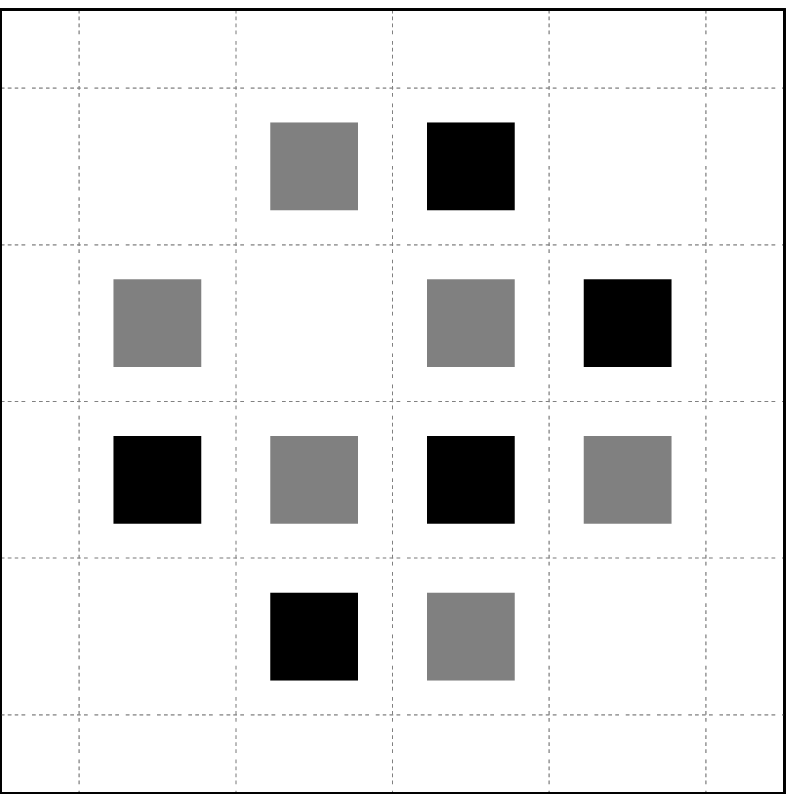}}
\quad
\subfigure[]
{\includegraphics[width=0.16\textwidth]{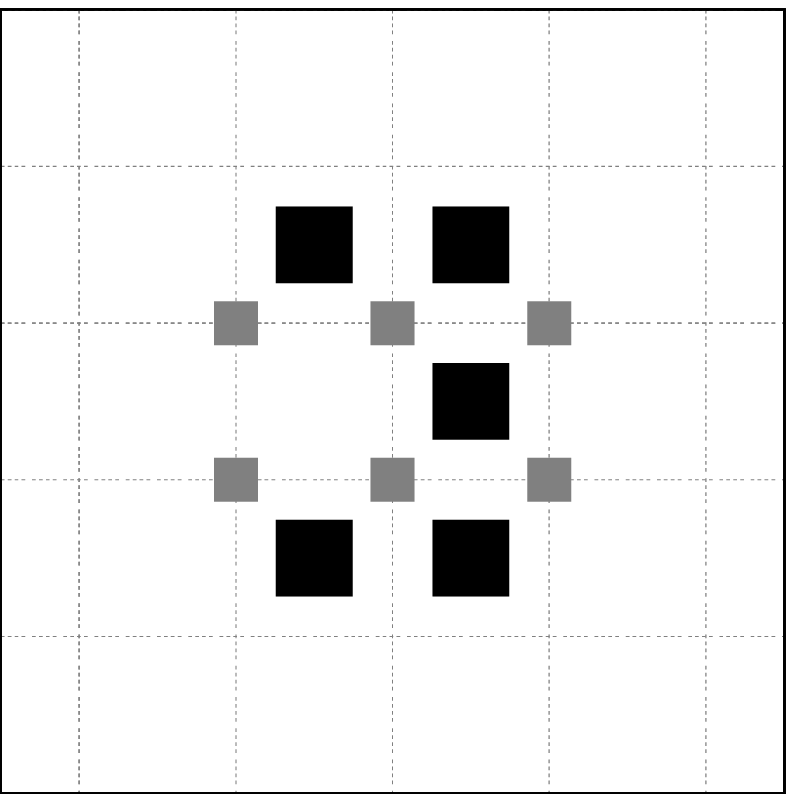}}
\quad
\subfigure[]
{\includegraphics[width=0.16\textwidth]{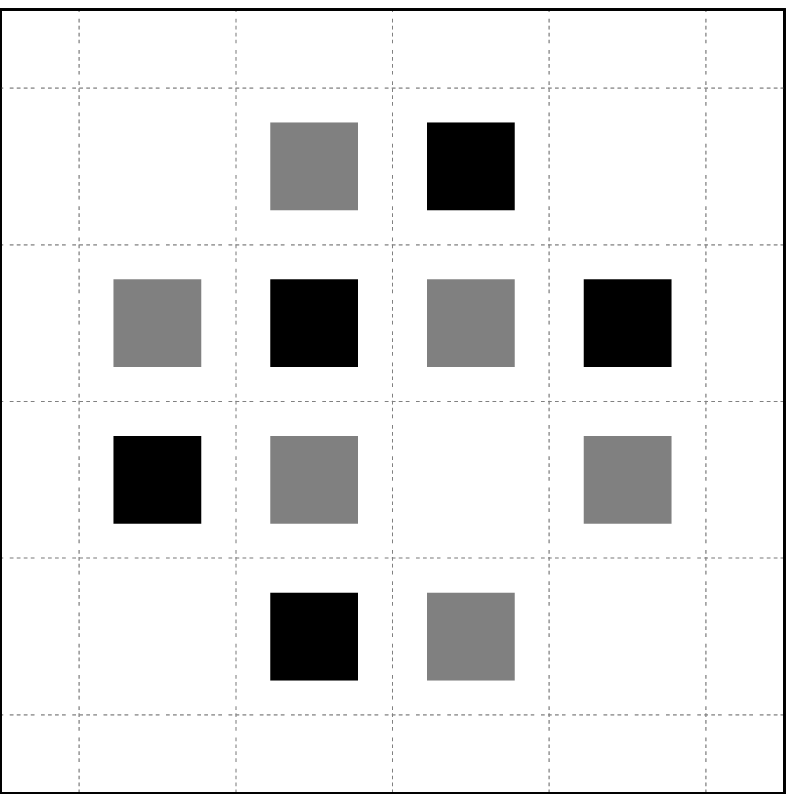}}
\quad
\subfigure[]
{\includegraphics[width=0.16\textwidth]{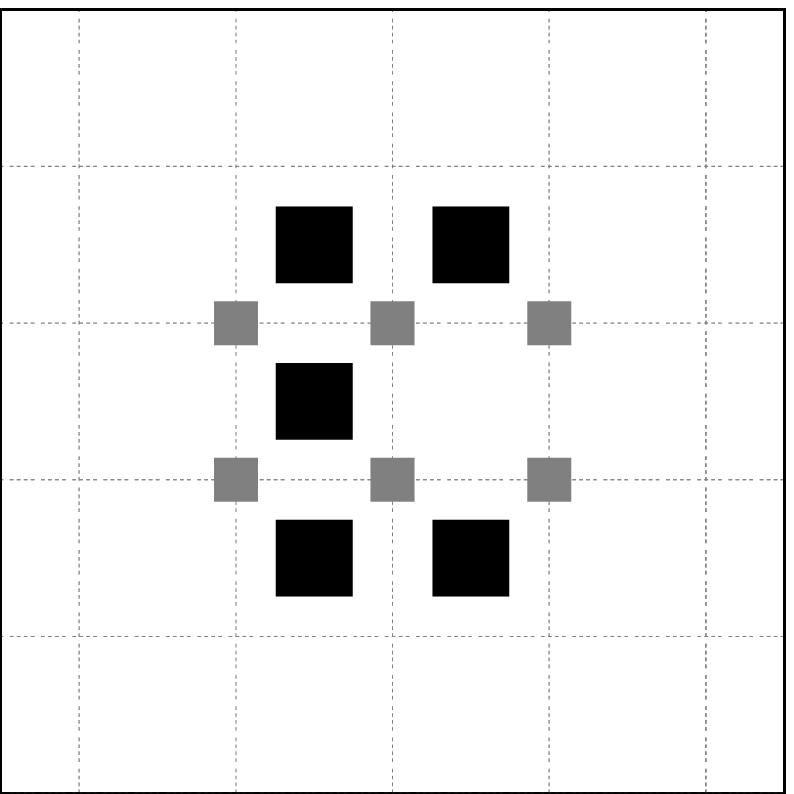}}
\caption{ 
(a) labelling of the sites in the slot (standard representation);
(b) example of $\eta_0$ in the slot (standard representation);
(c) example of $\eta_0$ in the slot (dual representation).
(d) $\eta_1$ in the slot (standard representation);
(e) of $\eta_1$ in the slot (dual representation).
}
\label{fig:slot lemma}
\end{figure}

\begin{proof}
W.l.o.g.\ we take $\eta_0$ as in Fig.~\ref{fig:slot lemma}(b--c). Let $a \to b$ 
denote the motion of a particle from site $a$ to site $b$. For the path $\omega$ 
we choose the following sequence of moves: $S_{4} \to S_{1}$; $S_{3} \to S_{4}$; 
$S_{2} \to S_{3}$; $S_{1} \to S_{2}$; $S_{4} \to S_{1}$; $S_{3} \to S_{4}$. The
first three moves and the second three moves each are a rotation by $\frac{\pi}{2}$ 
of the subconfiguration at the sites $S_{1}, S_{2}, S_{3}, S_{4}$. Note that all 
configurations in $\omega$ have the same number of particles of each type and 
hence the changes in energy only depend on the change in the number of active 
bonds. Let $M_{RF}$ be the loss of the number of active bonds between the rotating 
particles and the fixed particles, and $M_{R}$ the loss of the number of active 
bonds between the rotating particles. We must show that $M_{RF} + M_{R} \leq 6$ 
during the six moves. To that end, we first observe that $M_{RF} \leq 6$, since 
the total number of active bonds between the rotating particles and the fixed 
particles is at most $6$ (see Fig.~\ref{fig:slot lemma}(b)), and that $M_{RF}=6$ 
only after the first three moves are completed, i.e., when the configuration is 
such that all the rotating particles have a different parity with respect to the 
parity they had in configuration $\eta_{0}$ (recall that particles with different 
parity cannot share a bond). Next we observe that, by the choice of $\omega$, the 
value of $M_{R}$ can only be $0$ or $1$, and that $M_R=0$ after the first three 
moves are completed.
\end{proof}

Lemma~\ref{lemma on motion of particles inside a cluster} implies the following.

\begin{proposition}
\label{prop1}
Let $\eta$ be a $\btiled$ configuration with a unit hole. Then the configuration $\eta\prm$
obtained from $\eta$ by moving the unit hole elsewhere satisfies 
$H(\eta\prm)=H(\eta)$ and $\comlev(\eta,\eta\prm) \leq H(\eta) + 6U$.
\end{proposition}

\noindent
A possible $6U$-path for a unit hole inside a $\btiled$ cluster is given in 
Fig.~\ref{fig:picture motion of unit hole}. This path is obtained through 
an iteration of local moves as explained in Fig.~\ref{fig:slot lemma}.

\begin{figure}[htbp]
\centering
\includegraphics[width=0.2\textwidth]{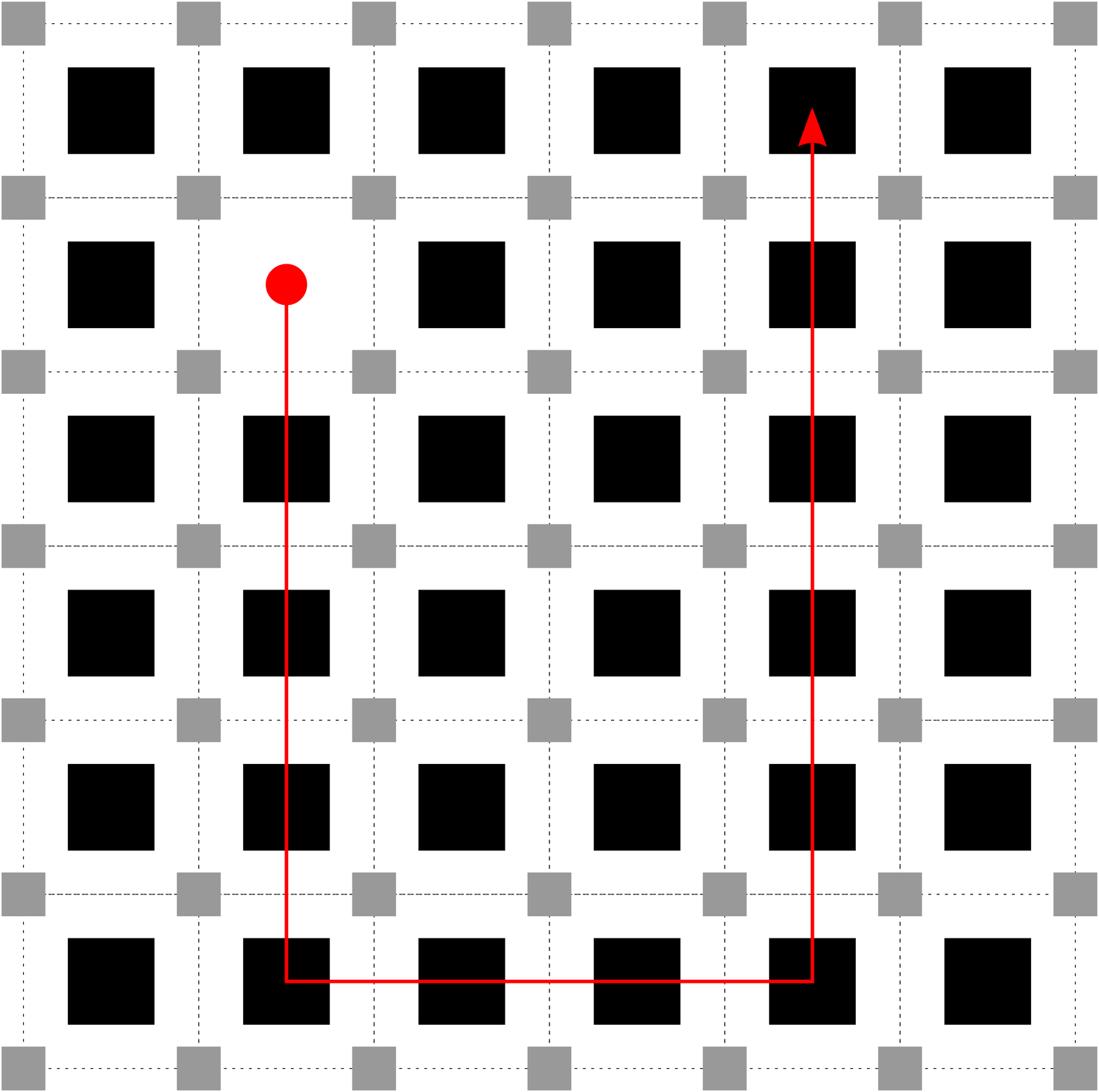}
\caption{Motion of a unit hole inside a $\btiled$ cluster.}
\label{fig:picture motion of unit hole}
\end{figure}


\subsubsection{Adding and removing $\abbars$ from lattice-connecting rectangles}
\label{sec growth of clusters}

\begin{lemma}\label{lemma barrier adding and removing bars}
Let $\eta$ be a configuration consisting of a single $\btiled$ lattice-connecting rectangle.
Then the configuration $\eta\prm$ obtained from $\eta$ by, respectively,
\begin{enumerate} 
	\item adding a $\abbar$ of length $\ell \ge \ell\starred$,
	\item adding a $\abbar$ of length $\ell < \ell\starred$,
	\item removing a $\abbar$ of length $\ell \ge \ell\starred$, 
	\item removing a $\abbar$ of length $\ell < \ell\starred$, 
\end{enumerate}
satisfies, respectively,
\begin{enumerate} 
	\item $H(\eta\prm) < H(\eta)$ and $\comlev(\eta, \eta\prm) \le H(\eta) + 2\Da + 2\Db - 4 U$,
	\item $H(\eta\prm) > H(\eta)$ and $\comlev(\eta, \eta\prm) \le H(\eta) + 2\Da + 2\Db - 4 U$,
	\item $H(\eta\prm) > H(\eta)$ and $\comlev(\eta, \eta\prm) \le H(\eta) + (\ell - 2)\epsi + 4U - \Da$,
	\item $H(\eta\prm) < H(\eta)$ and $\comlev(\eta, \eta\prm) \le H(\eta) + (\ell - 2)\epsi + 4U - \Da$.
\end{enumerate}

\end{lemma}

\begin{proof} Recall the computations in \
Sections~\ref{sec standard configurations are optimal btiled configurations} and \ref{sec id sub}.

\medskip\noindent
{\bf Adding a $\abbar$.}
Adding a $\abbar$ of length $\ell$ on a lattice-connecting side of a $\btiled$ rectangle 
(i.e., a side such that all the particles of type $\ta$ on that side are lattice-connecting) 
can be done in two steps: (i) initiate the $\abbar$ by adding a $\btiled$ protuberance
(see Fig.~\ref{fig-add_protuberance}); (ii) complete the $\abbar$ by adding a $\btile$ 
(in a ``corner'') $\ell-1$ times (see Fig.~\ref{fig-add_tile}). This can be achieved within 
energy barrier $\D H = 2\Da + 2\Db - 4 U$ by following the same moves as the reference path 
$\omega\starred$ described in Section~\ref{sec id sub}. The energy difference due to the extra 
$\abbar$ of length $\ell$ is $\D H(\ell) = \Da - \ell \epsi$, which changes sign at 
$\ell=\ell\starred$.

\begin{figure}[htbp]
\begin{centering}
{\includegraphics[height=0.25\textheight]{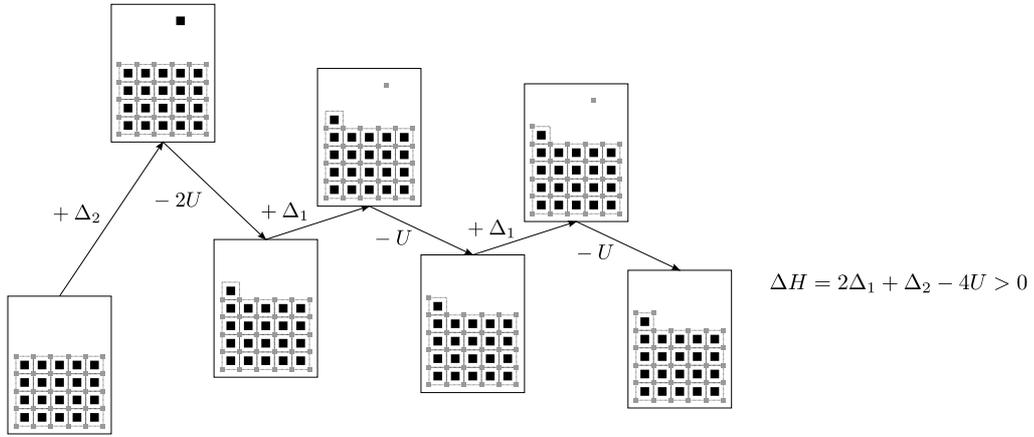}}
\par\end{centering}
\caption{A $\btiled$ protuberance is added to a side of a dual rectangle within 
energy barrier $\Db$.}
\label{fig-add_protuberance}
\end{figure}

\begin{figure}[htbp]
\begin{centering}
{\includegraphics[height=0.25\textheight]{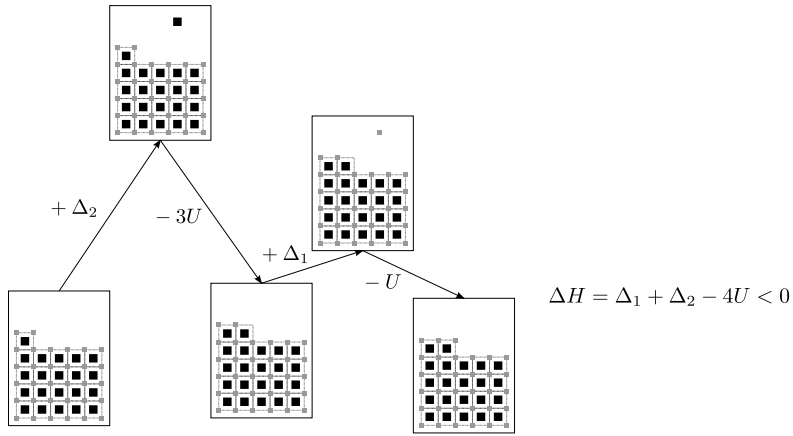}}
\par\end{centering}
\caption{A $\btile$ is added in a corner between $\btiles$ within a energy barrier $\Db$.}
\label{fig-add_tile}
\end{figure}

\medskip\noindent
{\bf Removing a $\abbar$.}
Removing a $\abbar$ of length $\ell$ from a lattice-connecting rectangle can be done
by following the reverse of the path used to add a $\abbar$: (i) remove $\ell-1$ 
times a $\btile$ from a bar; (ii) remove the last $\btiled$ protuberance. This can be 
achieved within energy barrier $\D H(\ell)=(\ell - 2)\epsi + 4U - \Da$. If the cluster consists of 
one $\abbar$ only, then the path just described leaves $\ell+1$ free particles of type 
$\ta$ inside $\Lambda$, which can be removed (free of energy cost) afterwards.
\end{proof}

We use Lemma~\ref{lemma barrier adding and removing bars}
to build  a \emph{northern rectangle} on top of a $\abbar$ as follows.

\begin{definition}
\label{def rectangles touching north side}
Let $b$ denote the vertical coordinate of the sites lying on the north-side of 
$\partial^-\Lambdaminus$. For a given $\btiled$ rectangle $r$ in $\Lambdaminus$,
let $b_{r}$ denote the vertical coordinate of the northern-most particles of 
type $\ta$in $r$. Then $r$ is said to be touching the north-side of 
$\partial^-\Lambdaminus$ if $b_{r} = b$ or $b_{r} = b - \frac{1}{2}$.
\end{definition}

\noindent
In words, a $\btiled$ rectangle is said to be touching the north-side of 
$\partial^-\Lambdaminus$ if it is not possible to add a $\abbar$ on the north-side 
within $\Lambdaminus$. Rectangles touching the south-, east- or west-side of 
$\Lambdaminus$ are defined similarly.

Let $\bar{b}$ be a horizontal $\abbar$ of length $\ell$, i.e., a $\btiled$ 
$\ell\times 1$ rectangle. Suppose that all sites above $\bar{b}$ are vacant. Then it is 
possible to successively add horizontal $\abbars$, say $m$ in total, on top of 
$\bar{b}$ until the north side of the rectangle grown in this way touches the 
north-side of $\Lambdaminus$. The $\btiled$ rectangle with $m+1$ rows and $\ell$ 
columns such that $\bar{b}$ is its lower-most horizontal $\abbar$ is denoted 
by $\nrec{\bar{b}}$ and is called the northern rectangle of $\bar{b}$. 

Lemma~\ref{lemma barrier adding and removing bars} implies the following.

\begin{proposition}
\label{prop2}
Let $\eta$ be a configuration containing a horizontal $\abbar$ $\bar{b}$ of length $\ell\geq
\ell\starred$. Then the configuration $\eta\prm$ obtained from $\eta$ by building 
$\nrec{\bar{b}}$ satisfies $H (\eta\prm) < H(\eta)$ and $\comlev(\eta,\eta\prm)
\leq H(\eta) + 2\Da + 2\Db - 4U$.
\end{proposition}


\subsubsection{Changing bridges into $\abbars$}
\label{sec main procedures}

\begin{definition}
A (south-)bridge $b$ consists of a beam $\tilde{b}$ and two (south-)pillars at 
the outer-most sites of the (south-)basement of $\tilde{b}$. The (south-)support 
of $b$ coincides with the (south-)support of $\tilde{b}$. If each of the central 
sites of the tiles of the (south-)support of the bridge is occupied by a particle 
of type $\tb$, then the bridge is said to be stable (see {\rm Fig.~\ref{fig:bridge}}). 
\end{definition}

\noindent
Clearly, a $\abbar$ is a stable bridge. North-, east- and west-bridges are 
defined in a similar way.

\begin{figure}[htbp]
\centering
{\includegraphics[height=0.1\textwidth]{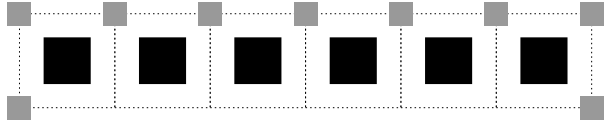}}
\caption{A stable bridge of length $6$.}
\label{fig:bridge}
\end{figure}

Given a bridge $b$, let $\bar{b}$ denote the $\abbar$ obtained by $\btiling$ $b$. 
Lemma~\ref{lemma btiles are optimal tiles} implies the following.

\begin{lemma}
\label{lemma btiled bridges are optimal}
Let $\eta$ be a configuration containing a bridge $b$ whose support is not 
$\btiled$. Then the configuration $\eta\prm$ obtained from $\eta$ by changing $b$ to $\bar{b}$ 
satisfies $H (\eta\prm) < H (\eta)$. 
\end{lemma}

Lemma~\ref{lemma btiled bridges are optimal} leads us to the following.

\begin{proposition}
\label{prop3}
\label{lemma btiled bridges can be achieved}
Let $\eta$ be a configuration containing a (south-)bridge $b$ whose (south-)support 
is not $\btiled$ such that the particles of its beam are lattice-connecting. Then 
the configuration $\eta\prm$ obtained from $\eta$ by $\btiling$ $\supp(b)$ satisfies
$H(\eta')<H(\eta)$ and $\comlev(\eta,\eta\prm)\leq H(\eta)+4U+\Da$. 
\end{proposition}

\begin{proof}
Let the (south-)bridge $b$ have length $\ell$. Label the $\ell + 1$ sites of its 
(south-)basement as $s_{0},s_{1},\ldots,s_{\ell}$, from the left to the right. In order 
to show that $\supp(b)$ can be $\btiled$ within energy barrier $4U+\Da$, it is enough 
to show that within the same energy barrier a particle of type $\ta$ can be brought 
to a site of the basement of $b$ (from the left) that is empty or is occupied by 
a particle of type $\tb$. W.l.o.g.\ $s_{1}$ may be assumed to be such a site. The 
configuration thus obtained has an energy that is at most the energy of the original 
configuration (see Lemma~\ref{lemma btiles are optimal tiles}). The claim follows by
noting that the particles of type $\ta$ at the extremal sites $s_{1}$ and $s_{\ell}$ 
are the two pillars of a (south-)bridge of length $\ell - 1$ whose basement consists 
of the sites $s_{1},s_{2},\ldots,s_{\ell}$.
	
It remains to show how a particle of type $\ta$ can be brought to site $s_{1}$.
Label the site north-west of $s_{1}$ by $v_{1}$ , and the site north-east of $v_{1}$
by as $v_{2}$. Two cases need to be distinguished:

\medskip\noindent
(1) If $\eta(s_{1}) = 0$, then, by the same argument as in the proof of Lemma~\ref{lemma
on motion of particles inside a cluster}, it is easy to show that the particle of type 
$\ta$ at $v_{2}$ can be moved to $s_{1}$ (to obtain a configuration $\bar\eta$ with 
$H(\bar\eta) \leq H(\eta)$) without exceeding energy level $H(\eta) + 4U$. The 
configuration $\eta\prm$ is reached within an energy barrier $\Da$ by bringing a particle 
of type $\ta$ inside $\Lambda$ and moving it to $v_{2}$.

\medskip\noindent
(2) If $\eta(s_{1}) = 2$, then consider the following path. First detach ($\D H =  2U$) 
and remove ($\D H =  -\Da$) the particle of type $\ta$ at $v_{2}$, and afterwards detach 
($\D H =  2U$) and remove ($\D H =  -\Db$) the particle of type $\tb$ at $v_{3}$. Next, 
move the particle of type $\tb$ at site $s_{1}$ to site $v_{1}$ ($\D H \leq 0$; 
this particle has at most $2$ active bonds when it sits at $s_{1}$), and finally bring 
a particle of type $\ta$ ($\D H =  \Da$) to site $v_{2}$ ($\D H =  -2U$). Call this 
configuration $\bar\eta$. Note that $H(\bar\eta) < H(\eta)$, since effectively 
a particle of type $\tb$ with at most two active bonds has been removed, and $\comlev
(\eta,\eta\prm) = H(\eta) + 4U + \Da$. Finally, observe that $\eta\prm$ is the same 
configuration as $\eta$ in Case (1). 
\end{proof}


\subsubsection{Maximally expanding $\btiled$ rectangles}
\label{maximal expansion btiled rectangle}

The mechanism presented in this section, which is called \emph{north maximal 
expansion} of a $\btiled$ rectangle, is such that it can be applied to a 
$\btiled$ rectangle whose north-side is lattice-connecting (even though this 
condition is not restrictive). South, east and west maximal expansion of a 
$\btiled$ cluster are analogous.

\begin{definition}
The north maximal expansion comes in two phases: a growing phase and a 
smoothing phase.\\
(i) The growing phase consists of the following three 
steps repeated cyclically:
\begin{enumerate}
\item 
If the particles of type $\ta$ on the south-side of the rectangle, either at 
the beginning or obtained after step~\ref{alg max expansion step west}, constitute a 
south-pillared beam $\tilde{b}_{s}$, then change $\supp(\tilde{b}_{s})$ into 
a $\abbar$.
\label{alg max expansion step south}
\item 
If the particles of type $\ta$ on the east-side of the rectangle, obtained after
step~\ref{alg max expansion step south}, constitute an east-pillared beam $\tilde{b}_{e}$, 
then change $\supp(\tilde{b}_{e})$ into a $\abbar$.
\label{alg max expansion step east}
\item 
If the particles of type $\ta$ on the west	-side of the rectangle, obtained after
step~\ref{alg max expansion step east}, constitute a west-pillared beam $\tilde{b}_{w}$,
then change $\supp(\tilde{b}_{w})$ into a $\abbar$.
\label{alg max expansion step west}
\end{enumerate} 
The growing phase ends after three consecutive steps leave the configuration unchanged.\\
(ii) The smoothing phase consists of removing all the particles of type $\tb$ that are 
adjacent to the ones on the sides of the rectangle that is built during the growing phase. 
Note that these particles have at most two active bonds (otherwise it would be possible 
to identify another pillared beam), and therefore removal of these particles lowers the 
energy.
\end{definition}

\noindent
The outcome of the north maximal expansion (see Fig.~\ref{fig:expansion}) 
of a $\btiled$ rectangle is again a $\btiled$
rectangle, containing the old rectangle and such that the northern-most $\abbar$ of the 
new rectangle has the same vertical coordinate.

\begin{figure}[htbp]
\centering
\subfigure[\label{maxexp1}]
{\includegraphics[width=0.25\textwidth]{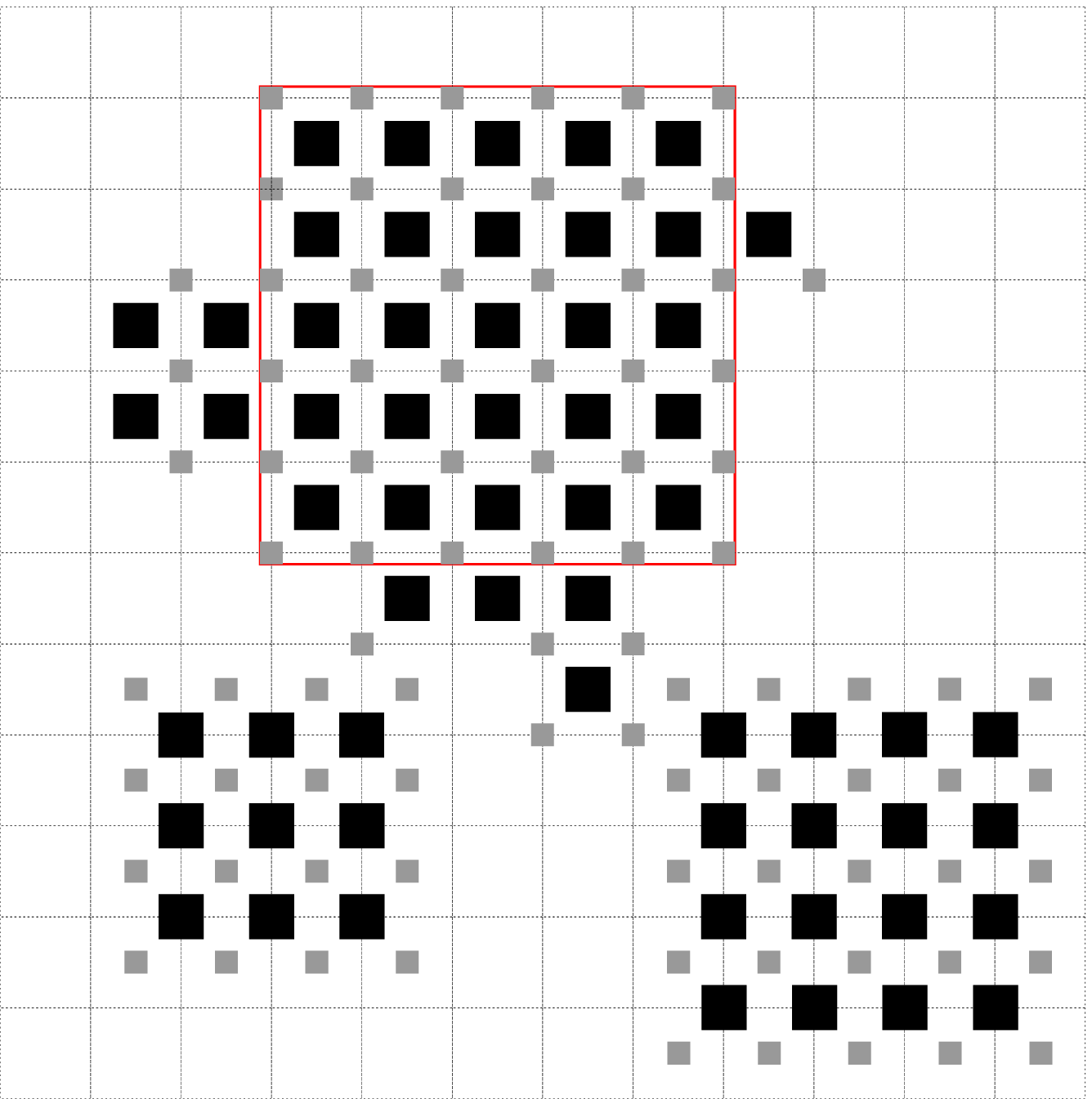}}
\qquad
\subfigure[\label{maxexp2}]
{\includegraphics[width=0.25\textwidth]{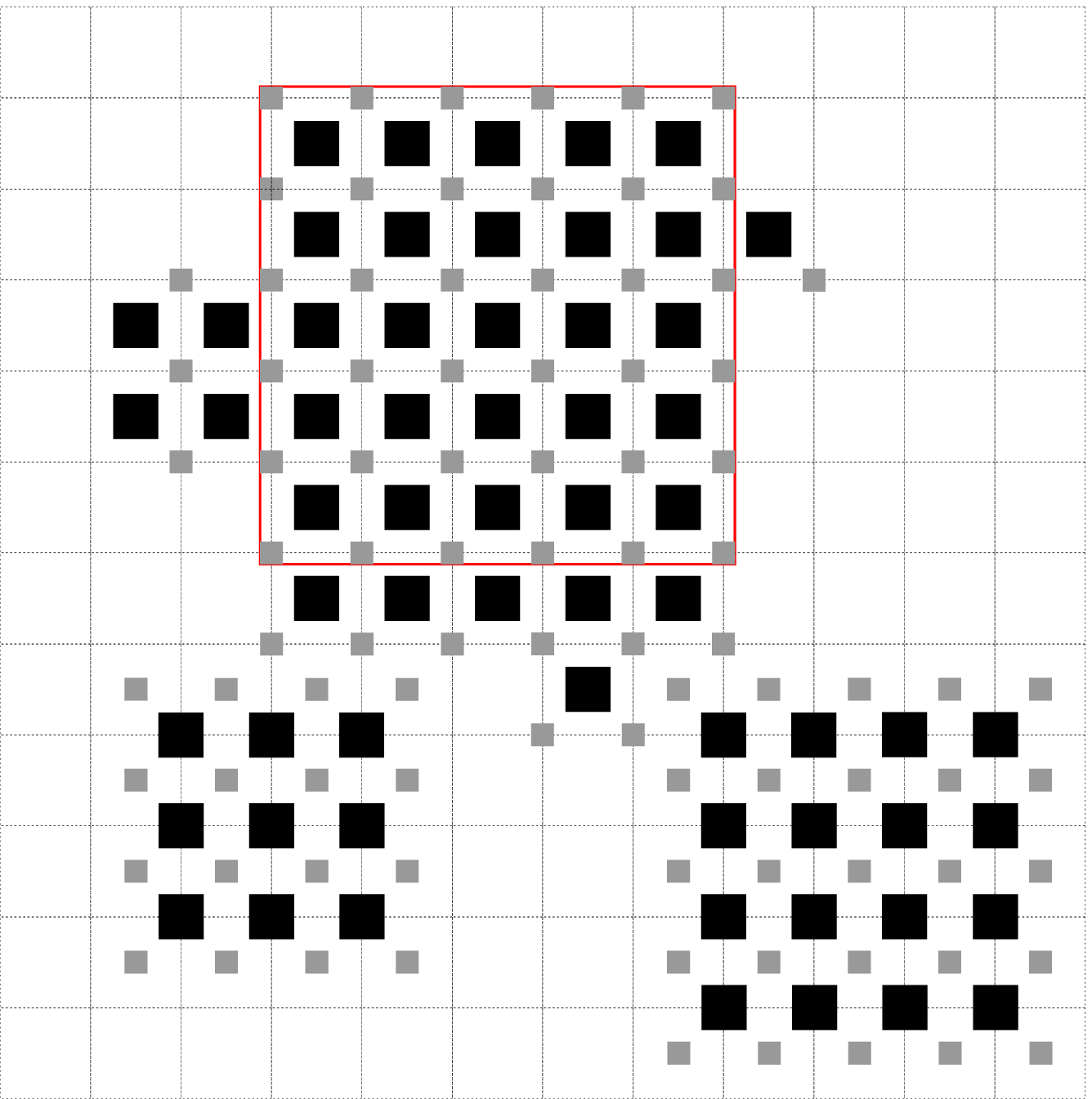}}
\qquad
\subfigure[\label{maxexp3}]
{\includegraphics[width=0.25\textwidth]{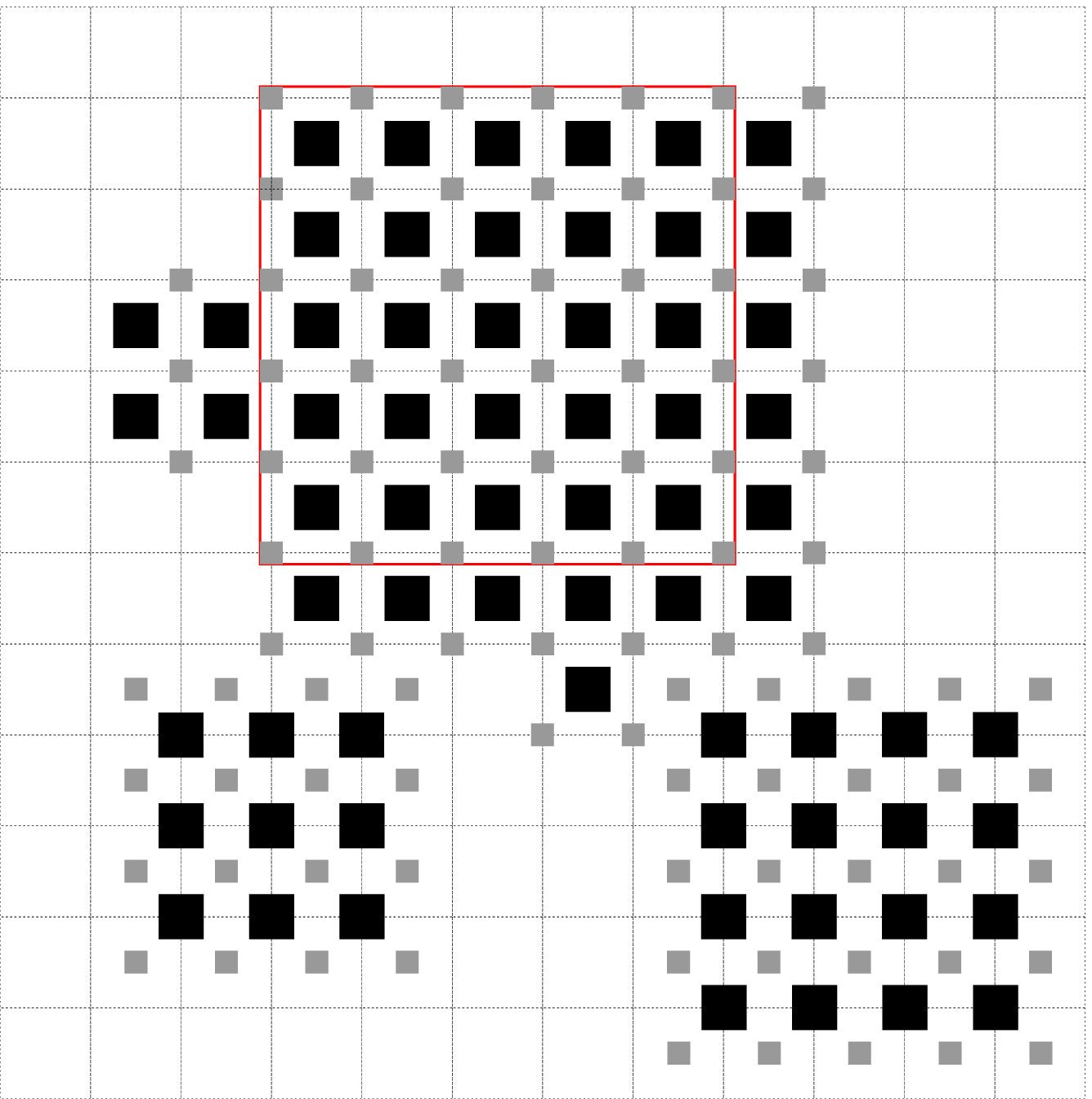}}
\\
\subfigure[\label{maxexp4}]
{\includegraphics[width=0.25\textwidth]{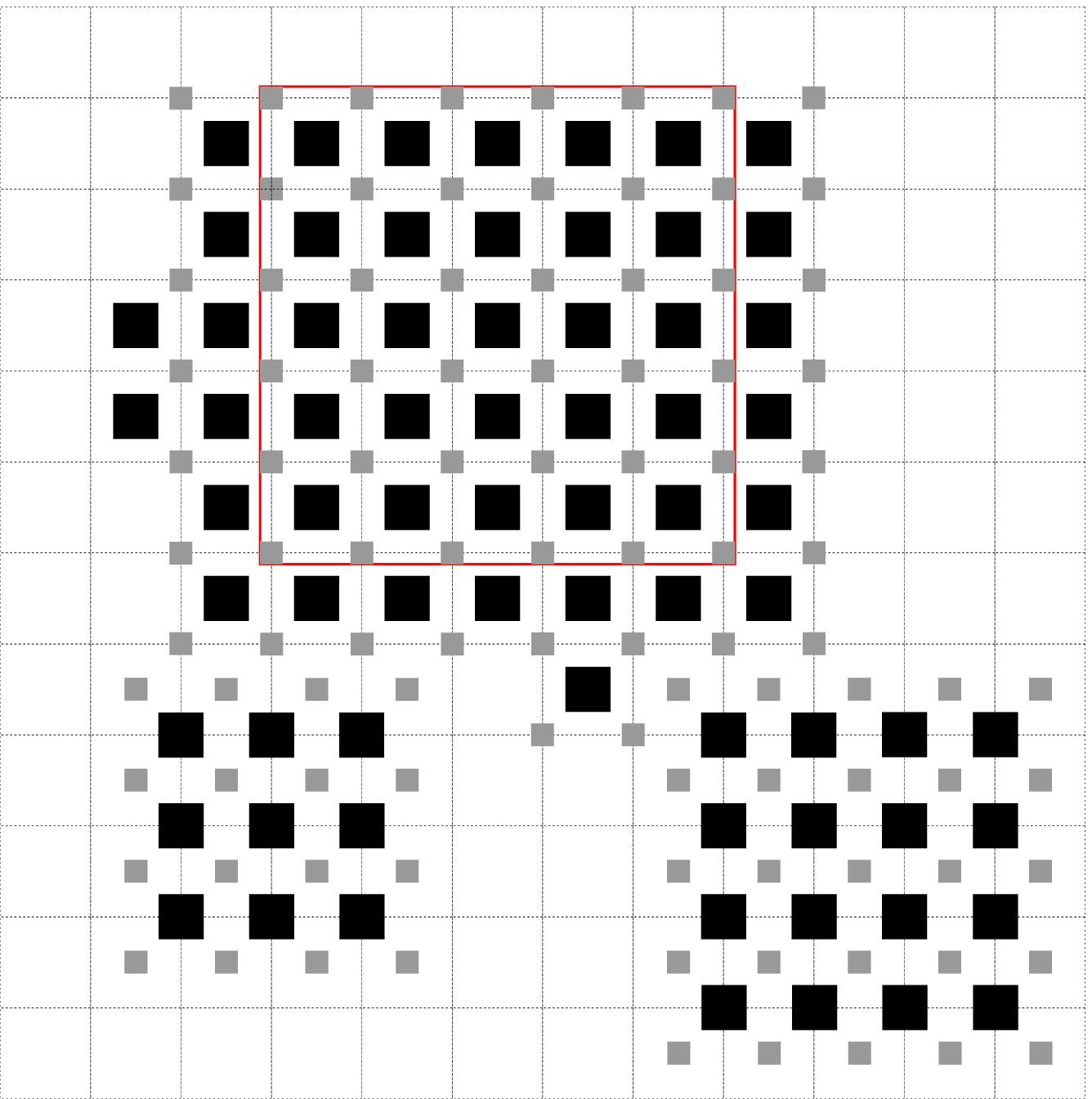}}
\qquad
\subfigure[\label{maxexp5}]
{\includegraphics[width=0.25\textwidth]{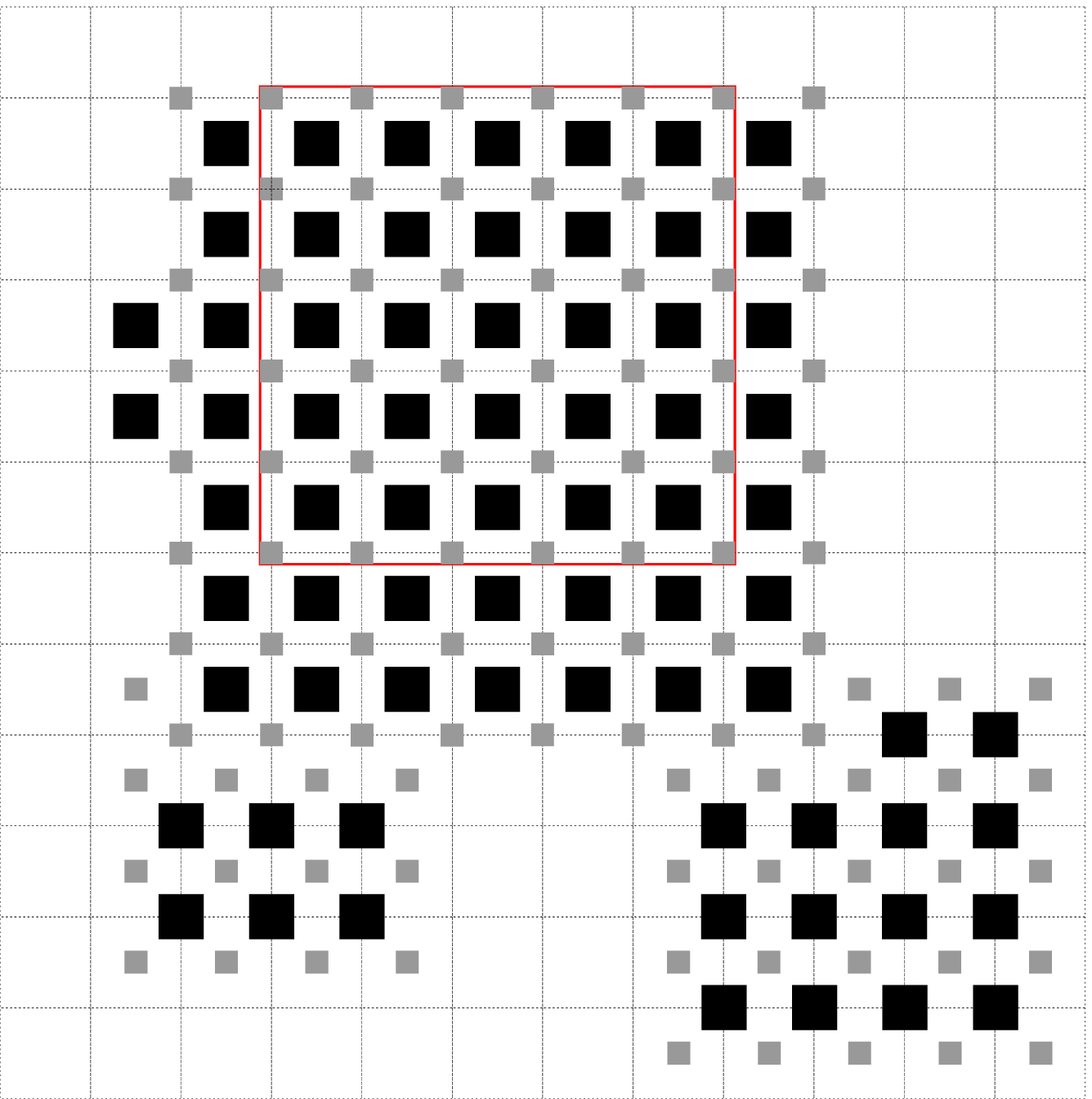}}
\qquad
\subfigure[\label{maxexp6}]
{\includegraphics[width=0.25\textwidth]{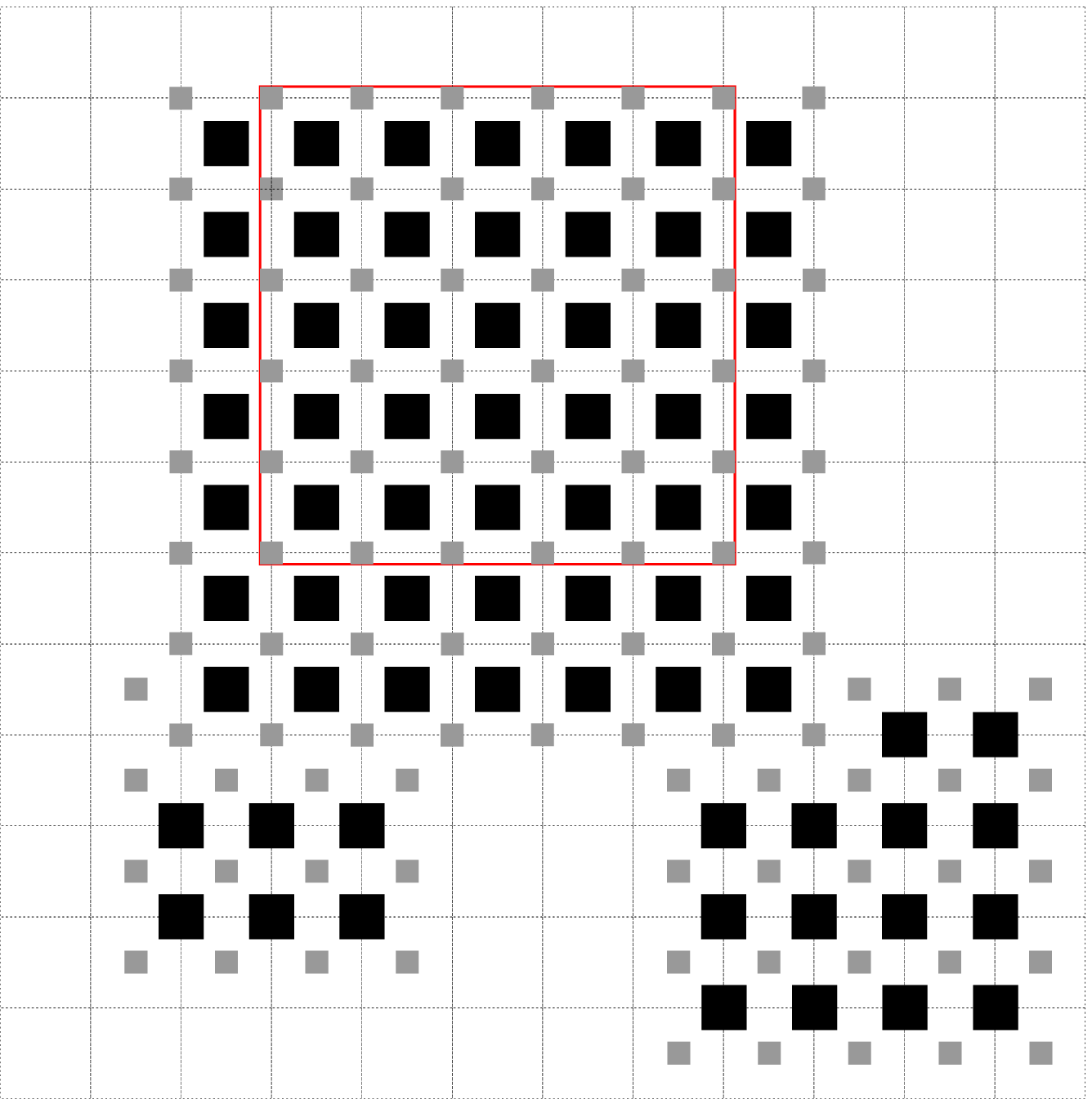}}
\caption{Example of north maximal expansion of a $\btiled$ rectangle. 
The outcome of the steps of the
growing phase are represented in pictures~(b--e),
while the outcome of the smoothing phase is represented in picture (f).}
\label{fig:expansion}
\end{figure}

Given a $\btiled$ rectangle $r$, let $\nmexpansion{r}$ denote the north maximal expansion 
of $r$. Corollary~\ref{corollary optimality of btiled pillared beams} implies the
following.

\begin{lemma}
\label{lemma maximal expansion of a rectangle lowers energy}
Let $\eta$ be a configuration containing a $\btiled$ rectangle. Then the configuration $\eta\prm$ 
obtained from $\eta$ via (north) maximal expansion of this $\btiled$ rectangle satisfies 
then $H (\eta\prm) \le H (\eta)$.
\end{lemma}

Lemma~\ref{lemma maximal expansion of a rectangle lowers energy} leads us to the following.

\begin{proposition}
\label{prop4}
\label{lemma maximal expansion of a rectangle can be achieved}
Let $\eta$ be a configuration containing a $\btiled$ rectangle $r$ whose north-side 
is lattice-connecting. Then the configuration $\eta\prm$ obtained from $\eta$ after 
replacing $r$ by $\nmexpansion{r}$ satisfies $H(\eta\prm)\leq H(\eta)$ and 
$\comlev(\eta,\eta\prm) \leq  H(\eta) + 10U - \Da$.
\end{proposition}

\begin{proof}
If $\nmexpansion{r} = r$, then there is nothing to prove. Therefore suppose that $r$ is 
such that one its sides is a pillared beam. W.l.o.g.\ we may assume that the south-side 
of $r$ is a beam $\tilde{b}$ with a south-pillar. We must show that the south-support 
of $\tilde{b}$ can be turned into a $\abbar$ within energy barrier $10U - \Da$.	
	
Since $\supp(\tilde{b})$ is not a $\abbar$, a pillar can be chosen in such a way that 
at least one of the $\btiles$ of the support the pillar belongs to (i.e., the first 
tile of each section of the support, counting from the pillar) is not a $\btile$. 
W.l.o.g.\ we let this tile be the first tile of the right-section and call it $t$. 
Let $v$ denote the tile adjacent to the right site of $v$. In the following, the term 
\emph{superficial} refers to tiles that are in the top tile-bar of the rectangle. 
In analogy with the proof of Lemma~\ref{lemma btiles are optimal tiles}, several cases 
need to be considered (we stick to the order in Fig.~\ref{fig-possible_quasitiles}). 

\begin{enumerate}[(i)]	
\item	
$(\eta(t_{c}),\eta(t_{e})) = (0,0)$.
A particle of type $\tb$ has to be brought to site $t_{e}$ and a particle of type 
$\ta$ to site $t_{c}$. First bring a particle of type $\tb$ to site $t_{e}$, to 
reach a configuration $\hat{\eta}$, and then proceed as in Case~(\ref{rectangle
expansion - type one to bottom right}). As we will see in Case~(\ref{rectangle expansion
- type one to bottom right}), since $H(\hat{\eta}) = H(\eta) -3U + \Db$, the 
second part of the path can be completed without exceeding energy level $H(\eta)
+ 6U + \Db $. To reach configuration $\hat{\eta}$, move the particle of type $\tb$ 
of the $\btile$ above $t$ to site $t_{e}$ to reach a configuration called $\eta\prm$. 
This can be done without exceeding energy level $H(\eta) + 6U$. Note that $H(\eta\prm)
= H(\eta) + U$. The unit hole that has been created at the central site of the tile 
above $t$ has to be filled. This can be done (see Lemma~\ref{lemma on motion of particles
inside a cluster}) by first moving the unit hole until it becomes superficial (configuration 
$\tilde{\eta}$ with energy $H(\tilde{\eta}) = H(\eta\prm)$) without exceeding energy level 
$H(\eta\prm) + 6U$, and then filling this unit hole with a particle of type $\tb$ within 
energy level $H(\eta\prm) + U - \Da + \Db = H(\eta) + 2U - \Da + \Db$. Thus, $\eta\prm$
can be reached without exceeding energy barrier $6U+\Db$.
\label{rectangle expansion - two particles to bring}

\item
$(\eta(t_{c}),\eta(t_{e})) = (0,2)$.
A particle of type $\ta$ has to be brought to site $t_{c}$. Depending on the state 
of site $v_{e}$, there are three cases.

\begin{enumerate}
\item 
Site $v_{e}$ is occupied by a particle of type $2$. Move the particle of type $\ta$ 
at site $t_{b}$ to site $t_{c}$, to reach a configuration $\eta\prm$ with energy
$H(\eta\prm)\le H(\eta) + 2U$ within an energy barrier of $6U$. The vacancy at site 
$t_{b}$ can be moved (again by Lemma~\ref{lemma on motion of particles inside a cluster})
to the north-side of the rectangle within energy barrier $6U$, to reach a configuration 
$\hat{\eta}$ with $H(\hat{\eta}) \le H(\eta)$, and then filled with an extra particle of 
type $\ta$. Thus, $\eta\prm$
can be reached without exceeding energy level $H(\eta) + 8U$.
\item 
Site $v_{e}$ is empty. Move the particle of type $\ta$ at site $t_{b}$ to site $v_{e}$
($\D H \le 3U$), and then to site $t_{d}$ ($\D H = 0$). Call this configuration
$\eta\prm$, and note that $H(\eta\prm) \leq H(\eta) + 2U$. Arguing as above, we
see that the vacancy at site $t_{b}$ can be filled without exceeding the energy level
$H(\eta) + 9U$.
\item 
Site $v_{e}$ is occupied by a particle of type $\ta$. Observe that the particle of 
type $\ta$ at $t_{b}$ has $k \le 3$ active bonds and the particle of type $\tb$ at $v_{e}$ 
has $m \le 2$ active bonds. It is possible to move the particle at site $v_{e}$ to site 
$t_{c}$ ($\D H = (m-k)U$), and then the particle at site $t_{b}$ to site $v_{c}$ 
($\D H = (k-m)U$). The configuration $\eta\prm$, reached within energy barrier $(k - m)U$, 
has energy $H(\eta\prm) \le  H(\eta) + k U$. Again, the vacancy at site $t_{b}$
has to be filled with a particle of type $\ta$. This can be done without exceeding 
the energy level $H(\eta) + (6 + k)U \le H(\eta) + 9U$.
\end{enumerate}
\label{rectangle expansion - type one to bottom right}

\item 
$(\eta(t_{c}),\eta(t_{e})) = (2,0)$.
The particle of type $\tb$ at site $t_{c}$ is moved to site $t_{e}$ without increasing 
the energy. Then argue as in Case~(\ref{rectangle expansion - type one to bottom right}).

\item 
$(\eta(t_{c}),\eta(t_{e})) = (2,2)$.
The particle of type $\tb$ at site $t_{c}$ has to be replaced by a particle of type $\ta$.
Remove the particle of type $\tb$ at $t_{e}$. To do this, first create a superficial 
unit hole (which can be done within energy barrier $4U - \Da$ by creating a hole in a 
corner tile of the rectangle) and move this vacancy to site $t_{e}$. By 
Lemma~\ref{lemma on motion of particles inside a cluster}, this can be achieved without 
exceeding energy level $H(\eta_{0}) + 10U - \Db$. Then move the particle of type $\tb$ 
at site $t_{c}$ to site $t_{e}$ ($\D H \le 0$). Call $\eta\prm$ the configuration that is 
reached in this way. Note that $H(\eta\prm) \le H(\eta)- \Db + 3U$. To bring a particle 
of type $\ta$ to site $t_{c}$, argue as in Case~(\ref{rectangle expansion - type one to
bottom right}), to arrive at $H(\hat{\eta}) \leq H(\eta) + 12U- \Db$.
\label{rectangle expansion - two becomes one}
	
\item	
$(\eta(t_{c}),\eta(t_{e})) = (1,0)$.
A particle of type $\tb$ has to be brought to site $t_{e}$. Move the unit hole at $t_{e}$ 
to the top tile--bar of the rectangle. This does not change the energy of the configuration
and can be done within energy barrier $6U$ by Proposition~\ref{prop1}. The task reduces to 
filling a superficial unit hole on the surface of the cluster with a particle of type $\tb$. 
This can be achieved within energy barrier $U + \Db - \Da$. Therefore the maximal energy level 
reached in this case is $H(\eta) + 6U$.
\label{rectangle expansion - unit hole}

\item 
$(\eta(t_{c}),\eta(t_{e})) = (0,1)$.
Move the particle of type $\tb$ from site $t_{e}$ to site $t_{c}$. This move does not increase
the energy of the configuration. Then proceed as in Case~(\ref{rectangle expansion - unit hole}).

\item 
$(\eta(t_{c}),\eta(t_{e})) = (2,1)$. 
The occupation numbers of sites $t_{c}$ and $t_{e}$ have to be exchanged. To do this, 
first remove the particle of type $\ta$ at site $t_{b}$ to obtain a configuration
$\eta\prm$ with energy $H(\eta\prm) \le H(\eta) + 3U$ without exceeding the energy 
level $H(\eta)+ 10U - \Da$ (again use Lemma~\ref{lemma on motion of particles inside a cluster}).
Move the particle of type $\ta$ from $t_{e}$ to $t_{b}$ ($\D H<0$) and the particle 
of type $\tb$ from $t_{c}$ to $t_{e}$ ($\D H=0$). Call $\hat{\eta}$ the configuration 
that is reached in this way. Note that $H(\hat{\eta}) \le H(\eta) + U - \Da$. Proceed as 
in Case~(\ref{rectangle expansion - type one to bottom right}) to conclude within
energy barrier of $10U - \Da$.

\item 
$(\eta(t_{c}),\eta(t_{e})) = (1,1)$. 
The particle of type $\ta$ at site $t_{e}$ has to be replaced by a particle of type $\tb$.
This can be done as follows. First the particle of type $\ta$ sitting a site $t_{b}$
is removed. To achieve this, first remove a particle of type $\ta$ at the north-side of 
the rectangle and then (use Lemma~\ref{lemma on motion of particles inside a cluster}) 
move the vacancy to site $t_{b}$. The configuration that is reached, which we call 
$\eta\prm$, is such that $H(\eta\prm) \le H(\eta) + 3U - \Da$. Next, move the particle 
of type $\ta$ at $t_{e}$ to site $t_{b}$ ($\D H=0$), to reach a configuration $\hat{\eta}$ 
whose energy is $H(\hat{\eta}) = H(\eta)-\Da$. Finally, argue as in Case~(\ref{rectangle
expansion - unit hole}), to arrive at $H(\hat{\eta}) \leq H(\eta) + 3U- \Da$.
	
\end{enumerate}

\noindent
Finally, note that (\ref{subpropmetreg}) implies $\max\{6U+\Db,10U-\Da,12U-\Db\} = 10U - \Da$.

By Lemma~\ref{lemma btiles are optimal tiles}, $H(\eta\prm) \le  H(\eta)$, and
therefore the same argument can be used to show that all the right-sections of the 
support can be $\btiled$ within the same energy barrier. The left-section can be 
$\btiled$ analogously.
	
To conclude, it remains to be shown how particles of type $\tb$, possibly adjacent to 
one side of the rectangle, can be removed from $\Lambda$. Call $t$ the tile associated 
with the particle $p$ of type $\tb$ that has to be removed ($p$ sits at site $t_{e}$)
and $v$ the tile adjacent to $t$ belonging to the rectangle. First bring a vacancy to 
site $v_{e}$ within energy barrier $10U - \Db$ (one way to achieve this has been described 
in Case~(\ref{rectangle expansion - two becomes one}) above) and then move $p$ to 
site $v_{e}$ (see Lemma~\ref{lemma on motion of particles inside a cluster}).
\end{proof}


\subsubsection{Merging adjacent $\btiled$ rectangles}
\label{sec sliding adjacent btiled rectangles}

\begin{definition}
A $\abbar$ $b_{1}$ of length $\ell$ of a cluster $c_{1}$ is said to be adjacent to a 
$\abbar$ $b_{2}$ of length $m \le \ell$ of a cluster $c_{2}$ if there exist $m$ mutually 
disjoint pairs $(q_{1}^{i}, q_{2}^{i})$ of particles of type $\ta$ with $q_{1}^{i} 
\in b_{1}$ and $q_{2}^{i} \in b_{2}$ such that $u(q_{1}^{i}) - u(q_{2}^{i})=v$ with 
$\|v\| = \frac12\sqrt{2}$ for $i=1,\ldots,m$. The vector $v$ is called the offset of 
$b_{2}$ with respect to $b_{1}$. The tiles in $b_{1}$ have a different parity than the 
tiles in $b_{2}$. The particles $q_{1}^{i} \in b_{1}$, $i = 1,\ldots m$, are called 
the external particles of $b_{1}$ with respect to $b_{2}$, and the particles $q_{2}^{i} 
\in b_{2}$, $i=1,\ldots,m$, are called the external particles of $b_{2}$ with respect 
to $b_{1}$.
\end{definition}

\begin{proposition}
\label{prop5}
Let $\eta$ be a configuration that contains two adjacent $\btiled$ rectangles. Then 
the configuration $\eta\prm$ obtained by ``merging'' these two rectangles satisfies 
$H(\eta\prm)=H(\eta)$ and $\comlev(\eta,\eta\prm) \leq H(\eta)+2U-\Da$.
\end{proposition}

\begin{proof}
Given two adjacent bars $b_{1}$ and $b_{2}$ with offset $v=(v_{1},v_{2})$ in a 
configuration $\eta$, we want to define the sliding of $b_{2}$ onto $b_{1}$ along 
$v$. The resulting configuration $\eta\prm$ is such that all the particles of type 
$\tb$ originally in $b_{2}$ are slid by $(v_{1},v_{2})$ with respect to their 
position in $\eta$, and all the external particles of type $\ta$ of $b_{2}$ with 
respect to $b_{1}$ are slid by $(v_{1},- v_{2})$ when the two bars are horizontal 
and by$(-v_{1},v_{2})$ when the two bars are vertical. Via the sliding, the $m$ 
$\btiles$ in $b_{2}$ are turned into $m$ $\btiles$ with the same parity as the tiles 
in $b_{1}$. It is easy to see that $H(\eta\prm) = H(\eta)$, since neither the total 
number of active bonds of the configuration nor the number of particles of each type is 
changed. 

To describe the sliding of a bar onto another bar along a vector $v$, we may assume w.l.o.g.\ 
that the two bars are vertical and that the vector $v$ is equal to $(- \tfrac12,-\tfrac12)$
(Fig.~\ref{adjab0}).
Start by moving the lower-most external particle of type $\ta$ in $b_{2}$ over the vector 
$v\prm = (\tfrac12, -\tfrac12)$ (Fig.~\ref{adjab1}). This leads to an increase by $U$ in energy. Then move 
the lower-most particle of type $\tb$ over the vector $v$ (Fig.~\ref{adjab2}). Since the number of deactivated 
bonds is equal to the number of new bonds activated, this move does not change the energy. Proceed 
by moving over the vector $v\prm$ the second particle of type $\ta$ from the bottom of the bar
(Fig.~\ref{adjab3}). 
This also is a move that does not change the energy. Afterwards, the second particle of type 
$\tb$ from the top is moved over the vector $v$ (Fig.~\ref{adjab4}). This sequence of moves proceeds iteratively 
(without a change in energy) until the $m$-th particle of type $\tb$ has been moved over the 
vector $v$. Finally, the $(m+1)$-st external particle of type $\ta$ is moved over the vector 
$v\prm$ (Fig.~\ref{adjabf}). 
This move decreases the energy by $U$. Thus, $U$ is the energy barrier that must
be overcome in order to realize the sliding of a $\abbar$ onto another $\abbar$ over the 
vector $v$. 

\begin{figure}[htbp]
\centering
\subfigure[$\eta_{0}$ \label{adjab0}]
{\includegraphics[width=0.2\textwidth]{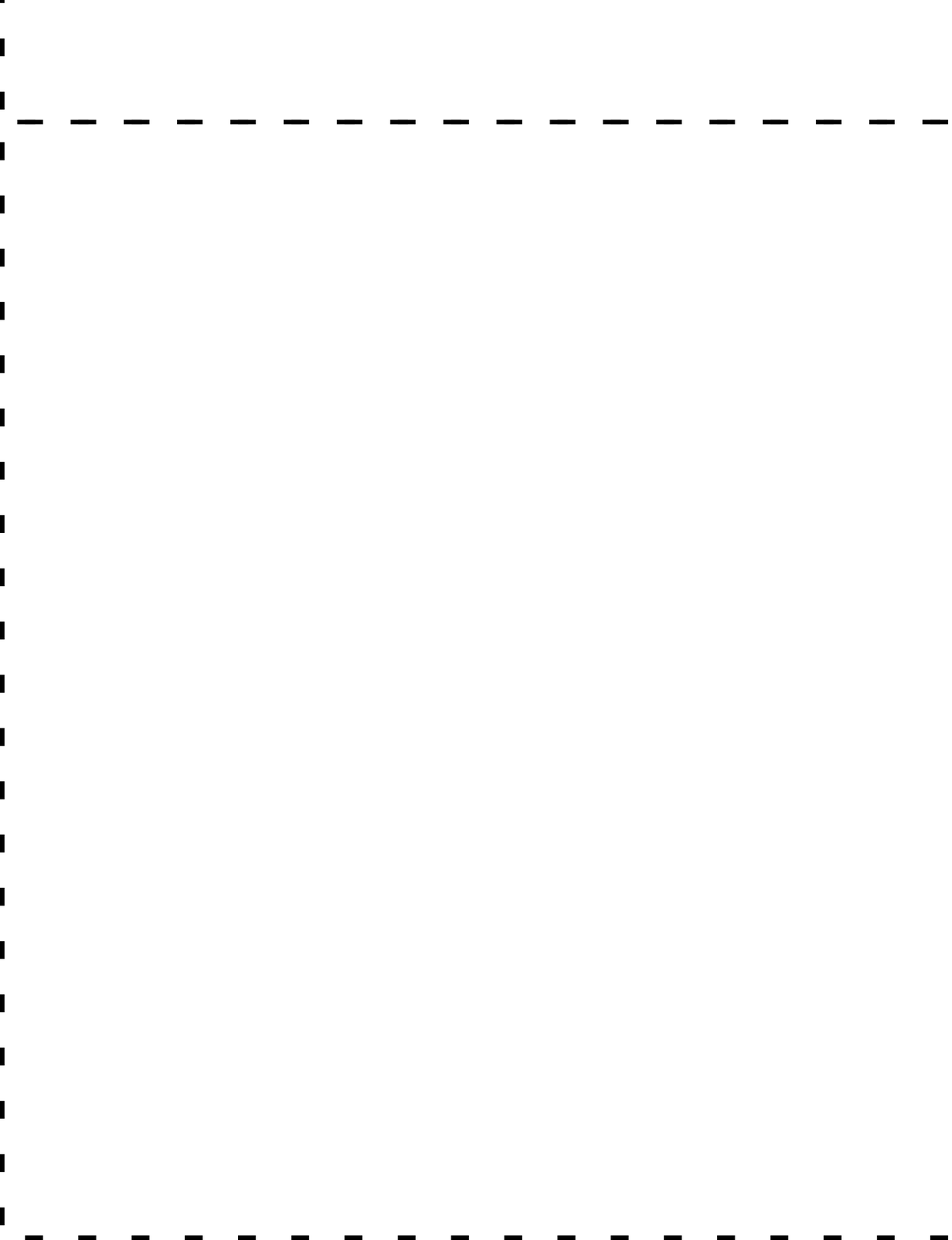}}
\qquad
\subfigure[\label{adjab1}]
{\includegraphics[width=0.2\textwidth]{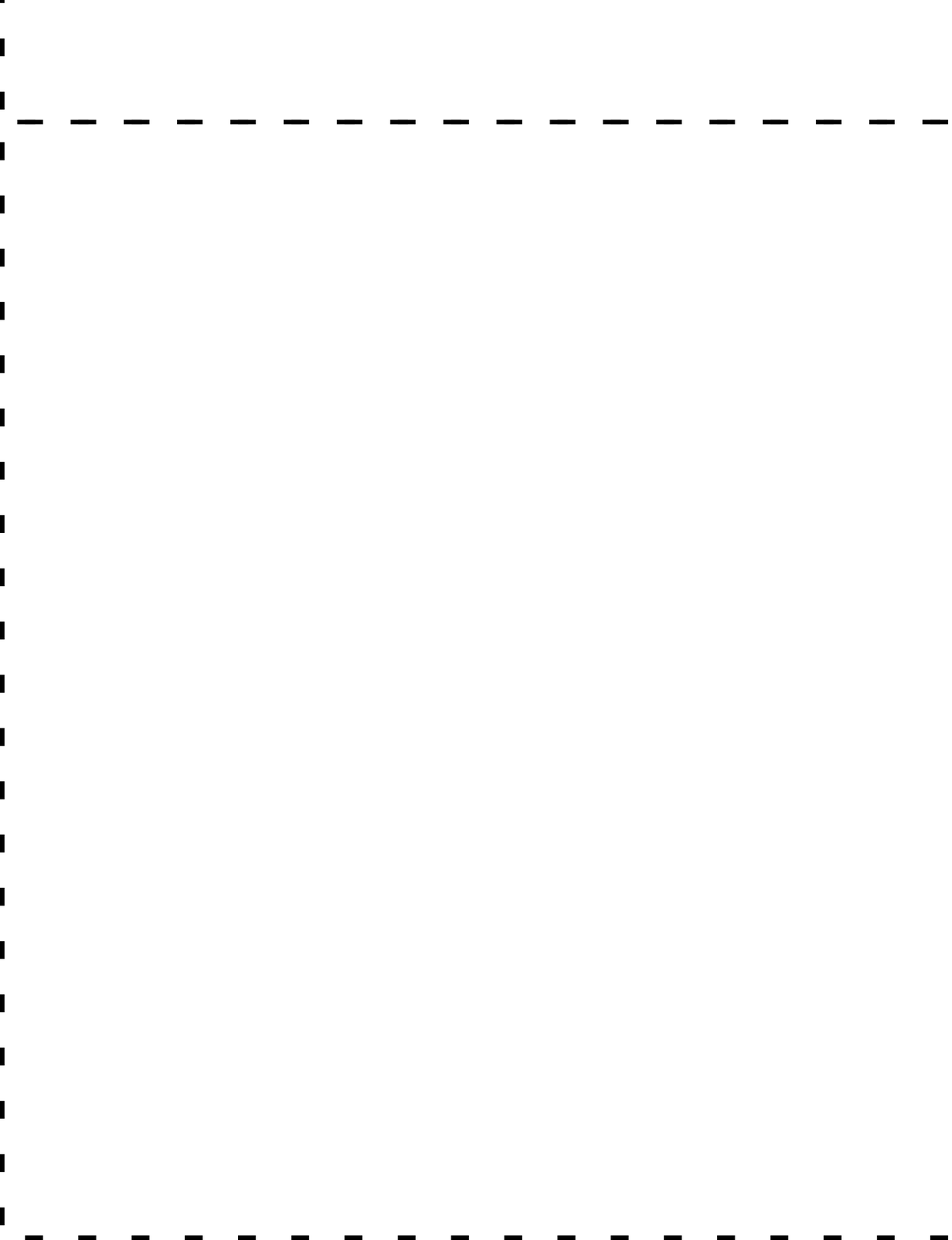}}
\qquad
\subfigure[\label{adjab2}]
{\includegraphics[width=0.2\textwidth]{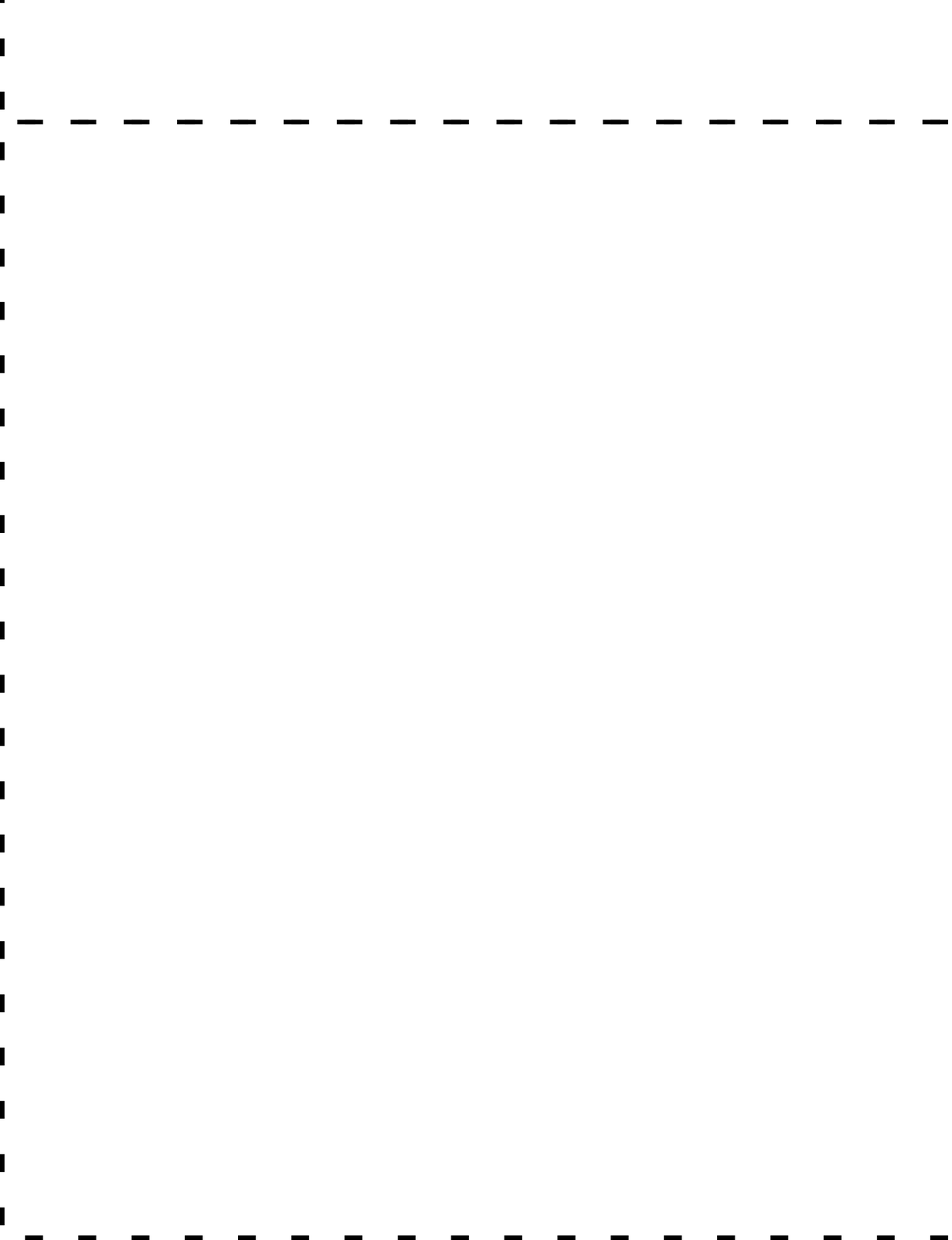}}
\qquad\qquad
\subfigure[\label{adjab3}]
{\includegraphics[width=0.2\textwidth]{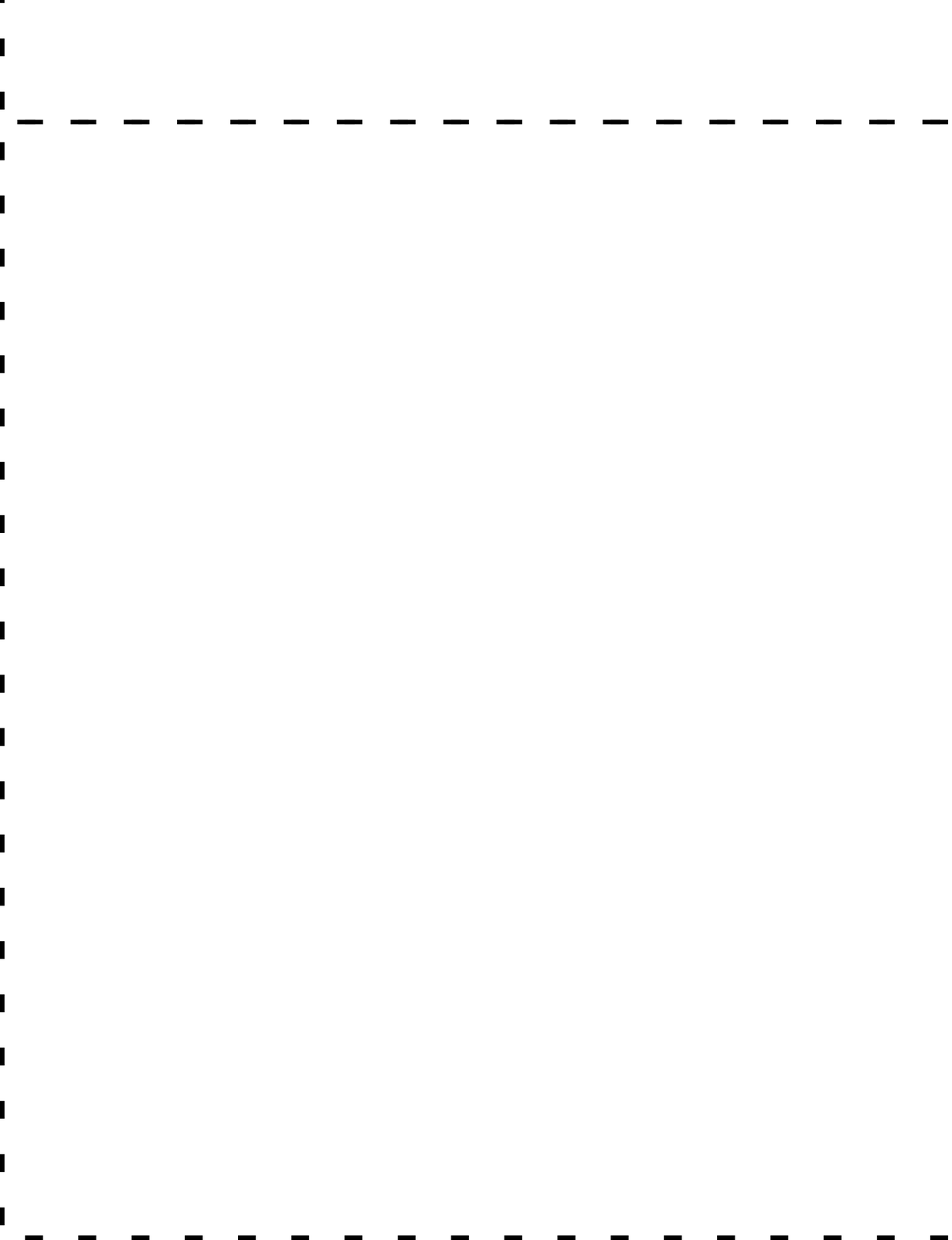}}
\qquad
\subfigure[\label{adjab4}]
{\includegraphics[width=0.2\textwidth]{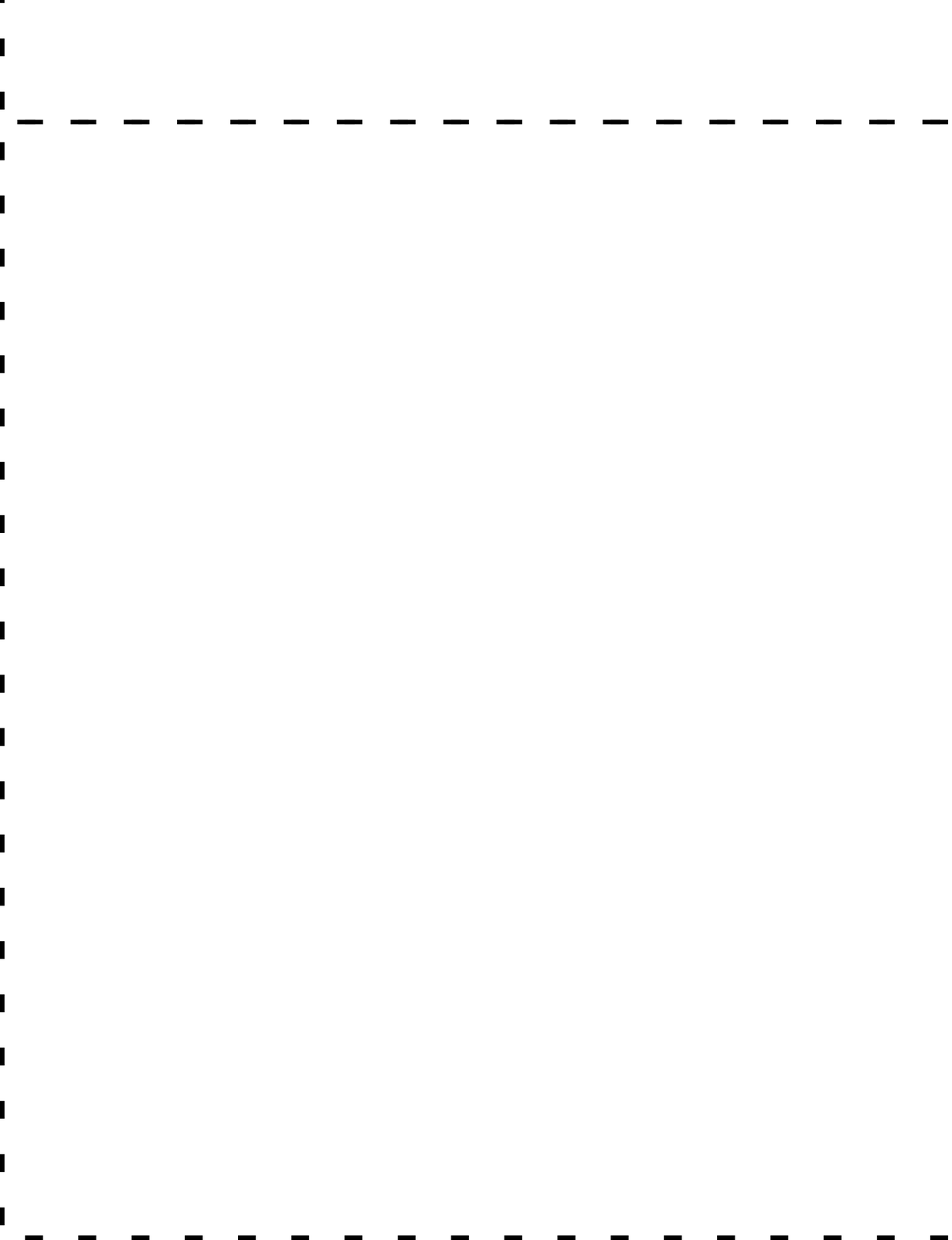}}
\qquad
\subfigure[$\eta_{1}$ \label{adjabf}]
{\includegraphics[width=0.2\textwidth]{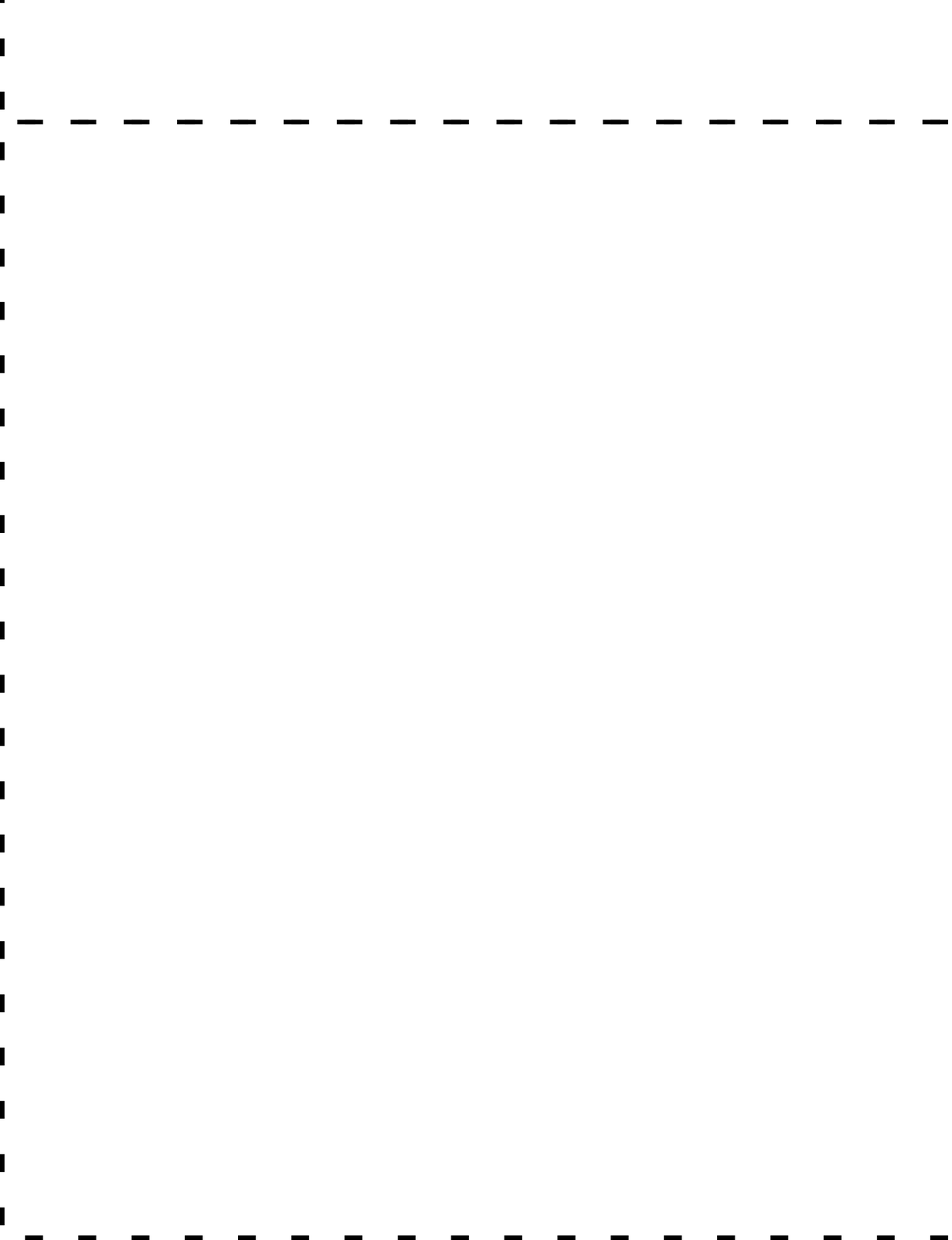}}
\caption{The sliding of $b_{2}$ onto $b_{1}$.}
\label{fig:adjacent abbars}
\end{figure}

It is clear that, given a configuration $\eta$ containing two $\btiled$ rectangles 
$c_{1}$ (with vertical side length $\ell$) and $c_{2}$ (with vertical side length
$m \le \ell$) with offset $v$, it is possible to reduce $\eta$ to a configuration 
$\eta\prm$ such that $c_{1}$ and $c_{2}$ are merged into of a single cluster by 
sliding one bar after another, without exceeding energy barrier $\Delta H = U$, 
provided the other clusters of $\eta$ do not interfere with this procedure. Sliding 
the last bar of $c_{2}$ we get an excess of free particles of type $\ta$, which can 
be removed from $\Lambda$, lowering the energy. In particular, the configuration 
$\eta\prm$ obtained via the sliding of $c_{2}$ onto $c_{1}$ along $v$ without
exceeding energy level $H(\eta) + U$ has energy $H(\eta\prm) = H(\eta)-(m+1)\Da$, 
since the two configurations consist of the same number of $\btiles$, and $\eta\prm$ 
contains $m + 1$ particles of type $\ta$ less than $\eta$. Moreover, $\comlev(\eta,
\eta\prm) = H(\eta) + U$.

In the argument above, the first move consisted of moving down-right a particle of 
type $\ta$ of $b_{2}$ to an empty site (say, site $i$). If in configuration $\eta$ 
site $i$ is occupied by a particle of type $\ta$, then the sliding of the vertical
$\abbar$ can be realized by modifying the procedure as follows. First remove from the 
box the top-left particle of type $\ta$ of $b_{2}$ sitting at site $j$ to reach a 
configuration with energy $H(\eta) + U - \Da$ (which can be done without exceeding 
energy level $H(\eta)+U$). Then move to $j$ the particle of type $\ta$ sitting at site 
$k=j+v=j+(-\frac{1}{2},-\frac{1}{2})$ in $\eta$, which increases the energy up 
to level $H(\eta) + 2U - \Da$. Then site $k$ is filled with the particle of type 
$\ta$ originally at site $k + (\frac{1}{2}, -\frac{1}{2})$ without an increase in 
energy. It is possible to continue in this way until the configuration obtained 
after the first step of the above case is reached. This configuration has energy 
$H(\eta) + U - \Da$. Then proceed as in the above case until $b_{2}$ is slid onto 
$b_{1}$. This leads to a configuration with energy $H(\eta)- \Da < H(\eta)$. In 
order to perform the (modified) sliding procedure, it is sufficient to assume that 
the north-side of rectangle $c_{2}$ is lattice-connecting.
\end{proof}


\subsubsection{Removing subcritical clusters}
\label{sec preparation procedure}

The \emph{cleaning mechanism} defined in this section produces a configuration
for which we have a certain control on the geometry of the constituent clusters.
In particular, these clusters will be suitable for the application of the previous five
energy reduction mechanisms.
We begin by looking at pending dimers (see Fig.~\ref{fig:pending dimer}).

\begin{figure}[htbp]
\centering
{\includegraphics[width=0.2\textwidth]{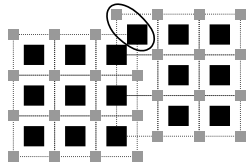}}
\caption{A pending dimer is the pair of particles circled in the picture.}
\label{fig:pending dimer}
\end{figure}

\begin{definition}
A pending dimer consists of two adjacent particles of different type such that 
the particle of type $\ta$ is lattice-connecting and has only one active bond
and the particle of type $\tb$ has at most three active bonds.
\end{definition}

\begin{proposition}
\label{prop6}
Let $\eta$ be a configuration containing pending dimers. Then there exists a 
configuration $\eta\prm$ not containing pending dimers that satisfies $H(\eta\prm)< H(\eta)$ 
and $\comlev(\eta, \eta\prm) \le H(\eta) + 3U + \Db$. 
\end{proposition}

\begin{proof}
If the particle of type $\tb$ has at most two active bonds, then simply remove the pending dimer.
This reduces the energy, since two bonds are deactivated and a particle of each type is removed 
from $\Lambda$ ($\D H \le 2U - \Da - \Db < 0$), and can be achieved within an energy barrier 
$2U - \Da$ along the following path: first detach ($\D H = U$) and remove ($\D H = - \Da$)
the particle of type $\ta$, then detach ($\D H \le U$)  and remove ($\D H = - \Db$) the particle 
of type $\tb$.

If the particle of type $\tb$ has three active bonds we have two cases:
\begin{itemize}
\item[(i)]
The fourth neighbor of the particle of type $\tb$ of the pending dimer is empty. In this case 
$\eta\prm$ is obtained by filling this empty site with a particle of type $\ta$ in order to 
obtain a $\btile$, which lowers the energy since $\Da < U$. To do this, temporarily remove the 
pending dimer as described above. This leads to a configuration $\tilde{\eta}$ with energy 
$H(\tilde{\eta})=H(\eta) + 3U - \Da - \Db$ reached within energy barrier $3U - \Da$. Then 
bring a particle of type $\ta$ to the designated site ($\D H \le \Da$) and finally put back 
the dimer. The whole path is realized within energy barrier $3U + \Db$.
\item[(ii)]
The fourth neighbor of the particle of type $\tb$ is occupied by a particle of type $\tb$.
In this case $\eta\prm$ is the configuration such that the dimer is removed and the site 
originally occupied by the particle of type $\tb$ of the dimer is occupied by a particle of 
type $\ta$. To obtain $\eta\prm$ from $\eta$, remove the pending dimer (again, as above, 
within energy barrier $3U - \Da$), to reach a configuration $\tilde{\eta}$ with energy 
$H(\tilde{\eta}=H(\eta) + 3U - \Da - \Db$, and bring a particle of type $\ta$ within energy 
barrier $\Da$. To conclude, observe that $H(\eta\prm) = H(\eta) + 2U - \Db < H(\eta)$.   
\end{itemize}
\end{proof}

The cleaning mechanism works as follows: 
\begin{enumerate}
\item 
Remove all the lattice-connecting free particles from the configuration.
\label{alg preparation procedure removing free particles}
\suspend{enumerate}
After that repeat cyclically the following two steps:
\resume{enumerate}
\item 
Iteratively remove/transform all the lattice-connecting pending dimers.
\label{alg preparation step removing pending dimers}
\item 
Bring a particle of type $\ta$ to any of the free sites adjacent to 
the lattice-connecting particles of type $\tb$.
\label{alg preparation step bringing particles of type ta}
\end{enumerate}
Repeat the cleaning mechanism until the configuration is not affected anymore. 
Each of the three steps can be performed within energy barrier $3 U + \Db$. 
Moreover, each step reduces the energy.

\begin{lemma}
\label{lemma preparation procedure}
The outcome of the cleaning mechanism is either a configuration such that 
the first particle encountered while scanning $\Lambda$ in the lexicographic 
order is a particle of type $\ta$ belonging to a horizontal stable (south-)bridge,
or the configuration $\Box$.
\end{lemma}

\begin{proof}
Call $q$ the first particle of $\Lambda$ in the lexicographic order. Recall 
that the dual coordinates of $q$ are denoted by $u(q)=(u_{1}(q),u_{2}(q))$. 
Step~\ref{alg preparation step bringing particles of type ta} of the 
cleaning mechanism guarantees that $q$ is a particle of type $\ta$.
The fact that $q$ is the first particle in the lexicographic order implies 
that: (i) all the sites above $u(q)$ are empty; (ii) all the sites with the 
same vertical coordinate as $q$ lying on the left of $q$ are empty as well.
As a consequence of (ii), all the sites on the left of $q$ with vertical 
coordinate $u_{2}(q) - \frac{1}{2}$ are lattice-connecting and therefore 
cannot be occupied by a particle of type $\tb$. Since $q$ cannot be a free 
particle, the site with coordinates $(u_{1}(q)+\frac{1}{2},u_{2}(q)-\frac{1}{2})$
must be occupied by a particle $p$ of type $\tb$. Let $s(p)$ be the 
longest sequence of tiles adjacent to $t(p)$ such that the central site 
is occupied by a particle of type $\tb$. Obviously, $p$ is the left-most 
particle of type $\tb$ in $s(p)$. Call $\tilde{p}$ the last particle of 
type $\tb$ in $s(p)$ and $\tilde{q}$ the particle of type $\ta$ with 
coordinates $(u_{1}(p)+\frac{1}{2},u_{2}(p)+\frac{1}{2})$. (Note that 
$p$ and $\tilde{p}$ may coincide.)  All the sites on the north-side of 
$s(p)$ are lattice-connecting and hence are occupied by a particle 
of type $\ta$. To conclude, observe that both $p$ and $\tilde{p}$ must 
be saturated, otherwise at least one of the pairs $(q,p)$ and $(\tilde{q},
\tilde{p})$ constitutes a pending dimer.	
\end{proof}


\subsection{Energy reduction of a general configuration: Proof of 
Theorem~\ref{theorem recurrence}}
\label{sec general reduction}

Fix any $\eta\notin\{\Box,\boxplus\}$. In this section we will give a general procedure, 
called \emph{energy reduction algorithm}, that allows us to construct a path $\omega\colon\,
\eta\to\eta_{r}$ with $\eta_{r} \in \{\Box,\boxplus\}$ such that $\max_{\xi\in\omega} H(\xi) 
\le H(\eta) + V\starred$ with $V\starred \le 10U-\Da$ and $H(\eta_{r}) < H(\eta)$. Note that 
if $\eta_{r}=\boxplus$, then $H(\eta_{r})< H(\eta)$ because $\groundset=\boxplus$. The construction 
uses the six energy reduction mechanisms described in 
Sections~\ref{sec motion of particles inside clusters}--\ref{sec preparation procedure} and 
relies on Propositions~\ref{prop1}, \ref{prop2}, \ref{prop3}, \ref{prop4}, \ref{prop5}, \ref{prop6},
which are the key results of these sections. The maximal energy barrier in these propositions is 
$10U - \Da$.
\emph{Note}: The energy reduction mechanisms in Sections~\ref{sec growth of clusters} and 
\ref{sec main procedures} concern single droplets far away from $\boundary^{-}\Lambda$
and have an energy barrier not exceeding $4U + \Da < \Gamma\starred$ (see below \eqref{epsdef}). 
For such configurations, the energy can be essentially reduced by saturating particles of type 
$\tb$ and by adding and removing $\abbars$. This explains the remark made in 
Section~\ref{sec discussion}, item 4.

In the remainder of this section we call \emph{supercritical} a $\abbar$ of length $\ge\ell\starred$. 
Similarly, we call \emph{supercritical} a dual rectangle with both side lengths $\ge \ell\starred$.

\begin{proof}
As a preliminary step, perform the cleaning mechanism. If the outcome 
is $\Box$, then the claim is proven. Otherwise, let $b_{1}$ be the first 
bridge encountered in the lexicographic order (which exists by 
Lemma~\ref{lemma preparation procedure}). This bridge can be turned 
into an $\abbar$ $\bar{b}_{1}$ (see Section~\ref{sec main procedures}). If the length 
of $b_{1}$ is $<\ell\starred$, then the $\abbar$ $\bar{b}_{1}$ can be removed, 
which lowers the energy (see Section~\ref{sec growth of clusters}). In this 
case, go back to performing the cleaning mechanism. W.l.o.g.\ we may therefore 
assume that the length of ${b_{1}}$ is $>\ell\starred$.

By construction, all sites above $\bar{b}_{1}$ are empty, and therefore it is possible 
first to construct the $\btiled$ rectangle $r_{1} = \nrec{\bar{b}_{1}}$ within energy 
barrier $2\Da + 2\Db - 4U$ (again lowering the energy), and then expand $r_{1}$ to the 
rectangle $R_{1} = \nmexpansion{r_{1}}$ (see Section~\ref{maximal expansion btiled rectangle}). 
If the vertical side length of $R_{1}$ is $<\ell\starred$, then $R_{1}$ can be 
removed (lowering the energy), and it is possible to perform again the cleaning 
mechanism.

Therefore suppose that $R_{1}$ has both its side lengths $\geq\ell\starred$. In 
the remainder of the section we will show how to reach within energy barrier 
$10U - \Da$ a configuration containing a rectangle $R_{NW}$ touching both the 
north-side and the west-side of $\Lambdaminus$ whose support contains the support 
of $R_{1}$. Once this has been achieved, it is possible to argue for $R_{NW}$ in 
the same way as for $R_{1}$ in order to reach a configuration containing a rectangle 
$R_{NWE}$ touching the north-side, the east-side and the west-side of $\Lambdaminus$ 
whose support contains the support of $R_{NW}$. Repeating the same argument for 
$R_{NWE}$, it is possible to reach $\boxplus$.

The construction of $R_{NW}$ is obtained by using an algorithm called \emph{invasion} 
of $R_{1}$, which is constructed with the help of techniques similar to the ones that
were used to build $R_{1}$.

\medskip\noindent
{\bf (A) Invasion of $R_{1}$.} See Fig.~\ref{fig:algreduction}.
Let $(a_{1}, b_{1})$ be, respectively, the horizontal and the vertical coordinate of
the left lower-most particle of $R_{1}$ (which is of type $\ta$). Define $\Lambda(R_{1})
\subset\Lambda$ to be the set consisting of the sites whose vertical coordinate is 
$\geq b_{1}$ and horizontal coordinate is $<a_{1}$. In words, $\Lambda(R_{1})$ 
contains the sites of $\Lambda$ on the left of $R_{1}$. Perform the cleaning mechanism 
(see Section~\ref{sec preparation procedure}) and scan $\Lambda(R_{1})$ in the lexicographic 
order. Three cases are possible.

\begin{enumerate}[1.]

\item	
$\Lambda(R_{1})$ is empty. Add, if possible ($R_{1}$ might already be touching the 
west-boundary of $\Lambdaminus$), $\abbars$ onto the left side of $R_{1}$ until the 
resulting cluster touches the west-boundary of $\Lambdaminus$.
\label{enum Lambda(R1) is empty}

\item 
The first horizontal bridge $b_{2}$ encountered in $\Lambda(R_{1})$ has length 
$<\ell\starred$. Remove the particles of the (south)-support of the bridge, lowering 
the energy of the configuration, and restart the covering of $\Lambda(R_{1})$.
\label{enum Lambda(R1) contains subcritical bridge}

\item 
The first horizontal bridge $b_{2}$ encountered in $\Lambda(R_{1})$ has length
$\geq\ell\starred$. As for $b_{1}$, first turn $b_{2}$ into the $\abbar$ $\bar{b}_{2}$, 
then build the $\btiled$ rectangle $r_{2} = \nrec{\bar{b}_{2}}$, after that expand 
$r_{2}$ to $R_{2} = \nmexpansion{r_{2}}$, and finally perform the cleaning mechanism. Note 
that the support of $R_{2}$ may cover (part or possibly all of) the support of $R_{1}$. 
This means that during the maximal expansion, some of the sites of $\supp(R_{1})$ were 
in the support of the pillared beam that is going to be $\btiled$. Each time this 
happens, $R_{2}$ absorbs an entire vertical supercritical $\abbar$ of $R_{1}$ (see 
Section~\ref{maximal expansion btiled rectangle}). Call $\tilde{R_{1}}$ what is left 
of $R_{1}$ after the maximal expansion of $R_{2}$. The following three cases are possible:
\begin{inparaenum}[(i)]
\item 
$\tilde{R}_{1}$ does not contain any particle ($\tilde{R}_{1} = \emptyset$);
\label{enum R1 covered}
\item 
$\tilde{R}_{1} \subconf R_{1}$ (in the proper sense); 
\label{enum R1 reduced}
\item 
$\tilde{R}_{1} = R_{1}$. 
\label{enum R1 untouched}
\end{inparaenum}
In Case~(ii), 
the rectangles $R_{2}$ and $\tilde{R}_{1}$ are necessarily 
adjacent (more precisely, the right-most $\abbar$ of $R_{2}$ is adjacent to the left-most 
$\abbar$ of $R_{1}$), whereas in 
Case~(iii) 
the two rectangles may or 
may not be adjacent. Note that this implies that if $\tilde{R_{1}} \subconf R_{1}$, then 
$R_{2}$ is necessarily supercritical. Obviously, if $\tilde{R}_{1} \neq \emptyset$, then 
it is again a $\btiled$ rectangle, and there are several possibilities.
\begin{enumerate}[(a)]
\item 
$R_{2}$ is not supercritical. This implies that $\tilde{R_{1}} = R_{1}$. Remove 
$R_{2}$ from $\Lambda$, put $R_{1} = \tilde{R_{1}}$ and restart the invasion of $R_{1}$.
\item 
$R_{2}$ is supercritical and $\tilde{R}_{1} = \emptyset$. Change the name of $R_{2}$ to 
$R_{1}$ and restart the covering of $\Lambda(R_{1})$.
\label{enum R1 disappeared}
\item 
$R_{2}$ is supercritical and is adjacent to $\tilde{R}_{1}$. Note that both rectangles 
touch the north-side of $\Lambdaminus$. Call $R^{\max}$ the rectangle with the largest 
vertical length (in case of a tie, w.l.o.g.\ choose $R_{1}$) and call $R^{\min}$ the 
other rectangle. Slide $R^{\min}$ onto $R^{\max}$. This is possible because the smoothing 
phase of the maximal expansion (see Section~\ref{maximal expansion btiled rectangle}) removes 
all the particles of type $\tb$ that may interfere with the sliding of the $\abbars$. 
Then perform again the maximal expansion of $R^{\max}$, i.e., the rectangle that has 
not been moved during the sliding. These steps bring the configuration to a rectangle 
whose support contains $\supp(R_{2})\cup\supp(R_{1})\cup\Lambda(R_{1})$. Call this 
rectangle $R_{1}$ and restart the invasion of $R_{1}$. 
\label{enum R1 and R2 adjacent}
\item 
$R_{2}$ is supercritical and is not adjacent to $\tilde{R}_{1}$. This implies $\tilde{R}_{1}
= R_{1}$. Start the invasion of $R_{2}$ (see below).
\end{enumerate}
\end{enumerate}

In order to complete the proof, it remains to show how the invasion of $R_{2}$ carries over.
To that end, we introduce the following \emph{recursive algorithm} realizing the 
\emph{invasion of $R_{i}$} for $i=2,3,\ldots$, etc.

\medskip\noindent
{\bf (B) Invasion of $R_{i}$.}
Call $\bar{R}_{i-1}$ what is left of $R_{i-1}$ after the invasion of $R_{i+1}$. 
There are three cases:
\begin{enumerate}[I.]
\item 
$\bar{R}_{i-1} = \emptyset$ (i.e., the support of $R_{i-1}$ is completely covered by $R_{i}$). 
Put $R_{i-1} = R_{i}$ and restart the invasion of $R_{i-1}$.
\item 
$\bar{R}_{i-1} \neq \emptyset$ and $R_{i}$ and $\bar{R}_{i-1}$ are adjacent. 
Call $R^{\max}$ the rectangle with the largest vertical side between $R_{i}$ and 
$\bar{R}_{i-1}$ (in case of a tie, w.l.o.g.\ choose $R^{\max} = R_{i}$) and call $R^{\min}$ 
the other rectangle. Slide $R^{\min}$ onto $R^{\max}$ and perform the maximal expansion of 
$R^{\max}$. Call $R_{i-1}$ the outcome of the maximal expansion of $R^{\max}$ and restart the 
invasion of $R_{i-1}$.
\item 
$\bar{R}_{i-1} \neq \emptyset$ and $R_{i}$ and $\bar{R}_{i-1}$ are not adjacent.
If $R_{i}$ is on the left of $R_{i-1}$, then let $(a_{i},b_{i})$ denote, respectively, 
the horizontal and the vertical coordinate of the lower right-most particle (which is 
of type $\ta$) of $R_{i}$, and call $\Lambda(R_{i})$ the subset of $\Lambda(R_{i-1})$
consisting of those sites whose vertical coordinates are  $\ge b_{i}$ and whose horizontal 
coordinates are $> a_{i}$. If $R_{i}$ is on the right of $R_{i-1}$, then let $(a_{i},b_{i})$ 
denote, respectively, the horizontal and the vertical coordinate of the lower left-most 
particle (which is of type $\ta$) of $R_{i}$, and call $\Lambda(R_{i})$ the subset of 
$\Lambda(R_{i-1})$ consisting of those sites whose vertical coordinates are $\ge b_{i}$ 
and whose horizontal coordinates are $< a_{i}$. In words, $\Lambda(R_{i})$ consists of 
those sites of $\Lambda(R_{i-1})$ between $R_{i-1}$ and $R_{i}$. Perform the cleaning
mechanism and scan $\Lambda(R_{i})$ in the lexicographic order. There are again several cases.
	
\begin{enumerate}[1.]
\item 
$\Lambda(R_{i})$ is empty. Call $R^{\max}$ the rectangle with the largest vertical side 
between $R_{i}$ and $\bar{R}_{i-1}$ (in case of tie, w.l.o.g.\ choose $R^{\max} = R_{i}$)
and call $R^{\min}$ the other rectangle. Add vertical $\abbars$ on the side of $R^{\min}$ 
facing $R^{\max}$ until (depending on the parity of the rectangles) it becomes adjacent 
(different parity) to $R^{\max}$ or it is at distance $1$ (same parity) from $R^{\max}$. 
In the first case, slide the extended $R^{\min}$ onto $R^{\max}$. Perform the maximal 
expansion of $R^{\max}$, and call $R_{i -1}$ the rectangle obtained in this way, whose 
support contains $\supp(R_{i})\cup R_{i-1}\cup\Lambda(R_{i-1})$. Restart the invasion 
of $R_{i-1}$.
\label{enum Lambda(Ri) is empty}

\item 
The first horizontal bridge $b_{i+1}$ encountered in $\Lambda(R_{i})$ has length
$<\ell\starred$. Remove the particles of the (south)-support of the bridge, lowering 
the energy of the configuration, and restart the invasion of $R_{i}$.
\label{enum Lambda(Ri) subcritical bridge}
		
\item 
The first horizontal bridge $b_{i+1}$ encountered in $\Lambda(R_{i})$ has length 
$\geq\ell\starred$. First turn $b_{i+1}$ into the $\abbar$ $\bar{b}_{i+1}$, then build 
the $\btiled$ rectangle $r_{i+1} = \nrec{\bar{b}_{i+1}}$, after that expand $r_{i}$ to $R_{i+1} 
= \nmexpansion{r_{i+1}}$, and finally perform the cleaning mechanism. Call $\tilde{R_{i}}$ 
what is left of $R_{i}$ after the maximal expansion of $R_{i+1}$. The following cases 
are possible.
\begin{enumerate}[(a)]
\item 
$R_{i+1}$ is not supercritical. This implies $\tilde{R}_{i} = R_{i}$. Remove $R_{i+1}$
from $\Lambda$, put $R_{i} = \tilde{R_{i}}$, and restart the invasion of $R_{i}$.
\item 
$R_{i+1}$ is supercritical and $\tilde{R}_{i} = \emptyset$. Change the name of 
$R_{i+1}$ to $R_{i}$, and restart the invasion of $R_{i}$.
\label{enum Ri disappeared}
\item 
$R_{i+1}$ is supercritical and is adjacent to $\tilde{R}_{i}$. Note that both rectangles 
touch the north-side of $\Lambdaminus$. Slide the rectangle with the shorter vertical 
length onto the other rectangle and perform again the maximal expansion of the rectangle 
that has not been moved during the sliding. These steps bring the configuration to a rectangle 
whose support contains $\supp(R_{i+1}) \cup \supp(R_{i}) \cup \Lambda(R_{i})$. Call 
this rectangle $R_{i}$ and restart the invasion of $R_{i}$.
\label{enum Ri and Ri+1 adjacent}
\item 
$R_{i+1}$ is supercritical and is not adjacent to $\tilde{R}_{i}$. This implies 
$\tilde{R}_{i} = R_{i}$. Start the invasion of $R_{i+1}$.
\label{enum Lambda(Ri) contains supercritical bridge}
\end{enumerate}		
\end{enumerate}
\end{enumerate}

The finiteness of $\Lambda$ ensures that the algorithm eventually terminates.
\end{proof}

\begin{figure}[htbp]
\centering
\subfigure[\label{algred0}]
{\includegraphics[width=0.25\textwidth]{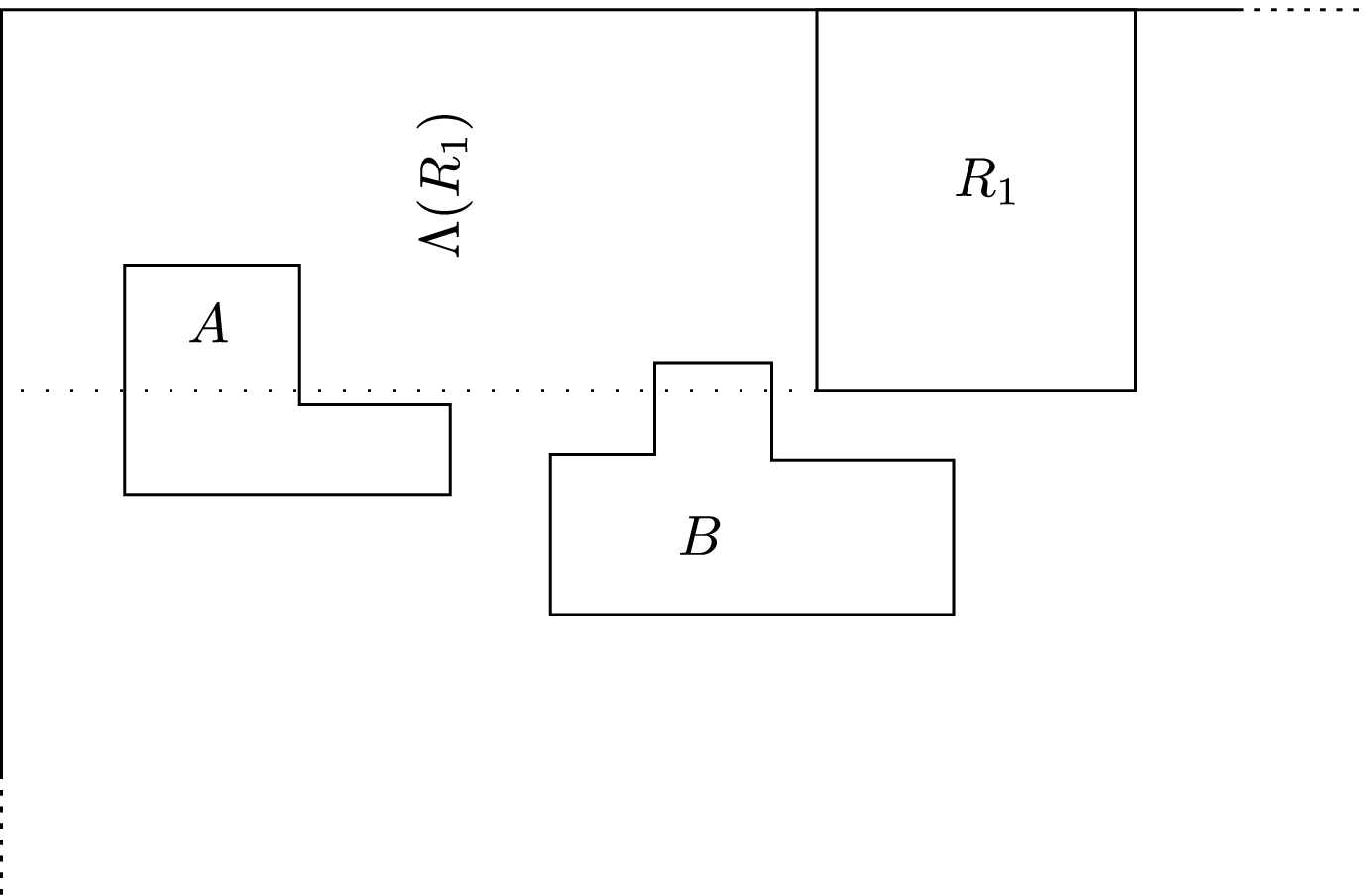}}
\qquad
\subfigure[\label{algred1}]
{\includegraphics[width=0.25\textwidth]{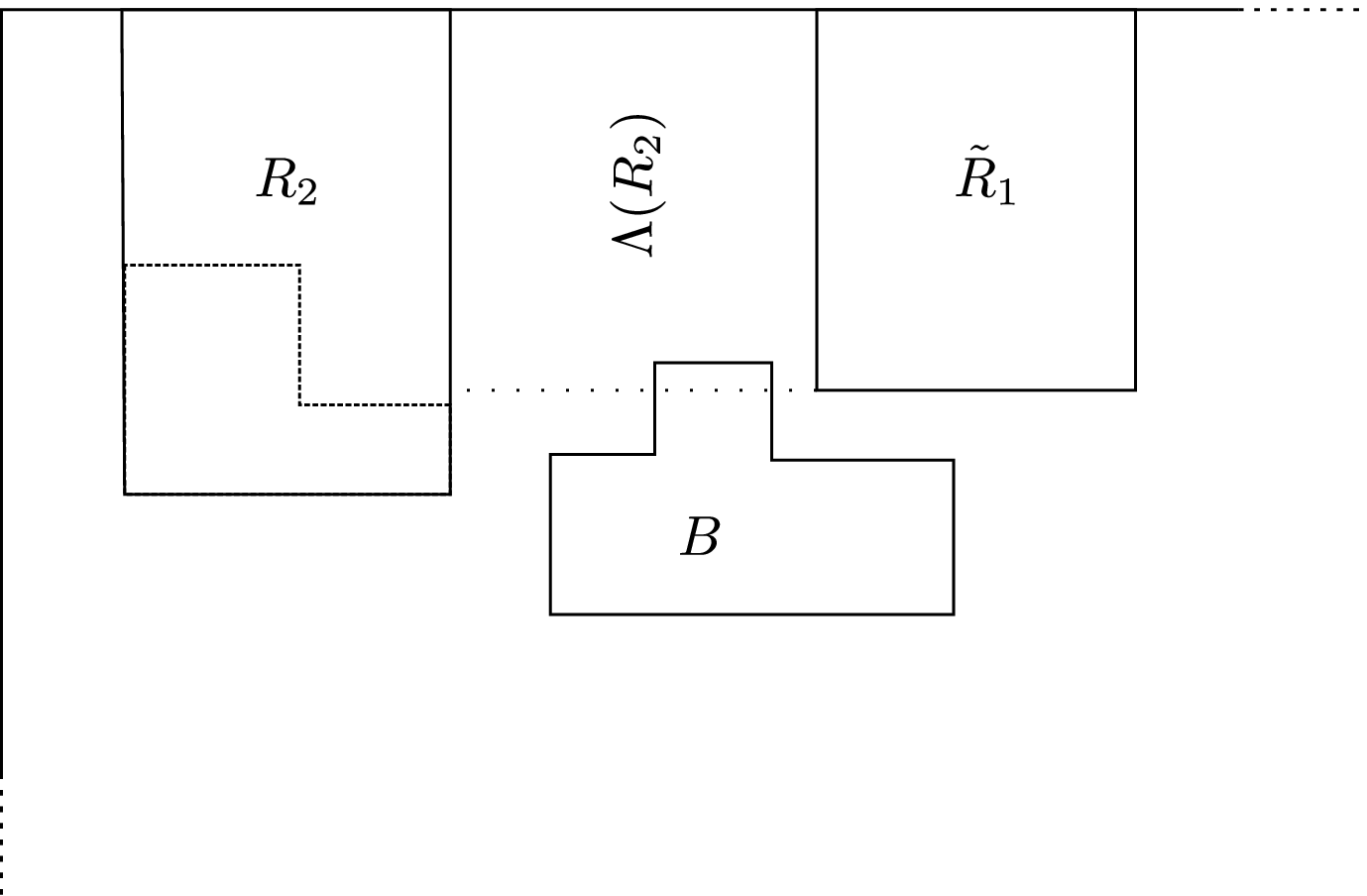}}
\qquad
\subfigure[\label{algred2}]
{\includegraphics[width=0.25\textwidth]{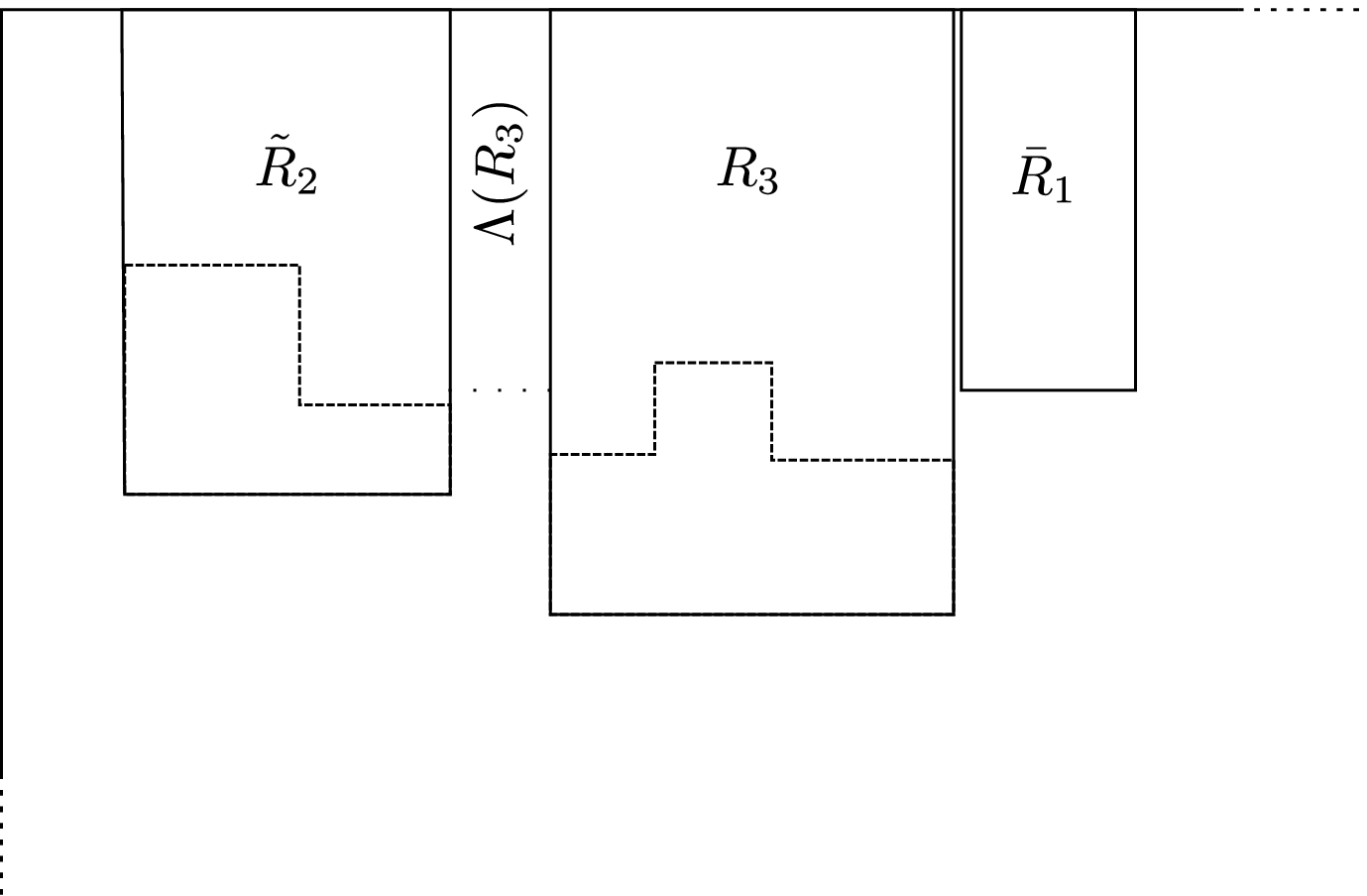}}
\\
\subfigure[\label{algred3}]
{\includegraphics[width=0.25\textwidth]{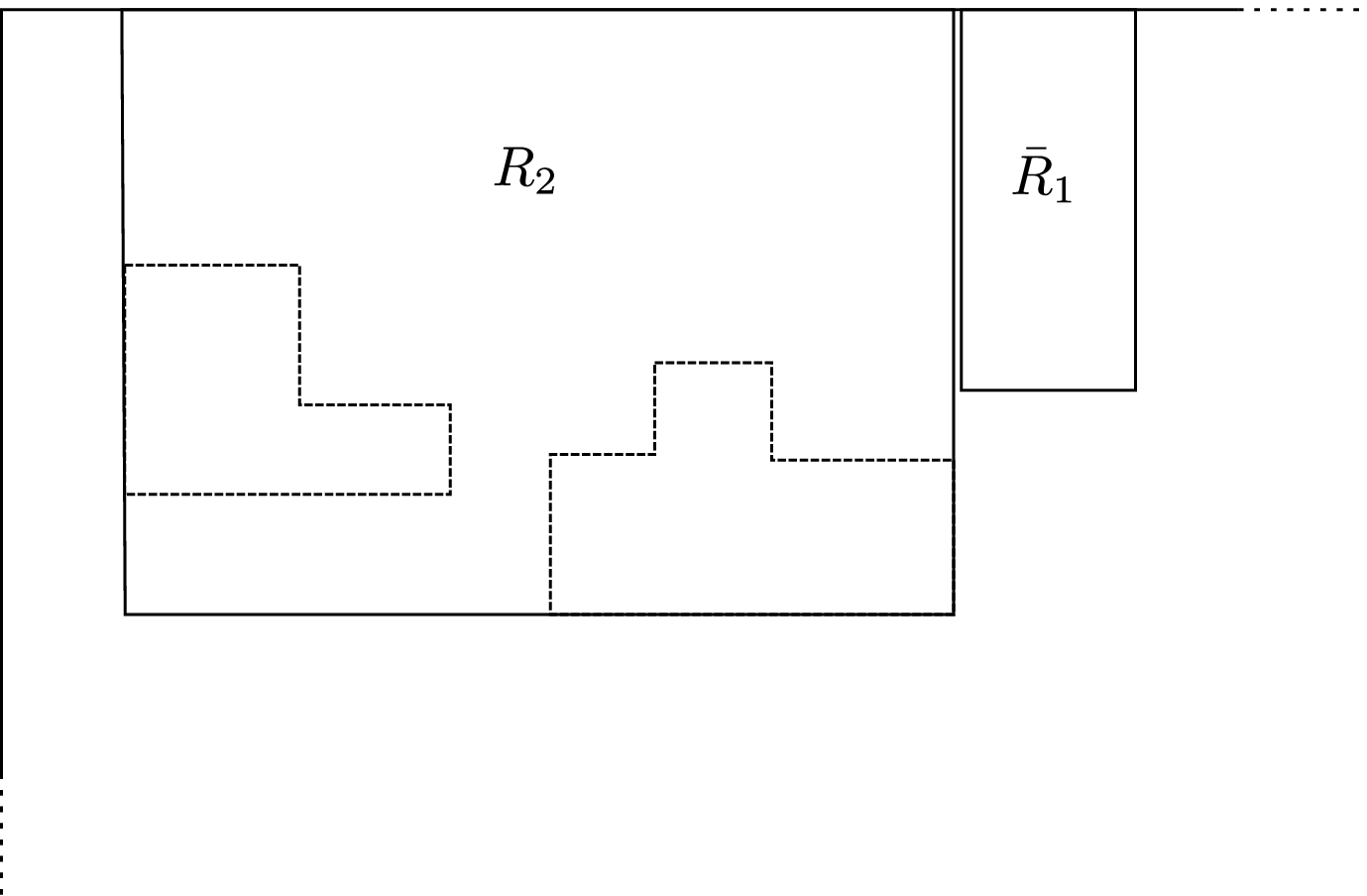}}
\qquad
\subfigure[\label{algred4}]
{\includegraphics[width=0.25\textwidth]{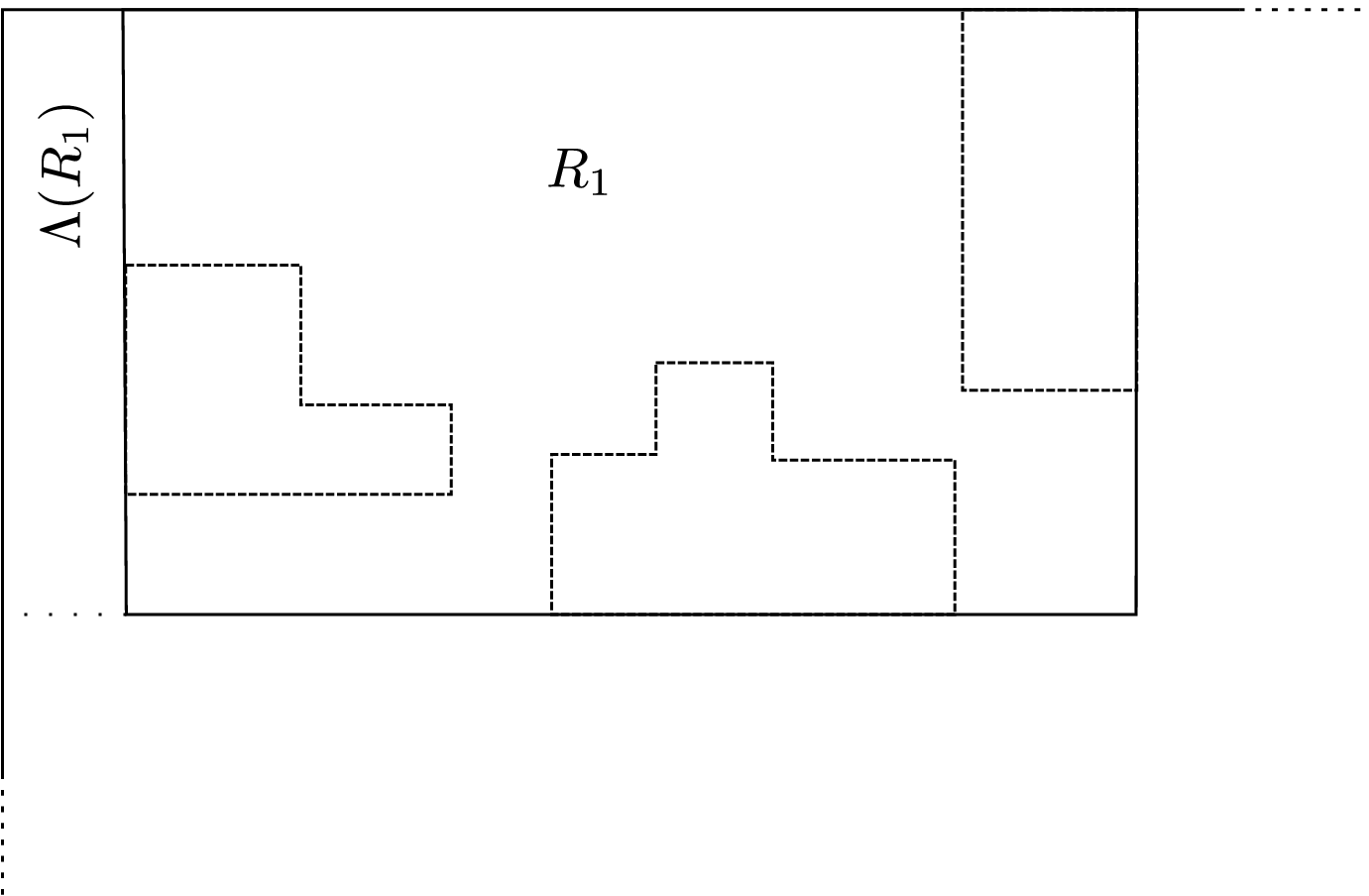}}
\qquad
\subfigure[\label{algred5}]
{\includegraphics[width=0.25\textwidth]{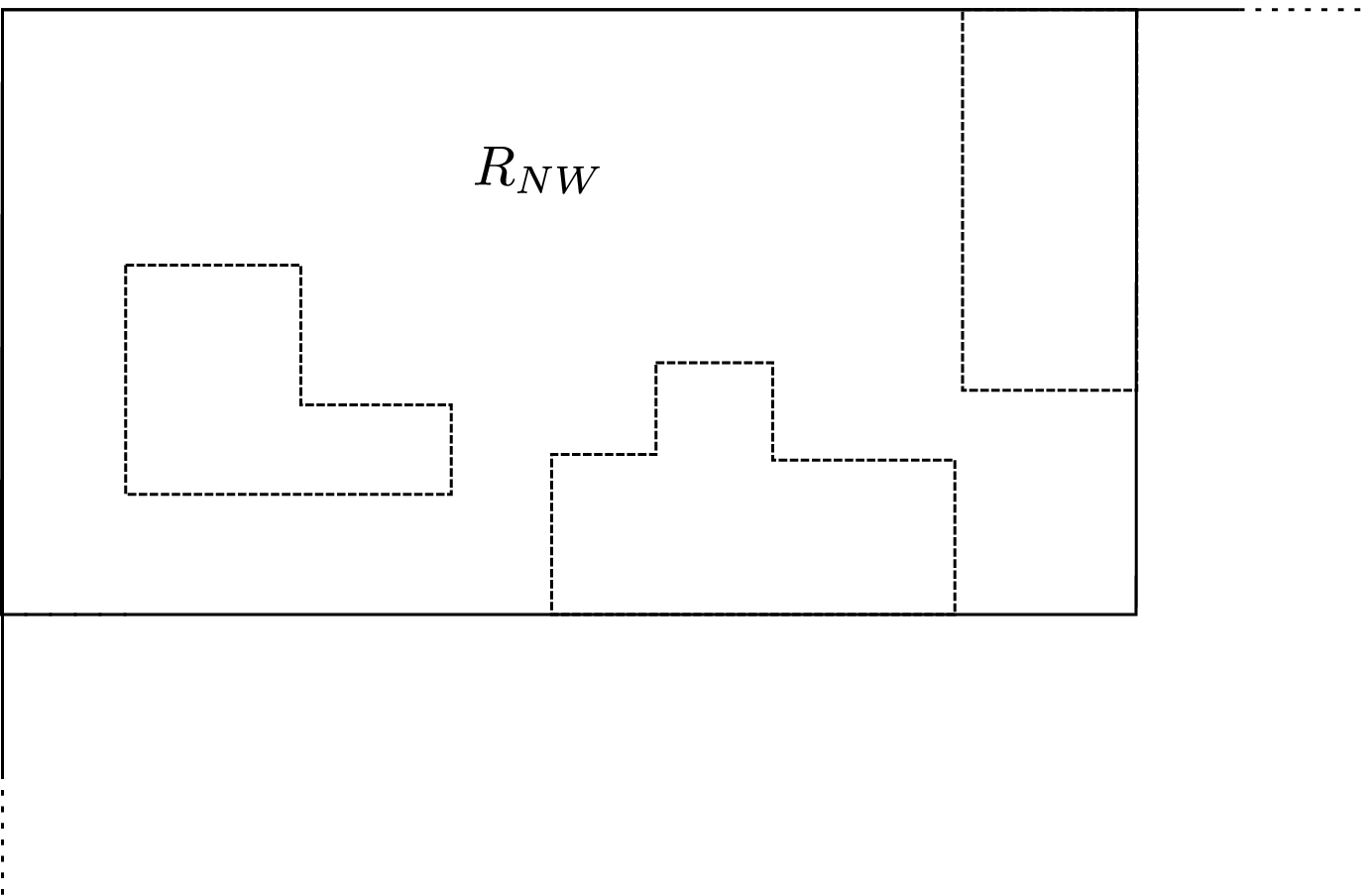}}
\caption{Example of invasion of the dual rectangle $R_{1}$.
Only the support of the relevant clusters are drawn and the parity
of different clusters is not indicated.
The set
$\Lambda(R_{1})$ contains a supercritical bridge belonging to cluster $A$ (Fig.~\ref{algred0}). 
Growing this bridge via the construction of its northern rectangle and its subsequent maximal expansion
leads to the supercritical rectangle $R_{2}$ (Fig.~\ref{algred1}). 
Next, the invasion of $\Lambda(R_{2})$ has  to be performed 
in order to complete the invasion of $R_{1}$. 
The set $\Lambda(R_{2})$ contains a supercritical bridge belonging
to cluster $B$, which is grown into the supercritical rectangle $R_{3}$ (Fig.~\ref{algred2}).
Note that $R_{3}$ partly covers the support of $\tilde{R}_{1}$ and that 
$R_{3}$ and $\bar{R}_{1}$ are adjacent.
The invasion of $R_{2}$ proceeds via the invasion of $R_{3}$. 
Since $\Lambda(R_{3})$ is empty, the invasion of $R_{3}$ is carried out 
by adding $\abbars$ to the left-side
of $R_{3}$ until $\tilde{R}_{2}$ is at dual distance $1$. 
After that a maximal expansion produces a dual
rectangle that covers the support of $\tilde{R}_{2}$ (Fig.~\ref{algred3}). 
The new dual rectangle $R_{2}$ is adjacent to
$\bar{R}_{1}$. The two rectangles are merged and a maximal expansion gives
a new rectangle $R_{1}$ (Fig.\ref{algred4}). 
Now $\Lambda(R_{1})$ is empty and can be filled by
adding $\abbars$ to the left-side of $R_{1}$ until the rectangle $R_{NW}$ is obtained
(Fig.~\ref{algred5}).
}
\label{fig:algreduction}
\end{figure}

\pagebreak


\begin{thebibliography}{99}

\bibitem{AC96}
L.\ Alonso and R.\ Cerf,
The three dimensional polyominoes of minimal area,
Electron.\ J.\ Combin.\ 3 (1996) Research Paper 27.

\bibitem{B09}
A.\ Bovier, Metastability, in: {\it Methods of Contemporary Mathematical
Statistical Physics} (ed.\ R.\ Koteck\'y), Lecture Notes in Mathematics 1970,
Springer, Berlin, 2009, pp.\ 177--221.

\bibitem{dHNT11}
F.\ den Hollander, F.R.\ Nardi and A.\ Troiani,
Metastability for Kawasaki dynamics at low temperature with two types
of particles, submitted to Electron.\ Comm.\ Probab, 	arXiv:1101.6069v1.

\bibitem{dHNTpr}
F.\ den Hollander, F.R.\ Nardi and A.\ Troiani,  
Kawasaki dynamics with two types of particles: critical droplets, manuscript in preparation. 

\bibitem{MNOS04}
F.\ Manzo, F.R.\ Nardi, E.\ Olivieri, and E.\ Scoppola, On the essential 
features of metastability: tunnelling time and critical configurations,
J.\ Stat.\ Phys.\ 115 (2004) 591--642.

\end{thebibliography}
\end{document}